\documentclass[a4paper,11pt,reqno]{amsart}

\usepackage{amsfonts}
\usepackage{amssymb}
\usepackage{amscd}
\usepackage{amsthm}
\usepackage{a4wide}
\usepackage{mathrsfs}
\usepackage[applemac]{inputenc}
\usepackage{amsmath}
\usepackage{amssymb}
\usepackage{amsthm}
\usepackage{textcomp}
\usepackage{graphicx}
\usepackage{enumerate}
\usepackage{mathrsfs}
\usepackage{frcursive}
\usepackage{tikz}
\usepackage[cyr]{aeguill}
\usepackage{xspace}
\usepackage{hyperref}
\usepackage{appendix}
\usepackage{esint}
\usepackage{cancel}
\usepackage{calc}

\newtheorem{defn}{Definition}[section]
\newtheorem{lemma}[defn]{Lemma}
\newtheorem{prop}[defn]{Proposition}
\newtheorem{theo}[defn]{Theorem}
\newtheorem{coro}[defn]{Corollary}

\newtheorem{claim}[defn]{Claim}
\newtheorem{rk}[defn]{Remark}

\def\Ric{\mathop{\rm Ric}\nolimits}

\def\Rm{\mathop{\rm Rm}\nolimits}
\def\tr{\mathop{\rm tr}\nolimits}

\def\div{\mathop{\rm div}\nolimits}

\def\Id{\mathop{\rm Id}\nolimits}

\def\Li{\mathop{\rm \mathscr{L}}\nolimits}

\def\Sym{\mathop{\rm Sym}\nolimits}

\def\supp{\mathop{\rm supp}\nolimits}

\def\R{\mathop{\rm R}\nolimits}

\def\W{\mathop{\mathcal{W}_{rel}}\nolimits}
\newcommand{\RR}{\mathbb{R}}

\newcommand{\br}{\mathbf{r}}
\newcommand{\bN}{\mathbf{N}}
\newcommand{\bX}{\mathbf{X}}

\newcommand{\calL}{\mathcal{L}}
\newcommand{\calB}{\mathcal{B}}
\newcommand{\calN}{\mathcal{N}}
\newcommand{\calR}{\mathcal{R}}

\allowdisplaybreaks
\numberwithin{equation}{section}

\def\warwickhome{@warwick.ac.uk}
\def\jussieuhome{@imj-prg.fr}

\begin{document}
\title{A relative entropy and a unique continuation result for Ricci expanders}
\author{Alix Deruelle}
\address{Alix Deruelle: 
Institut de Math\'ematiques de Jussieu, Paris Rive Gauche (IMJ-PRG) UPMC - Campus Jussieu, 4, place Jussieu Boite Courrier 247 - 75252 Paris Cedex 05, France}
\curraddr{}
\email{alix.deruelle\jussieuhome}

\author{Felix Schulze}
\address{Felix Schulze: 
Mathematics Institute, Zeeman Building, University of Warwick, Coventry CV4 7AL, UK}
\curraddr{}
\email{felix.schulze\warwickhome}

\subjclass[2000]{}

\dedicatory{}


\begin{abstract}  
We prove an optimal relative integral convergence rate for two expanding gradient Ricci solitons coming out of the same cone. As a consequence, we obtain a unique continuation result at infinity and we prove that a relative entropy for two such self-similar solutions to the Ricci flow is well-defined.
\end{abstract}

\maketitle
\tableofcontents

\date{\today}
\section{Introduction}
\subsection{Overview}

A \textit{Ricci soliton} is a triple $(M^n,g,X)$ where $(M^n,g)$ is a Riemannian manifold and a vector field $X$ satisfying the equation 
\begin{equation*}
\Ric(g)-\frac{1}{2}\Li_X(g)+\frac{\lambda}{2} g=0,
\end{equation*}
for some $\lambda\in\{-1,0,1\}$. We call $X$ the \textit{soliton vector field}. A soliton is said to be \textit{steady} if $\lambda=0$, \textit{expanding} if $\lambda=1$, and \textit{shrinking} if $\lambda=-1$. Moreover, if $X=\nabla^gf$ for some real-valued smooth function $f$ on $M$ called the potential function then $(M^n,g,\nabla^gf)$ is said to be a \textit{gradient} soliton. In this paper, we focus on expanding gradient Ricci solitons whose equation reduces to
\begin{eqnarray}
2\Ric(g)+g=\Li_{\nabla^g f}(g).\label{equ:sol-equ}
\end{eqnarray}
Notice that equation (\ref{equ:sol-equ}) normalizes the metric and defines the potential function $f$ up to an additive constant.

A Ricci soliton is said to be \textit{complete} if $(M^n,g)$ and the vector field $\nabla^g f$ are complete in the usual sense.
By \cite{Zhang-Com-Ricci-Sol}, the completeness of $(M^n,g)$ implies the completeness of $\nabla^g f$.
To each expanding gradient Ricci soliton $(M^n,g,\nabla^gf)$, one may associate a self-similar solution of the Ricci flow as follows:
\begin{equation*}
g(t)=t\varphi_{t}^*g,
\end{equation*}
where $(\varphi_{t})_{t>0}$ is the one-parameter family of diffeomorphisms generated by the vector field $-\nabla^g f/{t}$ such that $\varphi_{t=1}=\Id_M$.
This solution is \textit{Type III}, i.e.~there exists a nonnegative constant $C$ such that for any $t\in(0,+\infty)$,   
\begin{equation*}
t\sup_{M}\arrowvert\Rm(g(t))\arrowvert\leq C,
\end{equation*}
 if the curvature is bounded on the manifold $M^n$. Therefore, it is likely that expanding gradient Ricci solitons are good candidates for singularity models for Type III solutions to the Ricci flow. To illustrate these heuristics more accurately, let us mention that the second author and Simon \cite{Sch-Sim} have shown that expanding gradient Ricci solitons naturally arise as a blow-down of non-compact non-collapsed Type III solutions with non-negative curvature operator. 

Given 
an expanding gradient Ricci soliton $(M^n,g,\nabla^gf)$ with quadratic Ricci curvature decay together with covariant derivatives, one can associate a unique tangent cone $(C(S),dr^2+r^2g_S,o)$ with a smooth Riemannian link $(S,g_S)$: \cite{Sie-phd, Che-Der, Der-Asy-Com}. In particular, such an expanding gradient Ricci soliton is \textit{asymptotically conical}. Moreover, the metric cone $(C(S),dr^2+r^2g_S,o)$ can be interpreted as the initial condition of the Ricci flow $(g(t))_{t>0}$ associated to the soliton in the sense that $(M^n,d_{g(t)},p)_{t>0}$ converges to $(C(S),dr^2+r^2g_S,o)$ in the pointed Gromov-Hausdorff sense as $t\to 0$, if $p$ is a critical point of the potential function $f$.

We refer to Yai's recent work \cite{Yai-Ste-Exp} for an application of the existence of continuous families of positively curved asymptotically conical expanding gradient Ricci solitons, based on the work of the first author \cite{Der-Pos-Exp}, to show the existence of certain collapsed steady gradient Ricci solitons. 

In this article we investigate the uniqueness question among the class of asymptotically conical expanding gradient Ricci solitons coming out of a given metric cone over a smooth link.

Let us mention first some basic examples of such self-similar solutions. A fundamental example in all dimensions is the Gaussian soliton $(\mathbb{R}^n,\delta_\text{eucl},r\partial_r/2)$. In dimension $2$, any expanding gradient Ricci soliton asymptotic to a $2$-dimensional cone (necessarily flat) is rotationally symmetric since $J(\nabla^g f)$ is a Killing field, where $J$ is the natural almost complex structure associated to any surface. Therefore, the soliton equation reduces to an ordinary differential equation. It is shown in \cite[Chapter $4$]{Cho-Lu-Ni} that there is a unique one-parameter family $g_{c}$ of such expanding gradient Ricci solitons, each metric $g_c$ being K\"ahler and asymptotic to a Euclidean cone of angle $c\in(0,2\pi)$. The first examples in higher dimensions are due to Bryant: Chodosh \cite{Cho-Rot-Sym} proved that any expanding gradient Ricci soliton with positive sectional curvature and asymptotic to $(C(\mathbb{S}^{n-1}),dr^2+(cr)^2g_{\mathbb{S}^{n-1}},o)$, $c\in(0,1)$, must be rotationally symmetric which implies that it must be isometric to the corresponding Bryant soliton.
 Then Cao \cite{Cao-Egs} on $\mathbb{C}^n$ and Feldman-Ilmanen-Knopf \cite{Fel-Ilm-Kno} by extending Cao's ansatz have shown the existence of K\"ahler expanding gradient Ricci solitons living on the total space of the tautological line bundles $L^{-k}$, $k>n$, over $\mathbb{CP}^{n-1}$. These expanding solitons are $U(n)$-invariant and asymptotic to the rotationally symmetric cone $\left(C(\mathbb{S}^{2n-1}/\mathbb{Z}_k), i\partial\bar{\partial} |z|^{2p}/p\right)$, the aperture $p>0$ of the cone being arbitrary here. See \cite{Dan-Wan, Fut-Wan, Ang-Kno} for more examples based on ODE methods. 

Notice that uniqueness holds true when considering the class of complete expanding gradient K\"ahler-Ricci solitons coming out of K\"ahler cones: see \cite{Con-Der-Egs} and \cite[Corollary B]{Con-Der-Sun} for a precise statement. 

Observe that the uniqueness question is of interest even in the case of an asymptotic cone $(C(S),dr^2+r^2g_S,r\partial_r/2)$ which is Ricci flat and endowed with the radial vector field $r\partial_r/2$ since it is an exact expanding gradient Ricci soliton outside the tip. In particular, if a complete expanding gradient Ricci soliton comes out of $(C(S),dr^2+r^2g_S,r\partial_r/2)$ then uniqueness of the Ricci flow fails in case metric singularities are allowed. Now, even if we restrict our attention to complete expanding gradient Ricci solitons coming out of a given Ricci flat cone, Angenent-Knopf \cite{Ang-Kno} have proved that uniqueness still fails for some Ricci flat cones in dimension greater than $4$. 
 \subsection{Main results}

The first main result of this paper is the following unique continuation result at infinity for two expanding gradient Ricci solitons coming out of the same cone. 

 \begin{theo}[Informal statement of the main unique continuation result]\label{main-thm-uni-con-intro}
 Let $(M^n,g_1,\nabla^{g_1}f_1)$ and $(M^n,g_2,\nabla^{g_2}f_2)$ be two expanding gradient Ricci solitons coming out of the same cone $(C(S),g_C:=dr^2+r^2g_S,\frac{r}{2}\partial_r)$ over a smooth link $(S,g_S)$. Assume the soliton metrics $g_1$ and $g_2$ are gauged in such a way that their soliton vector fields coincide outside a compact set. Then the trace at infinity 
 $$\lim_{r\rightarrow +\infty}r^{n}e^{\frac{r^2}{4}}(g_1-g_2)=:\tr_{\infty}\Big(r^{n}e^{\frac{r^2}{4}}(g_1-g_2)\Big)$$ 
 exists in the $L^2(S)$-topology, it preserves the radial vector field $\partial_r$ and its tangential part is divergence free with respect to the metric on the link in the weak sense.  
Moreover, $g_1$ and $g_2$ coincide pointwise outside a compact set if and only if their associated trace at infinity vanishes, i.e. 
$$\tr_{\infty}\Big(r^{n}e^{\frac{r^2}{4}}(g_1-g_2)\Big)\equiv0\, .$$
 \end{theo}
 
The main tool to show the above result is as follows: we establish the existence of a suitable frequency function at infinity, where the method follows the work of Bernstein for mean curvature flow in codimension one \cite{Ber-Asym-Struct}, which itself is based on the  fundamental work of Garofalo-Lin \cite{Gar-Lin-Uni}. The main difficulty and crucial point in this approach comes from the fact that different to the case of mean curvature flow, where the graphical representation at infinity of one expander over the other yields a natural well-controlled gauge, in the current setting it is necessary to establish a suitable gauge at infinity between the two expanders. To establish the needed decay estimates for the frequency function it is necessary to simultaneously control the gauge together with the frequency function. The gauge we employ is a Bianchi gauge, motivated by the work of Kotschwar \cite{Kot-Bia} for the comparison of two general solutions of Ricci flow. Due to self-similarity of our solutions the evolution equation for the Bianchi gauge turns into an ODE, which then results in an ODE-PDE system for the frequency function set-up.
 
Kotschwar-Wang \cite{Kot-Wan-Bac} have employed Carleman estimates to prove the uniqueness of Ricci shrinkers smoothly asymptotic to a smooth cone. We expect that similar to work of Bernstein \cite{Ber-Asym-Struct} for mean curvature flow, the methods in this paper can be adapted to give an alternative proof of the result of Kotschwar-Wang. But different to the case treated by Bernstein, the setup for Ricci shrinkers does not directly transform in the system treated in the current paper. A unique continuation result for expanders asymptotic to Ricci flat cones was obtained by the first author using Carleman estimates \cite{Uni-Con-Egs-Der}. The results of Bernstein for mean curvature flow have been extend to the higher codimension case by Khan \cite{Khan-Uniqueness}. The unique continuation result of Bernstein \cite{Ber-Asym-Struct} has been employed centrally by Bernstein-Wang \cite{Ber-Wang-Space-Exp} in their proof that the space of expanders smoothly asymptotic to smooth cones has the structure of a smooth Banach manifold. Frequency bounds for solutions to a general class of drift laplacians equations have been obtained by Colding-Minicozzi in \cite{Col-Min-Frequ}.

In case the asymptotic cone is Ricci flat, the convergence rate was shown to hold pointwise in the smooth sense in \cite{Uni-Con-Egs-Der}. For an arbitrary asymptotic cone, Theorem \ref{main-thm-uni-con-intro} shows that the same convergence rate holds for the $L^2$ norm on level sets of the distance function from the tip of the cone. 

As an application of the decay estimates achieved via the frequency function, we show the existence of a relative entropy for two expanders asymptotic to the same cone.

\begin{theo}[A relative entropy for two expanders coming out of the same cone]\label{theo-rel-ent-def-intro}
 Let $(M^n,g_1,\nabla^{g_1}f_1)$ and $(M^n,g_2,\nabla^{g_2}f_2)$ be two expanding gradient Ricci solitons coming out of the same cone $(C(S),g_C:=dr^2+r^2g_S,\frac{r}{2}\partial_r)$ over a smooth link $(S,g_S)$. Then the following limit exists  for all $t>0$ and is constant in time:
 \begin{equation}
\W(g_2(t),g_1(t)):=\lim_{R\rightarrow+\infty}\bigg(\int_{f_2(t)\leq R}\frac{e^{f_2(t)}}{(4\pi t)^{\frac{n}{2}}}\,d\mu_{g_2(t)}-\int_{f_1(t)\leq R}\frac{e^{f_1(t)}}{(4\pi t)^{\frac{n}{2}}}\,d\mu_{g_1(t)}\bigg).\label{rel-ent-def-intro}
\end{equation}
 \end{theo}

Feldman-Ilmanen-Ni \cite{Fel-Ilm-Ni} have introduced a forward reduced volume and an expanding entropy (denoted by $\mathcal{W}_+$) that detect expanding gradient Ricci solitons on a \textit{closed} manifold. The purpose of Theorem \ref{theo-rel-ent-def-intro} is to provide a replacement of the aforementioned functionals in the non-compact setting.

In order to prove that the limit \eqref{rel-ent-def-intro} in Theorem \ref{theo-rel-ent-def-intro} is well-defined, we invoke the integral convergence rate for the difference of the soliton metrics $g_2-g_1$ obtained in Theorem \ref{main-thm-uni-con-intro}. Observe that comparing the solutions to their common initial cone metric only yields a quadratic decay and is therefore not sufficient to ensure the existence of the limit \eqref{rel-ent-def-intro}. We underline the fact that \eqref{rel-ent-def-intro} is established by taking differences rather than by considering a renormalization.

 For mean curvature flow, the authors established the existence of a relative entropy for expanders asymptotic to the same cone in \cite{Der-Sch-Rel-Ent}. For the harmonic map heat flow the first author established a similar result in \cite{Der-Rel-Ent-HMF}. Unlike the proof of the aforementioned works, we underline the fact that the main difficulty in proving Theorem \ref{theo-rel-ent-def-intro} lies in the absence of pointwise estimates on the difference  of the soliton metrics $g_2-g_1$ and its derivatives, usually obtained via the maximum principle. This partially motivates the technical effort needed to prove Theorem \ref{theo-rel-ent-def-intro} in Section \ref{sec-rel-ent}. 
  
 For an extension to general solutions  coming out of a cone in the case of mean curvature flow, see the paper by Bernstein-Wang \cite{Ber-Wan-Rel-Exp}. In future work, we will extend the above result to general Type III solutions of the Ricci flow coming out of a cone. 

\subsection{Outline of paper}
We begin in Section \ref{sec:setup} by recalling the basics of expanding gradient Ricci solitons, the soliton identities that we use all along the paper, as well as the asymptotic geometry of self-similar solutions of the Ricci flow coming out of a metric cone over a smooth link. Proposition \ref{system-diff-exp-bis} and Corollary \ref{coro-syst-cons} reduce the proof of Theorem \ref{main-thm-uni-con-intro} to the analysis of a coupled system of ODE-PDE equations. This lets us define the boundary and flux functions associated to such a system: Lemmata \ref{lemma-Ber-3-2} and \ref{lemma-Ber-3-4} start estimating the variation of the boundary function as in \cite{Ber-Asym-Struct}. 

In Section \ref{sec-int-est}, we analyse the above mentioned coupled system of ODE-PDE equations at the energy level to ensure that all relevant tensors do belong in a suitable weighted $L^2$ space as soon as two expanding soliton metrics are coming out of the same cone. This legitimates the definition of the rescaled energy of the difference of two soliton metrics coming out of the same cone over ends of the cone and it validates all the integration by parts in the remaining sections. Here, we proceed as in \cite[Section $9$]{Ber-Asym-Struct} with the difference that the Bianchi one-form needs to be estimated in terms of the difference of the two soliton metrics: see Lemmata \ref{lemma-bianchi-bis} and \ref{lemma-bianchi-level-set-var}. This cumulates in the statements of Theorem \ref{theo-dec-same-cone-ene-spa} and Corollary \ref{coro-dont-think-mild-dec-anymore}. 

Section \ref{pre-int-bd} proves that the two frequency functions related to the energy and the flux of the difference of two soliton metrics are well-defined unless the soliton metrics are equal: this is the content of Corollary \ref{lemma-3-5-Ber}. Moreover, Lemma \ref{lemma-3-3-Ber-a} shows these two frequency functions control each other whenever they are well-defined. Finally, Lemma \ref{lemma-3-3-Ber-b} establishes a crucial qualitative Pohozaev type estimate whose proof relies on a delicate integration by parts involving both the difference of the soliton metrics and the Bianchi one-form. 

We start estimating the variation of the frequency functions in Section \ref{sec-fre-bds}. Proposition \ref{prop-mix-4-2-4-3} takes care of the generic variation of the frequency function related to the flux through a so called Rellich-Necas identity adapted to $C^2_{loc}$ symmetric $2$-tensors. Based on Lemma \ref{lemma-3-3-Ber-b}, this frequency function is shown to satisfy a first order differential inequality: this is the content of Corollary \ref{coro-4-5-Ber}. The variation of the frequency function related to the energy is more subtle because of the presence of the $L^2$ norm of the Lie derivative of the Bianchi one-form on spheres of large radii. To circumvent this issue, an integration by parts is performed on such spheres thanks to Lemma \ref{lemma-IBP-Lie-level-set} and culminates in the statement of Proposition \ref{prop-4-4-Ber}. This frequency function satisfies in turn a first order differential inequality modulo the derivative of the frequency function associated to the Bianchi one-form, the latter being taken care by Proposition \ref{prop-never-stops}. 

In Section \ref{sec-fre-dec}, the frequency function $\hat{N}$ associated to the energy of the difference of two such soliton metrics is shown to decay sub-quadratically: this is the content of Theorem \ref{theo-5-1-Ber}. The almost optimality of this decay is due to the presence of the Bianchi one-form. As a preliminary step, Theorem \ref{theo-4-1-Ber} first asserts that $\hat{N}$ merely tends to $0$ at infinity then Proposition \ref{prop-5-4-Ber} starts discretizing the bootstrapping technique needed for the proof of Theorem \ref{theo-5-1-Ber}. 

With the results of Section \ref{sec-fre-dec} in hand, Theorem \ref{main-thm-uni-con-intro} is proved in Section \ref{sec-dec-est-tra-inf}: see Theorem \ref{theo-6-1-Ber} for a rigorous statement. Propositions \ref{prop-bianchi-id-sequ-lim} and \ref{prop-div-free-tr-infty} establish further properties of the trace at infinity of the difference of two soliton metrics coming out of the same cone. Their proofs would be almost straightforward if such trace at infinity (and its convergence to it) was smooth. These statements find their counterparts for K\"ahler expanding gradient Ricci solitons coming out of a cone in Section \ref{subsec-kah-egs-tra}. 

In the last Section \ref{sec-rel-ent}, Theorem \ref{theo-rel-ent-def} establishes the existence of a relative entropy between two expanding gradient Ricci solitons coming out of the same cone. Its proof relies on the integral convergence rate obtained in Section \ref{sec-dec-est-tra-inf} for the decay of the difference of two such soliton metrics only. The unique continuation result is not needed here. 

\subsection{Acknowledgements}
 The first author was supported by grant ANR-17-CE40-0034 of the French National Research Agency ANR (Project CCEM).	
The second author was supported by a Leverhulme Trust Research Project Grant RPG-2016-174.

\section{Setup}\label{sec:setup}

 \subsection{A gauge at infinity for expanding gradient Ricci solitons}~~\\[-2ex]
 
In this section, we explain the results obtained in \cite{Con-Der-Sun} to identify the soliton vector field associated to an asymptotically conical expanding gradient Ricci soliton to the radial vector field $\frac{r}{2}\partial_r$ on the tangent cone at infinity.
 
 A Riemannian manifold $(M^n,g)$ has quadratic curvature decay with derivatives if for some (hence any) point $p\in M$ and all $k\in\mathbb{N}_{0}$,
\begin{equation*}
A_{k}(g):=\sup_{x\in M}|(\nabla^{g})^{k}\operatorname{Rm}(g)|_{g}(x)d_{g}(p,\,x)^{2+k}<\infty,
\end{equation*}
where $\operatorname{Rm}(g)$ denotes the curvature of $g$ and $d_{g}(p,\,x)$ denotes the distance between $p$ and $x$ with respect to $g$.

Finally, to each expanding gradient Ricci soliton $(M^n,g,\nabla^gf)$, consider the one-parameter family $(\varphi_t)_{t>0}$ of diffeomorphisms generated by $-\frac{\nabla^gf}{t}$ and define as in Siepmann \cite{Sie-phd} the following sets, for any $a\geq 0$:
\begin{equation*}
M_{a}:=\Big\{x\in M\,|\,\liminf_{t\to 0^{+}}tf(\varphi_t(x))>a\Big\}.
\end{equation*}
These sets are invariant under the diffeomorphisms $(\varphi_t)_{t\in(0,1]}.$

With these definitions in hand, we can state the main result we need in this paper:
 
 \begin{theo}[Map to the tangent cone for expanding Ricci solitons]\label{ac-str-exp}
Let $(M^n,g,\nabla^gf)$ be a complete expanding gradient Ricci soliton with quadratic curvature decay with derivatives. 
Then on each end of $M$,
\begin{enumerate}
\item the one-parameter family of functions $(tf\circ\varphi_{t})_{t>0}$ converges to a non-negative
continuous real-valued function $q(x):=\lim_{t\to0^{+}}tf(\varphi_{t}(x))$ pointwise on $M$ as $t\to 0^{+}$. Moreover, the convergence is locally smooth on $M_0$.\\[-1ex]
\item The metrics $(g(t))_{t>0}$ converge to a Riemannian metric $\tilde{g}_{0}$ in $C^{\infty}_{\operatorname{loc}}(M_{0})$ as $t\to0^{+}$ and $\tilde{g}_{0}=2\operatorname{Hess}_{\tilde{g}_{0}}q$.\\[-1ex]
\item The function $q$ is proper and satisfies on $M_0$, $|\nabla^{\tilde{g}_{0}}q|^{2}_{\tilde{g}_{0}}=q$ and $\nabla^{\tilde{g}_{0}}q=\nabla^gf$. In particular, there exists $c_0>0$ such that the level sets $q^{-1}\{\frac{c^2}{4}\}$ are compact connected smooth hypersurfaces for $c\geq c_0.$\\[-1ex]
\item Define $\rho:=2\sqrt{q}$ on $M_0$ and let $S:=\rho^{-1}(\{c\})$ for any $c$ large enough. 
Then there exists a diffeomorphism $\iota:(c,\,\infty)\times S\to M_{\frac{c^{2}}{4}}$ such that $g_{0}:=\iota^{*}\tilde{g}_{0}=dr^{2}+r^{2}\frac{g_{S}}{c^{2}}$ and $d\iota(r\partial_{r})=2\nabla^gf$, where
$r$ is the coordinate on the $(c,\,\infty)$-factor and $g_{S}$ is the restriction of $\tilde{g}_{0}$ to $S$. Moreover, along $M_{\frac{c^{2}}{4}}$, we have that
\begin{equation*}\label{northbeach}
|(\nabla^{g_{0}})^{k}(\iota^{*}g-g_{0}-2\operatorname{Ric}(g_{0}))|_{g_{0}}=O(r^{-4-k})\quad\textrm{for all $k\in\mathbb{N}_{0}$}.
\end{equation*}
\end{enumerate}
\end{theo}
The statement and the proof of Theorem \ref{ac-str-exp} is the combination of the statements and the proofs contained in \cite[Section $3.1$]{Con-Der-Sun}.

\begin{rk}
In the sequel, we apply Theorem \ref{ac-str-exp} to two expanding gradient Ricci solitons $(M^n,g_i,\nabla^{g_i}f_i)_{i=1,2}$ with quadratic curvature decay with derivatives which come out of the same cone. In particular, with the notations of Theorem \ref{ac-str-exp}, there exist two diffeomorphisms $\iota_i$, $i=1,2$, such that $\lim_{t\rightarrow 0^+}t\iota_i^*{\varphi^i_t}^{*}f_i=\frac{r^2}{4}$ on the end of the common metric cone where $(\varphi_t^i)_{t>0}$ is the flow generated by $-\nabla^{g_i}f_i/t$ and 
\begin{equation}\label{northbeach-bis}
|(\nabla^{g_{0}})^{k}(\iota_2^{*}g_2-\iota_1^{*}g_1)|_{g_{0}}=O(r^{-4-k})\quad\textrm{for all $k\in\mathbb{N}_{0}$}.
\end{equation}
\end{rk}

\subsection{Soliton identities}\label{sol-id-sec}~~\\[-2ex]

 We consider an asymptotically conical expanding gradient Ricci soliton $(M^n,g,\nabla^g f)$. We recall the following fundamental identities
\begin{eqnarray}
&& \nabla^{g,2} f = \Ric(g) + \frac{g}{2}, \label{eq:0-0} \\
&&\Delta_g f = \R_g+\frac{n}{2}, \label{eq:1-0} \\
&&\nabla^g \R_g+ 2\Ric(g)(\nabla^g f)=0, \label{eq:2-0} \\
&&\arrowvert \nabla^g f \arrowvert_g^2+\R_g=f, \label{eq:3-0}\\
&&\div_g\Rm(g)(Y,Z,U)=\Rm(g)(Y,Z, \nabla^g f,U),\label{eq:4-0}
\end{eqnarray}
for any vector fields $Y$, $Z$, $U$ whose proofs can be found in \cite[Chapter $4$]{Cho-Lu-Ni} for instance. Identity \ref{eq:3-0} holds up to an additive constant in general. This choice normalizes the potential function $f$. 

In the sequel, an expanding gradient Ricci soliton is said to be \textbf{normalized} if \eqref{eq:3-0} holds on $M$.

Now, outside of a compact set $K \subset M$ we have that $f$ is smooth and positive and we can define
$$\br := 2 \sqrt{f}\, .$$
We note that \eqref{eq:0-0} implies that for $\br \geq R$ and $R$ sufficiently large there exists $\Lambda >0$ such that
\begin{equation}\label{eq:5}
\left|\nabla^{g,2}\br^2 - 2g\right|_g \leq \frac{\Lambda}{\br^2} \leq 1/2\ ,
\end{equation}
as well as \eqref{eq:3-0} implies that
\begin{equation}\label{eq:6}
\left| |\nabla^g\br|_g -1\right| \leq  \frac{\Lambda}{\br^4} \leq 1/2\ ,
\end{equation}
where we used the quadratic decay of the curvature along the end as well as the existence of a $C>0$ such that
$$ \frac{r^2}{4} - C \leq f \leq \frac{r^2}{4} + C$$
where $r$ is the radial coordinate along the asymptotic cone $C(S)$.
Moreover, we record the following estimates that follow essentially from \eqref{eq:5}:
\begin{equation}
\left|\div_g\nabla^g\br-\frac{n-1}{\br}\right|=\left|\Delta_g\br-\frac{n-1}{\br}\right|\leq \frac{\Lambda}{\br^3}.\label{eq:div-rad}
\end{equation}
 Then we see that we have what is denoted in \cite{Ber-Asym-Struct} a 'weakly conical end'. We also define the vector fields
$$\partial_\br = \nabla^g\br,\quad \bN=\frac{\nabla^g\br}{|\nabla^g\br|} \text{ and } \bX = \br \frac{\nabla^g\br}{|\nabla^g\br|^2}\, .$$
For further notation and basic estimates we refer to \cite[Section 2]{Ber-Asym-Struct}.

We now define 
$$ \Phi_m = \br^m e^{-\frac{\br^2}{4}} = 2^m f^\frac{m}{2} e^{-f} .$$
The self-adjoint operator with respect to this weight is
$$\calL_m = \Delta_g - \frac{\br}{2} \nabla^g_{\nabla^g \br} + \frac{m}{\br}\nabla^g_{\nabla^g\br} =: \Delta_{-f} +  \frac{m}{\br}\partial_{\br}\, .$$
Similarly we consider
$$ \Psi_m = \br^m e^\frac{\br^2}{4} = 2^m f^\frac{m}{2} e^f .$$
Note that the self adjoint operator with respect to this weight is given by 
$$\calL^+_m = \Delta_g + \frac{\br}{2} \nabla^g_{\nabla^g \br} + \frac{m}{\br}\partial_{\br} =: \Delta_f +  \frac{m}{\br}\partial_{\br}\, .$$

Finally, we recall the link between the Lie derivative of a tensor $T$ on $M$ along a smooth vector field $X$ on $M$ and the covariant derivative of $T$ along $X$:
\begin{lemma}\label{lemm-link-lie-der-cov-der}
Let $(M^n,g)$ be a Riemannian manifold and let $X$ be a smooth vector field on $M$. If $T$ is a tensor of type $(p,0)$ on $M$, then:
\begin{equation*}\label{li-der-levi-civ-rel.1}
\Li_XT(Y_1,...,Y_p)=\nabla^g_XT(Y_1,...,Y_p)+\sum_{i=1}^pT(Y_1,...,Y_{i-1},\nabla^g_{Y_i}X,Y_{i+1},...,Y_p),
\end{equation*}
where $Y_i$, $i=1,...,p$, are smooth vector fields on $M$. In particular, if $T$ is a $(2,0)$-tensor and if $X=\nabla^gf$ for some smooth function $f:M\rightarrow \R$, then,
\begin{equation}
\Li_{\nabla^gf}T=\nabla^g_{\nabla^gf}T+T\circ\nabla^{g,2}f+\nabla^{g,2}f\circ T.\label{li-der-levi-civ-rel}
\end{equation}

\end{lemma}~~\\[-3ex]
\subsection{An ODE-PDE system}\label{ode-pde-sec}~~\\[-2ex]

 Let $(M^n,g_1,\nabla^{g_1}f_1)$ and $(M^n,g_2,\nabla^{g_2}f_2)$ be two expanding gradient Ricci solitons smoothly coming out of the same metric cone $(C(S),dr^2+r^2g_X,r\partial_r/2)$ such that outside a compact set, $\nabla^{g_1}f_1=\frac{1}{2}r\partial_r=\nabla^{g_2}f_2$. This is legitimated by Theorem \ref{ac-str-exp} up to a diffeomorphism at infinity. We say that the soliton metrics $g_1$ and $g_2$ are gauged with respect to the soliton vector field.

If $h:=g_2-g_1$, we define a one parameter family of metrics $(g_{\sigma})_{\sigma\in[1,2]}$ and a one parameter family of potential functions $(f_{\sigma})_{\sigma\in[1,2]}$ on a fixed end of $M$ as follows:
\begin{equation*}
g_{\sigma}:=(2-\sigma)g_1+(\sigma-1)g_2=g_1+(\sigma -1)h, \quad f_{\sigma}:=(2-\sigma)f_1+(\sigma-1)f_2.
\end{equation*}
Moreover, we define the linearized Bianchi one-form associated to $g_1$ and $g_2$ by:
\begin{equation}
\calB:=\div_{g_1}h-\frac{1}{2}g_1\left(\nabla^{g_1}\tr_{g_1}h,\cdot\right).\label{def-bianchi-gau}
\end{equation}
Finally, define the action of the curvature tensor on symmetric $2$-tensors on $M$ as follows : $$(\Rm(g_1)\ast h)_{ij}:=\Rm(g_1)_{iklj}h_{kl},\quad\text{ for $h\in S^2T^*M.$}$$

We record the following lemma that sums up the qualitative properties of the soliton identities from the previous section applied to the family of approximate Ricci expanding solitons with metric $g_{\sigma}$ and potential function $f_{\sigma}$ for $\sigma\in[1,2]$ we will use in the next sections:
\begin{lemma}\label{sol-id-qual}
Under the above assumptions and notations, define $\br_{\sigma}:=2\sqrt{f_{\sigma}}$, $\sigma\in[1,2]$. Then the following estimates hold true and are uniform in $\sigma\in[1,2]$:
\begin{eqnarray}
&&|f_{\sigma}-f_1|\leq C,\label{eq:-1}\\
&& \left|\nabla^{g_{\sigma},2} f_{\sigma} - \frac{g_{\sigma}}{2}\right|_{g_{\sigma}}\leq C\br_{\sigma}^{-2}, \label{eq:0} \\
&&\left|\Delta_{g_{\sigma}} f_{\sigma} -\frac{n}{2}\right|\leq C\br_{\sigma}^{-2}, \label{eq:1} \\
&&|\Ric(g_{\sigma})(\nabla^{g_{\sigma}} f_{\sigma})|_{g_{\sigma}}\leq C\br_{\sigma}^{-3}, \label{eq:2} \\
&&\arrowvert \nabla^{g_{\sigma}}f_{\sigma}-\nabla^{g_1}f_1\arrowvert_{g_{\sigma}}\leq C\br_{\sigma}|h|_{g_{\sigma}}\leq C\br_{\sigma}^{-3}, \label{eq:3}\\
&&|\Rm(g_{\sigma})(Y,Z, \nabla^{g_{\sigma}} f_{\sigma},U)|_{g_{\sigma}}\leq C\br_{\sigma}^{-3},\label{eq:4}
\end{eqnarray}
for any unit vector fields $Y$, $Z$, $U$ with respect to $g_{\sigma}$.
In particular, $\br_{\sigma}$ satisfies estimates \eqref{eq:5}, \eqref{eq:6} and \eqref{eq:div-rad} which are uniform in $\sigma\in[1,2]$.
\end{lemma}

\begin{rk}
In particular, estimate \eqref{eq:-1} not only implies that the potential functions $f_2$ and $f_1$ are equivalent but also that the exponential weights $e^{f_1}$ and $e^{f_2}$ are uniformly comparable.
\end{rk}
\begin{proof}
We only prove \eqref{eq:-1} and \eqref{eq:3}, the other estimates can be proved along the same lines.

In order to prove \eqref{eq:-1}, recall by \eqref{eq:3-0} that if $(\varphi_t)_{t>0}$ denotes the flow generated by $-\nabla^{g_1}f_1/t$ such that $\varphi_t|_{t=1}=\Id_M$ and if $f_1(t):=\varphi_t^*f_1$, $\partial_t(tf_1)=t\R_{g_1(t)}$ for $t>0$. By invoking Theorem \ref{ac-str-exp} again and by integrating the previous differential equation from $t=0$ to $t=1$, one gets on each end of $M$,
\begin{equation*}
\left|f_1-\frac{r^2}{4}\right|=\left|\int_0^1t\R_{g_1(t)}\,dt\right|\leq C,
\end{equation*}
since $t\R_{g_1(t)}$ is bounded on $M\times(0,1]$. The same is true for the soliton metric $g_2$ and the potential function $f_2$ so that the triangle inequality implies the expected estimate, i.e. $|f_2-f_1|\leq C$ on each end of $M$ for some positive constant $C$.

Regarding the proof of \eqref{eq:3}, the main observation is that $|g_{\sigma}-g_1|_{g_1}\leq |g_2-g_1|_{g_1}\leq C\br^{-4}$ by Theorem \ref{ac-str-exp} together with the fact that $\nabla^{g_2}f_2=\nabla^{g_1}f_1$. Indeed, 
\begin{equation*}
\begin{split}
\left|\nabla^{g_{\sigma}}f_{\sigma}-\nabla^{g_1}f_1\right|_{g_1}&=\left|(\sigma-1)\nabla^{g_{\sigma}}f_2+(2-\sigma)\nabla^{g_{\sigma}}f_1-\nabla^{g_1}f_1\right|_{g_1}\\
&=\left|(\sigma-1)\left(\nabla^{g_{\sigma}}f_2-\nabla^{g_2}f_2\right)+(2-\sigma)\left(\nabla^{g_{\sigma}}f_1-\nabla^{g_1}f_1\right)\right|_{g_1}\\
&\leq C\left(|g_{\sigma}-g_1|_{g_1}+|g_{\sigma}-g_2|_{g_1}\right)\br_1\leq C |h|_{g_1} \br_1 \leq C\br_1^{-3},
\end{split}
\end{equation*}
where we have used the estimate $|\nabla^{g_i}f_i|_{g_i}=O(\br_1)$ for $i=1,2$.
\end{proof}

\begin{prop}\label{system-diff-exp-bis}
Under the above assumptions and notations, define $\Delta_{f_{\sigma}}h:=\Delta_{g_{\sigma}}h+\nabla^{g_{\sigma}}_{\nabla^{g_{\sigma}}f_{\sigma}}h$. Then, the following system holds outside a sufficiently large compact set of $M$ for all $\sigma\in[1,2]$:
\begin{eqnarray}
\Delta_{f_{\sigma}}h &=&\Li_{\calB}(g_{\sigma})+R[h],\label{eq:main.1}\\
\left|\nabla^{g_{\sigma}}_{\frac{1}{2}\br_{\sigma}\partial_{\br_{\sigma}}}\calB-\frac{1}{2}\calB\right|_{g_{\sigma}}&\leq& C\left(\br_{\sigma}^{-3}|h|_{g_{\sigma}}+\br_{\sigma}^{-2}|\nabla^{g_{\sigma}}h|_{g_{\sigma}}\right),\label{eq:main.2}
\end{eqnarray}
where
\begin{equation}\label{est-remain-term}
|R[h]|_{g_{\sigma}}\,\leq C \br_{\sigma}^{-2}\left(|h|_{g_{\sigma}}+\br_{\sigma}^{-1}|\nabla^{g_{\sigma}}h|_{g_{\sigma}}\right),
\end{equation}
for some positive constant $C$ uniform in $\sigma\in[1,2]$.
\end{prop}
\begin{proof}
We first prove this proposition for $\sigma=1$, i.e. with respect to the soliton metric $g_1$ and the potential function $f_1$.

Notice that the associated Ricci flows to the soliton metrics $g_1$ and $g_2$ are respectively $g_1(t)=t\psi_t^*g_1$ and $g_2(t)=t\psi_t^*g_2$ since the flows of $\nabla^{g_1}f_1$ and $\nabla^{g_2}f_2$ are identified by Theorem \ref{ac-str-exp} as explained above. Define accordingly $X:=\nabla^{g_1}f_1=\nabla^{g_2}f_2$ outside a sufficiently large compact set of $M$.

Therefore, outside a compact set of $M$, 
\begin{equation*}
-2(\Ric(g_2)-\Ric(g_1))=(g_2-g_1)-\Li_{X}g_2+\Li_{X}g_1.
\end{equation*}

Now, apply formula [\eqref{li-der-levi-civ-rel}, Lemma \ref{li-der-levi-civ-rel.1}] to the Lie derivative $\Li_{X}h$ defined on $(M^n,g_1)$ to get:
\begin{equation*}
\begin{split}
\Li_{X}g_2-\Li_{X}g_1&=\Li_{X}h\\
&=\nabla^{g_1}_{X}h+\Sym(h\circ\nabla^{g_1,2}f_1)\\
&=\nabla^{g_1}_{X}h+h+\Sym(h\circ\Ric(g_1))
\end{split}
\end{equation*}
where $\Sym(h\circ\nabla^{g_1,2}f_1)(Y,Z):=h(\nabla^{g_1}_Y\nabla^{g_1}f_1,Z)+h(Y,\nabla^{g_1}_Z\nabla^{g_1}f_1)$ for vector fields $Y,Z\in TM.$ 

As a first conclusion, we get:
\begin{equation*}
\begin{split}
-2(\Ric(g_2)-\Ric(g_1))=&-\nabla^{g_1}_{X}h-\Sym(h\circ\Ric(g_1)).
\end{split}
\end{equation*}

On the other hand, according to [\cite{Kot-Bia}, Lemma $4$], see \cite{Koc-Lam} for a similar expression:
\begin{equation*}
\begin{split}
-2(\Ric(g_2)-\Ric(g_1))&= g_2^{-1}\ast\nabla^{g_1,2}h-\Li_{\calB}(g_1)+g_2^{-1}\ast\Rm(g_1)\ast h+g_2^{-1}\ast g_2^{-1}\ast\nabla^{g_1} h\ast \nabla^{g_1} h\\
&= g_2^{-1}\ast\nabla^{g_1,2}h-\Li_{\calB}(g_1)+2\Rm(g_1)\ast h-\Sym(h\circ \Ric(g_1))+R[h],
\end{split}
\end{equation*}
where $\left(g_2^{-1}\ast\nabla^{g_1,2}h\right)_{ij}:=g_2^{kl}\nabla^{g_1,2}_{kl}(h)_{ij}$ and where $R[h]$ satisfies pointwise \eqref{est-remain-term}. Observe as in \cite{Koc-Lam} that:
\begin{equation*}
g_2^{kl}\nabla^{g_1,2}_{kl}h =g_1^{kl}\nabla^{g_1,2}_{kl}h+\left(g_2^{kl}-g_1^{kl}\right)\nabla^{g_1,2}_{kl}h =\Delta_{g_1}h+R[h],
\end{equation*}
where
$$|R[h]|_{g_1} \leq C |\Rm(g_1)|_{g_1}|h|_{g_1} + C \left(|\Rm(g_1)|_{g_1}|h|_{g_1}+|\nabla^{g_1,2}h|_{g_1}\right)|h|_{g_1}+ C |\nabla^{g_1}h|_{g_1}^2.$$
Note that the decay established in Theorem \ref{ac-str-exp} yields
$$|R[h]|_{g_1}\, \leq \,C\br_1^{-2}\left(|h|_{g_1}+\br_1^{-1}|\nabla^{g_1}h|_{g_1}\right). $$
which shows \eqref{eq:main.1} and \eqref{est-remain-term}.\\

To prove [\eqref{eq:main.2}, Proposition \ref{system-diff-exp-bis}], notice by the definition of the Bianchi gauge given in (\ref{def-bianchi-gau}) that:
\begin{equation*}
\begin{split}
\calB(t):=&\div_{g_1(t)}(g_2(t)-g_1(t))-\frac{1}{2}g_1(t)\left(\nabla^{g_1(t)}\tr_{g_1(t)}(g_2(t)-g_1(t)),\cdot\right)\\
=& \psi_t^*\calB,\label{div-tr-bian}
\end{split}
\end{equation*}
which implies by the soliton equation \eqref{eq:0},
\begin{equation}
\begin{split}
\partial_t\calB(t)\Big|_{t=1}=&-\Li_{X}\calB\\
=&-\nabla^{g_1}_{X}\calB+\nabla^{g_1}_{\calB}X\\
=&-\nabla^{g_1}_{X}\calB+\frac{1}{2}\calB+\Ric(g_1)( \calB).\label{first-step-evo-bian}
\end{split}
\end{equation}
Along the same lines as in \cite{Kot-Bia}, one gets by differentiating the righthand side of (\ref{div-tr-bian}):
\begin{equation}
\begin{split}
\partial_t\calB(t)\Big|_{t=1}&=\div_{g_1}(-2\Ric(g_2))-\frac{1}{2}g_1\big(\nabla^{g_1}\tr_{g_1}(-2\Ric(g_2)),\cdot\big)\\
&\quad +h\ast \nabla^{g_1}\Ric(g_1)+\nabla^{g_1}h\ast\Ric(g_1)\\
&=-2\left(\div_{g_1}-\div_{g_2}\right)\Ric(g_2)+g_1\left(\nabla^{g_1}\tr_{g_1}\Ric(g_2),\cdot\right)\\
&\quad -g_2\left(\nabla^{g_2}\tr_{g_2}\Ric(g_2),\cdot\right)+h\ast \nabla^{g_1}\Ric(g_1)+\nabla^{g_1}h\ast\Ric(g_1)\\
&=h\ast O(\br_1^{-3})+\nabla^{g_1}h\ast O(\br_1^{-2}).\label{sec-step-evo-bian}
\end{split}
\end{equation}
This fact uses the crucial Bianchi identity $\div_{g_1}\Ric(g_1)-\frac{1}{2}g_1\left(\nabla^{g_1}\tr_{g_1}\Ric(g_1),\cdot\right)=0$ together with the quadratic decay on the curvature and its derivatives.
The combination of \eqref{first-step-evo-bian} and \eqref{sec-step-evo-bian} leads to the desired result.

If one considers an approximate soliton metric $g_{\sigma}$ and its associated approximate potential function $f_{\sigma}$ for $\sigma\in[1,2]$, the proof of [\eqref{eq:main.1}--\eqref{est-remain-term}] is implied by the following estimate:
 \begin{equation}\label{intermed-est-syst-sigma}
|\Delta_{f_{\sigma}}h-\Delta_{f_1}h|_{g_{\sigma}}\,\leq C \br_{\sigma}^{-2}\left(|h|_{g_{\sigma}}+\br_{\sigma}^{-1}|\nabla^{g_{\sigma}}h|_{g_{\sigma}}\right).
\end{equation}
The proof of \eqref{intermed-est-syst-sigma} is in turn due to Lemma \ref{sol-id-qual} together with the following linearization of the rough laplacian:
\begin{equation*}
\left|\Delta_{g_{\sigma}}h-\Delta_{g_1}h\right|_{g_1}\leq C\sum_{i=0}^2|\nabla^{g_1,i}(g_{\sigma}-g_1)|_{g_1}|\nabla^{g_1,2-i}h|_{g_1}\leq C\sum_{i=0}^2|\nabla^{g_1,i}h|_{g_1}|\nabla^{g_1,2-i}h|_{g_1},
\end{equation*}
fo some positive constant $C$ uniform in $\sigma\in[1,2]$. The same reasoning applies to $\nabla^{g_{\sigma}}_{\nabla^{g_{\sigma}}f_{\sigma}}h$ and $\nabla^{g_{\sigma}}_{\nabla^{g_{\sigma}}f_{\sigma}}\calB$. Indeed, if $T$ is any tensor defined on $M$, then pointwise:
\begin{equation*}
\begin{split}
\left|\nabla^{g_{\sigma}}_{\nabla^{g_{\sigma}}f_{\sigma}}T-\nabla^{g_{1}}_{\nabla^{g_{1}}f_{1}}T\right|_{g_1}&\leq C\left(|\nabla^{g_1}T|_{g_1}\left|\nabla^{g_{\sigma}}f_{\sigma}-\nabla^{g_{1}}f_{1}\right|_{g_1}+\br_1\left|\nabla^{g_{\sigma}}T-\nabla^{g_{1}}T\right|_{g_1}\right)\\
&\leq C\left(|\nabla^{g_1}T|_{g_1}\left|\nabla^{g_{\sigma}}f_{\sigma}-\nabla^{g_{1}}f_{1}\right|_{g_1}+\br_1\left|\nabla^{g_{1}}(g_{\sigma}-g_1)\right|_{g_1}|T|_{g_1}\right)\\
&\leq C\left(\br_1|\nabla^{g_1}T|_{g_1}|h|_{g_1}+\br_1\left|\nabla^{g_{1}}h\right|_{g_1}|T|_{g_1}\right)
\end{split}
\end{equation*}
where we have crucially used \eqref{eq:3} in the last line.
\end{proof}

 We now consider 
\begin{equation*}\label{eq:defhhat}
\hat h := \Psi_\mu h\ .
\end{equation*}
We recall from \cite[p.~17]{Ber-Asym-Struct} that
\begin{equation*}
\calL_m \Psi_\mu = \frac{1}{2}(\mu+n+m)\Psi_\mu + O_{\mu,m}(\br^{-2}\Psi_\mu),
\end{equation*}
where $O_{\mu,m}$ means that the estimate depends on $\mu$ and $m$.

Combining this with \eqref{eq:main.1}, one can compute that
\begin{equation*}\label{eq:eqhhat_prel}
\left(\calL_{-2\mu}+ \frac{\mu-n}{2}\right) \hat{h} = \Psi_\mu \Li_{\calB}(g) + O_{\mu, n}(\br^{-2}) \hat{h}+\Psi_{\mu}R[h].
\end{equation*}
The same computation is true for the operator $\calL_{-2\mu}$ defined with respect to the one-parameter families of metrics $(g_{\sigma})_{\sigma\in[1,2]}$ and potential functions $(f_{\sigma})_{\sigma\in[1,2]}$. In order to keep the notations to a minimum, we denote such operator by the same symbol in the sequel.

Choosing $\mu =n$ and recalling the properties of the quadratic term $R[h]$ from Proposition \ref{system-diff-exp-bis}
together with the decay on $h$ and its derivatives given by Theorem \ref{ac-str-exp}, we summarize this discussion as follows:
\begin{coro}\label{coro-syst-cons}
Under the same assumptions and notations as those of Proposition \ref{system-diff-exp-bis}, the following system holds outside a sufficiently large compact set of $M$ for all $\sigma\in[1,2]$:
\begin{eqnarray}
\calL_{-2n}\,\hat{h} &=&\Psi_n \Li_{\calB}(g_{\sigma}) +R[\hat{h}],\label{eq:eqhhat}\\
\left|\nabla^{g_{\sigma}}_{\frac{1}{2}\br_{\sigma}\partial_{\br_{\sigma}}}\calB-\frac{1}{2}\calB\right|_{g_{\sigma}}&\leq& C\left(\br_{\sigma}^{-3}|h|_{g_{\sigma}}+\br_{\sigma}^{-2}|\nabla^{g_{\sigma}}h|_{g_{\sigma}}\right),
\notag
\end{eqnarray}
where,
\begin{equation*}
\begin{split}
|R[\hat{h}]|_{g_{\sigma}}\,\leq&\,C\br_{\sigma}^{-2}\left(|\hat{h}|_{g_{\sigma}}+|\nabla^{g_{\sigma}}\hat{h}|_{g_{\sigma}}\right).
\end{split}
\end{equation*}
\end{coro}
In particular, Corollary \ref{coro-syst-cons} shows that we have reduced our setting to that of \cite[Section $3$]{Ber-Asym-Struct} up to the Lie derivative term on the righthand side of \eqref{eq:eqhhat}.\\

Because of Lemma \ref{sol-id-qual}, Proposition \ref{system-diff-exp-bis} and Corollary \ref{coro-syst-cons}, \textbf{we omit the reference to the parameter $\sigma\in[1,2]$ from now on unless we explicitly make reference to it}.

\subsection{Frequency functions} \label{sec:frequency-functions}~~\\[-2ex]

On $E_R:=\{p\, |\, \br(p) > R\}$, we consider $C^k(\bar{E}_R)$ with the (incomplete) norms
 $$\|T\|_m^2 = \|T\|_{m,0}^2 = \int_{\bar{E}_R} |T|^2\,\Phi_m \text{ and }  \|T\|_{m,1}^2 =  \|T\|_{m,0}^2 +  \|\nabla T\|_{m,0}^2\ ,$$
together with the classes of tensors of controlled growth
\begin{align*}
C^k_m(\bar{E}_R) := \{ T \in C^k(\bar{E}_R)\, |\, \| T\|_{m} < \infty \} \text{ and }\\
C^k_{m,1}(\bar{E}_R) := \{ T \in C^k(\bar{E}_R)\, |\, \| T\|_{m,1} < \infty \}\, .
\end{align*}
Fixing $R>R_\Sigma$ and denote for $\rho \geq R$
\begin{equation*}
 B(\rho) := \int_{S_\rho} |\hat{h}|^2 |\nabla\br| \ \text{ and }\  \hat{B}(\rho) := \Phi_{-2n}(\rho)B(\rho)\, .
 \end{equation*}
We also define the boundary flux
\begin{equation*}\label{def-flux}
F(\rho) := \int_{S_\rho} \langle \hat{h} , \nabla_{-\bN} \hat{h} \rangle = - \int_{S_\rho} \frac{\langle \hat{h} , \nabla_{\partial_\br} \hat{h} \rangle}{|\nabla \br|} \ \text{ and } \ \hat{F}(\rho) = \Phi_{-2n}(\rho)F(\rho),
\end{equation*}
and the corresponding frequency functions
\begin{equation*}
 N(\rho) := \frac{\rho F(\rho)}{B(\rho)}= \frac{\rho \hat{F}(\rho)}{\hat{B}(\rho)}\, .
 \end{equation*}
For $\hat{h} \in C^2_{-2n,1}(\bar{E}_R)$ we let
\begin{equation*}
 \hat{D}(\rho) := \int_{\bar{E}_\rho} |\nabla \hat{h}|^2 \Phi_{-2n}\ \text{ and } \ D(\rho):=\Phi_{-2n}(\rho)^{-1} \hat{D}(\rho),
 \end{equation*}
be the $\Phi_{-2n}$-weighted Dirichlet energy and a normalized version of it and
\begin{equation*}
 \hat{N}(\rho) = \frac{\rho \hat{D}(\rho)}{\hat{B}(\rho)}= \frac{\rho D(\rho)}{B(\rho)}\, ,
 \end{equation*}
the corresponding frequency functions. We also set
\begin{equation*}
\hat{L}(\rho) = \int_{E_\rho} \langle \hat{h} , \calL_{-2n} \hat{h}\rangle \Phi_{-2n}\ \text{ and } \ L(\rho)= \Phi_{-2n}(\rho)^{-1}\hat{L}(\rho)\, ,
\end{equation*}
so that if $\calL_{-2n} h \in C^0_{-2n}(\bar{E}_R)$ then integration by parts is justified and gives
\begin{equation}
 \hat{F}(\rho) = \hat{D}(\rho) +\hat{L}(\rho) \ \text{ and }\ F(\rho) = D(\rho) +L (\rho)\, .\label{id-F-D-L}
 \end{equation}

We quote the two following lemmata whose proofs are verbatim those given in \cite{Ber-Asym-Struct}.

\begin{lemma}{\rm (\cite[Lemma 3.2]{Ber-Asym-Struct})} \label{lemma-Ber-3-2}We have 
$$ B'(\rho) = \frac{n-1}{\rho} B(\rho) - 2F(\rho) + B(\rho) O(\rho^{-3})$$
and
$$\hat{B}'(\rho) = -\frac{n+1}{\rho} \hat{B}(\rho) - \frac{\rho}{2}\hat{B}(\rho) - 2\hat{F}(\rho) + \hat{B}(\rho) O(\rho^{-3})\, .$$
\end{lemma}

\begin{lemma}{\rm (\cite[Lemma 3.4]{Ber-Asym-Struct})} \label{lemma-Ber-3-4} There is an $R>0$ so that for any $\rho\geq R$, if $\hat{h} \in C^2(\bar{A}_{\rho+2\rho})$ satisfies $-1 \leq N(\rho') \leq N_+$ for $\rho' \in [\rho, \rho+2]$ and $N_+ \geq 0$, then for all $\tau \in [0,2]$,
$$ \Big(1-2N_+ \frac{\tau}{\rho}\Big) B(\rho) \leq B(\rho+\tau) \leq \Big(1+2(n+3)\frac{\tau}{\rho}\Big) B(\rho)\, .$$
 
\end{lemma}

 \section{Integral decay estimates}\label{sec-int-est}
 From now on, we consider an asymptotically conical expanding gradient Ricci soliton $(M^n,g,\nabla^gf)$ gauged outside a compact set as in Theorem \ref{ac-str-exp}. We consider a couple $(h,\calB)$ satisfying \eqref{eq:main.1} and \eqref{eq:main.2} and derive a priori integral bounds. For the sake of clarity, we omit the dependence of the Levi-Civita connection on the background metric $g$.
   
Given a symmetric $2$-tensor $h\in H^1_{loc}$, $r_2>r_1$ and $m\in\RR$, we define the following weighted energy:
\begin{equation*}
\check{D}_m(h,r_1,r_2):=\int_{A_{r_1,r_2}}|\nabla h|^2\Psi_m.
\end{equation*}
Similarly as in Section \ref{sec:frequency-functions} for $R>R_\Sigma$ and $\rho \geq R$ we define
\begin{equation*}\label{def-B-nohat}
 \check B(\rho) := \int_{S_\rho} |h|^2 |\nabla\br|\, ,
 \end{equation*}
as well as the boundary flux
\begin{equation}\label{def-flux-nohat}
\check F(\rho) := \int_{S_\rho} \langle h , \nabla_{-\bN} h \rangle = - \int_{S_\rho} \frac{\langle {h} , \nabla_{\partial_\br}{h} \rangle}{|\nabla \br|}.
\end{equation}

\begin{rk}\label{rem-lemma-Ber-3-2}
Note that the proof of \cite[Lemma 3.2]{Ber-Asym-Struct} depends only on the first variation formula and the weak conical hypotheses, so as in Lemma \ref{lemma-Ber-3-2} we have
$$ \check B'(\rho) = \frac{n-1}{\rho} \check B(\rho) - 2 \check F(\rho) + \check B(\rho) O(\rho^{-3})\, .$$
\end{rk}

\begin{lemma}{\rm (\cite[Lemma 9.2]{Ber-Asym-Struct})}\label{lemm-9-2-Ber}
Let $\alpha\in \RR$. Then there is a radius $R$ such that for any $r_2\geq r_1\geq R$ and $h\in C^2(A_{r_1,r_2})$,
\begin{equation*}
\int_{S_{r_1}}\br^{-\alpha-1}|\nabla\br||h|^2\Psi_{m}+\int_{A_{r_1,r_2}}\br^{-\alpha}|h|^2\Psi_{m}\leq \frac{C_{\alpha}}{r_1^{\alpha+2}}\check{D}_m(h,r_1,r_2)+C_{\alpha}\int_{S_{r_2}}\br^{-\alpha-1}|\nabla\br||h|^2\Psi_{m}.
\end{equation*}
\end{lemma}

\begin{proof}
The proof consists in applying the divergence theorem to the vector field $$\mathbf{V}:=\br^{-\alpha-1}|h|^2\Psi_m\partial_{\br}.$$
Indeed, by \eqref{eq:6} and \eqref{eq:div-rad},
\begin{equation*}
\div \mathbf{V}=\frac{n+m-\alpha-2}{\br^{\alpha+2}}|h|^2\Psi_m+\frac{2}{\br^{\alpha+1}}\langle h,\nabla_{\partial_{\br}}h\rangle\Psi_m+\frac{1}{2\br^{\alpha}}|h|^2\Psi_m+|h|^2\Psi_mO(\br^{-4-\alpha}).
\end{equation*}
The use of Young's inequality then leads to the desired estimate.
\end{proof}

The following lemma gives a priori integral estimates on the Bianchi gauge both with respect to polynomial and exponential weights:
\begin{lemma}\label{lemma-bianchi-bis}
Given $m\in\RR$, there is a radius $R=R(n,m)$ such that if $r\geq R$ and $(h,\calB)$ satisfies \eqref{eq:main.1}--\eqref{est-remain-term}, then for all $r_2>r_1\geq R$ and $\varepsilon\in(0,1)$,
\begin{equation*}
\begin{split}
\int_{A_{r_1,r_2}}\Big((1-\varepsilon)\frac{\br^2}{4}&+\frac{n+m+2}{2}+O(\br^{-2})\Big)|\calB|^2\Psi_m+\frac{1}{2}\int_{S_{r_1}}\br|\nabla\br||\calB|^2\Psi_m\\
&\leq\frac{1}{2}\int_{S_{r_2}}\br|\nabla\br||\calB|^2\Psi_m+C_{\varepsilon}\int_{A_{r_1,r_2}}\left(\br^{-8}|h|^2+\br^{-6}|\nabla h|^2\right)\Psi_m,\label{Bianchi-bis-exp-wei}
\end{split}
\end{equation*}
and,
\begin{equation}
\begin{split}
\int_{A_{r_1,r_2}}\Big(1+\frac{n-\alpha}{2}&+O(\br^{-2})\Big)|\calB|^2\br^{-\alpha}+\frac{1}{2}\int_{S_{r_1}}|\nabla\br||\calB|^2\br^{1-\alpha}\\
&\leq\frac{1}{2}\int_{S_{r_2}}|\nabla\br||\calB|^2\br^{1-\alpha}+C\int_{A_{r_1,r_2}}\left(\br^{-4}|h|^2+\br^{-2}|\nabla h|^2\right)\br^{-\alpha}.\label{Bianchi-bis-pol-wei}
\end{split}
\end{equation}

\end{lemma}

\begin{proof}
By integration by parts:
\begin{equation*}
\begin{split}
\int_{A_{r_1,r_2}}\left\langle\nabla\left(\frac{\br^2}{4}\right),\nabla|\calB|^2\right\rangle\Psi_m&=-\int_{A_{r_1,r_2}}\div\left(\Psi_m\nabla\left(\frac{\br^2}{4}\right)\right)|\calB|^2+\frac{1}{2}\int_{S_r}\br|\calB|^2\langle\partial_{\br},\bN\rangle\Psi_m\Bigg|_{r_1}^{r_2}\\
&=-\int_{A_{r_1,r_2}}\left(\left\langle\nabla\left(\frac{\br^2}{4}\right),\nabla\left(\frac{\br^2}{4}+m\ln\br\right)\right\rangle+\frac{\Delta\br^2}{4}\right)|\calB|^2\Psi_m\\
&\quad+\frac{1}{2}\int_{S_r}\br|\calB|^2\langle\partial_{\br},\bN\rangle\Psi_m\Bigg|_{r_1}^{r_2}.
\end{split}
\end{equation*}
Using (\ref{eq:main.2}) now leads to:
\begin{equation*}
\begin{split}
\int_{A_{r_1,r_2}}\Bigg(\frac{\br^2}{4}+\frac{n+m+2}{2} &+O(\br^{-2})\Bigg) |\calB|^2\Psi_m+\frac{1}{2}\int_{S_{r_1}}\br|\nabla\br||\calB|^2\Psi_m\\
&\leq \frac{1}{2}\int_{S_{r_2}}\br|\nabla\br||\calB|^2\Psi_m+C\int_{A_{r_1,r_2}}\left(\br^{-3}|h|+\br^{-2}|\nabla h|\right)|\calB|\Psi_m\\
&\leq\frac{1}{2}\int_{S_{r_2}}\br|\nabla\br||\calB|^2\Psi_m+C_{\varepsilon}\int_{A_{r_1,r_2}}\left(\br^{-8}|h|^2+\br^{-6}|\nabla h|^2\right)\Psi_m\\
&\quad+\varepsilon\int_{A_{r_1,r_2}}\frac{\br^2}{4}|\calB|^2\Psi_m,
\end{split}
\end{equation*}
where we have used Young's inequality in the last line. This shows the desired inequality (\ref{Bianchi-bis-exp-wei}).

As for the inequality (\ref{Bianchi-bis-pol-wei}) with polynomial weight, we proceed similarly by applying the divergence theorem to the vector field $\mathbf{V}:=|\calB|^2\br^{-\alpha}\nabla f$.
\end{proof}

\begin{lemma}{\rm (\cite[Lemma 9.3]{Ber-Asym-Struct})}\label{lemm-9-3-Ber}
Given $m\in\RR$ and $\varepsilon\in(0,1)$, there is a radius $R=R(n,m,\varepsilon)$ such that if $(h,\calB)$ satisfies \eqref{eq:main.1}--\eqref{est-remain-term} on $E_R$, then for all $r_2>r_1\geq R$,
\begin{equation*}
\Psi_m(r_1)\check F(r_1)-\Psi_m(r_2)\check F(r_2)\geq -\varepsilon\frac{\Psi_m(r_2)}{r_2}\check B(r_2)-C_{\varepsilon}\int_{S_{r_2}}\br|\calB|^2\Psi_m.
\end{equation*}

\end{lemma}

 \begin{proof}
 Using the definition of the flux $F$ given in \eqref{def-flux-nohat}, integration by parts shows that:
 \begin{equation*}
\begin{split}
\Psi_m(r_1)\check F(r_1)-\Psi_m(r_2)\check F(r_2)&=\int_{A_{r_1,r_2}}|\nabla h|^2\Psi_m+\int_{A_{r_1,r_2}}\langle h,\calL_m^{+}h\rangle\Psi_m\\
&=\check{D}_{m}(h,r_1,r_2)+\int_{A_{r_1,r_2}}\left\langle h,\Li_{\calB}(g)+R[h]+\frac{m}{\br}\nabla_{\partial_{\br}}h\right\rangle\Psi_m\\
&=\check{D}_{m}(h,r_1,r_2)+\int_{A_{r_1,r_2}}\left\langle h,R[h]\right\rangle\Psi_m-2\langle\div(\Psi_mh),\calB\rangle\\
&\quad+2\Bigg[\int_{S_r}h(\bN,\calB)\Psi_m\Bigg]_{r_1}^{r_2}+\int_{A_{r_1,r_2}}\frac{m}{\br}\langle h,\nabla_{\partial_{\br}}h\rangle\Psi_m.
\end{split}
\end{equation*}
Now, observe thanks to [\eqref{est-remain-term}, Proposition \ref{system-diff-exp-bis}] that for $\varepsilon>0$ and $\br\geq R$ sufficiently large,
\begin{equation*}
\begin{split}
|\div(\Psi_mh)||\calB|&\leq \frac{\varepsilon}{4}\left(|h|^2+\frac{|\nabla h|^2}{\br^2}\right)\Psi_m+C_{\varepsilon}\br^2|\calB|^2\Psi_m,\\
\left|\left\langle h,R[h]\right\rangle\right|\Psi_m&\leq \frac{C}{\br^{2}}\left(|h|^2+\br^{-1}|h||\nabla h|\right)\Psi_m\leq \frac{\varepsilon}{4}\left(|h|^2+\frac{|\nabla h|^2}{\br^2}\right)\Psi_m.
\end{split}
\end{equation*}
In particular, we get:
\begin{equation*}
\begin{split}
\Psi_m(r_1)\check F(r_1)-\Psi_m(r_2) \check F(r_2)&\geq\check{D}_{m}(h,r_1,r_2)-\varepsilon\int_{A_{r_1,r_2}}\left(|h|^2+\frac{|\nabla h|^2}{\br^2}\right)\Psi_m\\
&\quad-C_{\varepsilon}\int_{A_{r_1,r_2}}\br^2|\calB|^2\Psi_m-2\int_{S_{r_1}}|h||\calB|\Psi_m-2\int_{S_{r_2}}|h||\calB|\Psi_m\\
&\quad -C\int_{A_{r_1,r_2}}\frac{|h|^2+|\nabla h|^2}{\br}\Psi_m.
\end{split}
\end{equation*}
According to Lemma \ref{lemm-9-2-Ber} with $\alpha=1$, we get in case $r_1\geq R$ is sufficiently large,
\begin{align*}
\Psi_m(r_1)\check F(r_1) &- \Psi_m(r_2)\check F(r_2)\geq \left(1-\frac{C}{r_1}-\frac{\varepsilon}{r_1^2}-\frac{C}{r_1^3}\right)\check{D}_{m}(h,r_1,r_2)-\varepsilon\int_{A_{r_1,r_2}}|h|^2\Psi_m\\
&\quad-C_{\varepsilon}\int_{A_{r_1,r_2}}\br^2|\calB|^2\Psi_m-2\int_{S_{r_1}}|h||\calB|\Psi_m-2\int_{S_{r_2}}|h||\calB|\Psi_m\\
&\quad-\frac{C}{r_2^2}\Psi_m(r_2)\int_{S_{r_2}}|\nabla\br||h|^2\\
&\geq \left(1-\frac{C}{r_1}\right)\check{D}_{m}(h,r_1,r_2)-\varepsilon\int_{A_{r_1,r_2}}|h|^2\Psi_m\\
&\quad-C_{\varepsilon}\left(\int_{S_{r_2}}\br|\calB|^2\Psi_m+\int_{A_{r_1,r_2}}\left(\br^{-8}|h|^2+\br^{-6}|\nabla h|^2\right)\Psi_m\right)\\
&\quad-\int_{S_{r_1}}\left(\varepsilon\br^{-1}|h|^2+C_{\varepsilon}\br|\calB|^2\right)\Psi_m-\int_{S_{r_2}}\left(\varepsilon\br^{-1}|h|^2+C_{\varepsilon}\br|\calB|^2\right)\Psi_m\\
&\quad-\frac{C}{r_2^2}\Psi_m(r_2)\int_{S_{r_2}}|\nabla\br||h|^2\\
&\geq \left(1-\frac{C}{r_1}\right)\check{D}_{m}(h,r_1,r_2)-\varepsilon\int_{A_{r_1,r_2}}|h|^2\Psi_m\\
&\quad-C_{\varepsilon}\int_{S_{r_2}}\br|\calB|^2\Psi_m-\varepsilon\int_{S_{r_1}}\br^{-1}|h|^2\Psi_m-\frac{C\varepsilon}{r_2}\Psi_m(r_2)\check B(r_2).
\end{align*}
where we have used Lemma \ref{lemma-bianchi-bis} in the second and third inequalities. 
Using Lemma \ref{lemm-9-2-Ber} with $\alpha=0$, we finally establish:
\begin{equation*}
\begin{split}
\Psi_m(r_1)\check F(r_1)-\Psi_m(r_2)\check F(r_2)&\geq \left(1-\frac{C}{r_1}-\frac{C}{r_1^2}\right)\check{D}_{m}(h,r_1,r_2)\\
&\quad-C_{\varepsilon}\int_{S_{r_2}}\br|\calB|^2\Psi_m-\frac{C\varepsilon}{r_2}\Psi_m(r_2)\check B(r_2)\\
&\geq-C_{\varepsilon}\int_{S_{r_2}}\br|\calB|^2\Psi_m-\frac{C\varepsilon}{r_2}\Psi_m(r_2)\check B(r_2),
\end{split}
\end{equation*}
if $r_1\geq R$ is sufficiently large.
 \end{proof}
 
 The next lemma derives a differential inequality satisfied by 
 \begin{equation*}
 B_{\calB}(\rho):=\int_{S_{\rho}}|\calB|^2|\nabla\br|,\quad\rho\geq R,
 \end{equation*}
 defined for $R$ large enough.
 
 \begin{lemma}\label{lemma-bianchi-level-set-var}
 There exists a radius $R>0$ such that for $\rho\geq R$,
 \begin{equation*}
\begin{split}
B_{\calB}'(\rho)&=\frac{n+1}{\rho}B_{\calB}(\rho)+O(\rho^{-3})B_{\calB}(\rho)+\bigg(\int_{S_{\rho}}O(\br^{-8})|h|^2+O(\br^{-6})|\nabla h|^2\bigg)^{\frac{1}{2}}B_{\calB}^{\frac{1}{2}}(\rho)\\
&=\frac{n+1}{\rho}B_{\calB}(\rho)+O(\rho^{-2})B_{\calB}(\rho)+O(\rho^{-6})\check B(\rho)+O(\rho^{-4})\int_{S_{\rho}}|\nabla h|^2.
\end{split}
\end{equation*}
 \end{lemma}
 
 \begin{proof}
 According to the first variation formula and \eqref{eq:main.2},
 \begin{equation*}
\begin{split}
B_{\calB}'(\rho)&=\int_{S_{\rho}}|\calB|^2|\nabla\br|(\br^{-1}\langle\bX,\bN\rangle H_{S_{\rho}}+2\langle \br^{-1}\nabla_{\bX}\calB,\calB\rangle|\nabla\br|+|\calB|^2\br^{-1}\bX\cdot|\nabla\br|\\
&=\frac{n-1}{\rho}B_{\calB}(\rho)+2\int_{S_{\rho}}\br^{-1}|\calB|^2|\nabla\br|\\
&\quad+\int_{S_{\rho}}\left(O(\br^{-4})|h|+O(\br^{-3})|\nabla h|\right)|\calB|+\int_{S_{\rho}}O(\br^{-3})|\calB|^2\\
&=\frac{n+1}{\rho}B_{\calB}(\rho)+O(\rho^{-3})B_{\calB}(\rho)+\bigg(\int_{S_{\rho}}O(\br^{-8})|h|^2+O(\br^{-6})|\nabla h|^2\bigg)^{\frac{1}{2}}B_{\calB}^{\frac{1}{2}}(\rho).
\end{split}
\end{equation*}
This proves the expected estimate by the Cauchy-Schwarz inequality.
 \end{proof}
 The main result of this section is the following exponential integral decay of a couple $(h,\calB)$ that satisfies \eqref{eq:main.1}--\eqref{est-remain-term} under the sole assumption of an a priori polynomial decay.
 
\begin{theo}\label{theo-dec-same-cone-ene-spa}
Let $(h,\calB)$ satisfy \eqref{eq:main.1}--\eqref{est-remain-term}. Assume 
\begin{equation*}
\fint_{S_{\rho}}\left(|h|^2+\rho^2|\calB|^2\right)=O(\rho^{-\varepsilon}),\label{ass-prim-dec-h-B}
\end{equation*}
 for some $\varepsilon>1$. Then for all $m\in\mathbb{R}$, there exists $R>0$ and $\beta>0$ such that for all $\rho\geq R$,
 \begin{equation*}
 \int_{S_{\rho}}|h|^2\leq C \Phi_{m}(\rho),\quad \rho^2\int_{S_{\rho}}|\calB|^2\leq C \Phi_{m}(\rho)+C\rho^{n-1}e^{-\frac{\beta\rho^3}{3}}\int_{A_{R,\rho}}|\nabla h|^2e^{\frac{\beta\br^3}{3}}\br^{-n+1},
 \end{equation*}
 for some positive constant $C=C(n,m,h)$.
\end{theo}
\begin{rk}
Ideally one would impose $\fint_{S_{\rho}}|h|^2+\rho^2|\calB|^2=o(1)$ only to reach the same conclusions of Theorem \ref{theo-dec-same-cone-ene-spa}. Nonetheless, assumption \eqref{ass-prim-dec-h-B} will be satisfied in our main applications.
\end{rk}
 
 Before proving Theorem \ref{theo-dec-same-cone-ene-spa}, we ensure the integrability of the energy of a solution $(h,\calB)$ satisfying \eqref{eq:main.1}--\eqref{est-remain-term} in polynomial weighted $L^2$ spaces:
 \begin{lemma}\label{lemma-finite-h-bianchi-nabla-h}
 Let $(h,\calB)\in C^2(E_R)$ satisfy \eqref{eq:main.1}--\eqref{est-remain-term}. Assume 
\begin{equation*}
\int_{E_R}\left(|h|^2+|\calB|^2\right)\br^{-\alpha}<+\infty,\label{ass-int-h-B}
\end{equation*}
for some $\alpha\in \RR$. Then, 
\begin{equation}
\int_{E_R}|\nabla h|^2\br^{-\alpha}<+\infty.\label{ene-est-a-priori}
\end{equation}
Moreover, for $\gamma\in(0,1)$, $\alpha\in\RR$, there exists $R$ such that if $\rho\geq R$,
\begin{equation}
\begin{split}
\int_{S_{\rho}}\br^{-\alpha}\langle h,\nabla_{\bN}h\rangle+(1-\gamma)&\int_{E_{\rho}}|\nabla h|^2\br^{-\alpha}+\frac{n-\alpha}{4}\int_{E_{\rho}}|h|^2\br^{-\alpha}\\
&\leq C_{\gamma}\int_{E_{\rho}}|h|^2\br^{-\alpha-2}+C_{\gamma}\int_{E_{\rho}}|\calB|^2\br^{-\alpha}+C\rho^{-\alpha} B_{\calB}(\rho).\label{ene-est-a-priori-bis}
\end{split}
\end{equation}
 \end{lemma}
 \begin{proof}
 To do so, let $\rho\geq 4\rho_0\geq 8R$ and let $\eta_{\rho_0,\rho}$ be a smooth cut-off function such that $0\leq\eta_{\rho_0,\rho}\leq 1$, $\eta_{\rho_0,\rho}\equiv 1$ on $A_{\frac{\rho_0}{2},\rho}$, $\supp \eta_{\rho_0,\rho}\subset A_{\frac{\rho_0}{2},2\rho}$ and $\rho^k|\nabla^k\eta_{\rho_0,\rho}|\leq C$ on $A_{\rho,2\rho}$ for $k=1,2$. Observe by the divergence theorem applied to the vector field $$\mathbf{V}:=\br^{-\alpha}\left(|h|^2\nabla\eta^2_{\rho_0,\rho}-\eta^2_{\rho_0,\rho}\nabla|h|^2\right),$$
 that one gets:
 \begin{equation*}
\begin{split}
2\int_{S_{\rho_0}}\br^{-\alpha}\langle h,\nabla_{\bN}h\rangle&=\int_{E_{\rho_0}}\div\left(\br^{-\alpha}\left(|h|^2\nabla\eta^2_{\rho_0,\rho}-\eta^2_{\rho_0,\rho}\nabla|h|^2\right)\right)\\
&=\int_{E_{\rho_0}}\br^{-\alpha}|h|^2\Delta\eta^2_{\rho_0,\rho}-\br^{-\alpha}\eta^2_{\rho_0,\rho}\Delta|h|^2+\langle\nabla\br^{-\alpha},\nabla\eta^2_{\rho_0,\rho}\rangle|h|^2\\
& \quad -\langle\nabla\br^{-\alpha},\nabla|h|^2\rangle\eta^2_{\rho_0,\rho}.
\end{split}
\end{equation*}
In particular, by Young's inequality, for $\gamma\in(0,1)$,
\begin{equation}
\begin{split}\label{inequ-1st-IBP}
\int_{S_{\rho_0}}\br^{-\alpha}\langle h,\nabla_{\bN}h\rangle+(1-\gamma)\int_{E_{\rho_0}}\eta^2_{\rho_0,\rho}|\nabla h|^2\br^{-\alpha}
&\leq C_{\gamma}\int_{E_{\rho_0}}\br^{-\alpha-2}|h|^2\\
&\quad -\int_{E_{\rho_0}}\langle\Delta h,h\rangle \eta_{\rho_0,\rho}^2\br^{-\alpha},
\end{split}
\end{equation}
for some positive constant $C$ independent of $\rho$ that may vary from line to line. Here we have used the assumption on the integrability of $|h|^2\br^{-\alpha}$. Now, by using (\ref{eq:main.1}) and integration by parts,
\begin{equation*}
\begin{split}
-\int_{E_{\rho_0}}\!\!\langle\Delta h,h\rangle \eta_{\rho_0,\rho}^2\br^{-\alpha}&=\frac{1}{2}\int_{E_{\rho_0}}\langle\nabla f,\nabla|h|^2\rangle\eta_{\rho_0,\rho}^2\br^{-\alpha}+\int_{E_{\rho_0}}\langle R[h],h\rangle\eta_{\rho_0,\rho}^2\br^{-\alpha}\\
&\quad+2\int_{E_{\rho_0}}\langle\div\left(\br^{-\alpha}\eta_{\rho_0,\rho}^2 h),\calB\right\rangle+2\int_{S_{\rho_0}}\br^{-\alpha}h(\calB,\bN)\eta_{\rho_0,\rho}^2\\
&=-\frac{1}{2}\int_{E_{\rho_0}}\div\left(\br^{-\alpha}\eta_{\rho_0,\rho}^2
\nabla f\right)|h|^2+\int_{E_{\rho_0}}\langle R[h],h\rangle\eta_{\rho_0,\rho}^2\br^{-\alpha}\\
&\quad+2\int_{E_{\rho_0}}\!\!\!\langle\div\left(\br^{-\alpha}\eta_{\rho_0,\rho}^2 h),\calB\right\rangle+\int_{S_{\rho_0}}\!\!\!\Big(2h(\calB,\bN)-\frac{1}{2}\langle\nabla f,\bN\rangle|h|^2\Big)\eta_{\rho_0,\rho}^2\br^{-\alpha}.
\end{split}
\end{equation*}
Thanks to [\eqref{est-remain-term}, Proposition \ref{system-diff-exp-bis}], observe that:
\begin{equation}
\begin{split}\label{inequ-2nd-IBP}
\int_{E_{\rho_0}}\left|\langle R[h],h\rangle\right|\eta_{\rho_0,\rho}^2\br^{-\alpha}&\leq C\rho_0^{-2}\int_{E_{\rho_0}}|\nabla h|^2\eta_{\rho_0,\rho}^2\br^{-\alpha}+C\int_{E_{\rho_0}}|h|^2\eta_{\rho_0,\rho}^2\br^{-\alpha-2},\\
-\int_{E_{\rho_0}}\div\left(\br^{-\alpha}\eta_{\rho_0,\rho}^2\nabla f\right)|h|^2&=\int_{E_{\rho_0}}\left(\frac{\alpha-n}{2}+O(\br^{-2})\right)\eta_{\rho_0,\rho}^2|h|^2\br^{-\alpha}\\
&\quad -\int_{E_{\rho_0}}\langle\nabla\eta_{\rho_0,\rho}^2,\nabla f\rangle|h|^2\br^{-\alpha},\\
\int_{E_{\rho_0}}\langle\div\left(\br^{-\alpha}\eta_{\rho_0,\rho}^2 h),\calB\right\rangle&\leq \gamma\int_{E_{\rho_0}}\left(|\nabla h|^2+\br^{-2}|h|^2\right)\eta_{\rho_0,\rho}^2\br^{-\alpha}\\
&\quad +C_{\gamma}\int_{E_{\rho_0}}|\calB|^2\eta_{\rho_0,\rho}^2\br^{-\alpha} +C\int_{E_{\rho_0}}|h|^2|\nabla \eta_{\rho_0,\rho}|^2\br^{-\alpha},
\end{split}
\end{equation}
for any $\gamma>0$ and,
\begin{equation}
\int_{S_{\rho_0}}\Big(2h(\calB,\bN)-\frac{1}{2}\langle\nabla f,\bN\rangle|h|^2\Big)\eta_{\rho_0,\rho}^2\br^{-\alpha}\leq C\rho_0^{-\alpha} B_{\calB}(\rho_0)\, .\label{inequ-bdy-B}
\end{equation}
In particular, if $\gamma$ is chosen sufficiently small and $R$ is sufficiently large, injecting the previous estimates \eqref{inequ-2nd-IBP} and \eqref{inequ-bdy-B} back to \eqref{inequ-1st-IBP} leads to:
\begin{equation}
\begin{split}\label{never-ends-inequ-poly-spa}
\int_{S_{\rho_0}}\br^{-\alpha}\langle h,\nabla_{\bN}h\rangle&+(1-\gamma)\int_{E_{\rho_0}}\eta^2_{\rho_0,\rho}|\nabla h|^2\br^{-\alpha}+\frac{n-\alpha}{4}\int_{E_{\rho_0}}\eta_{\rho_0,\rho}^2|h|^2\br^{-\alpha}\\
&\leq C\int_{E_{\rho_0}}|h|^2\left(1+\br^2|\nabla \eta_{\rho_0,\rho}|^2\right)\br^{-\alpha-2}+C_{\gamma}\int_{E_{\rho_0}}|\calB|^2\eta_{\rho_0,\rho}^2\br^{-\alpha} \\
&\quad+C\int_{A_{\rho,2\rho}}|h|^2\br^{-\alpha}+C\rho_0^{-\alpha}B_{\calB}(\rho_0)\\
&\leq C\int_{A_{\rho,2\rho}}|h|^2\br^{-\alpha}+ C_{\gamma}\int_{E_{\rho_0}}|h|^2\br^{-\alpha-2}+C_{\gamma}\int_{E_{\rho_0}}|\calB|^2\eta_{\rho_0,\rho}^2\br^{-\alpha}\\
&\quad+C\rho_0^{-\alpha} B_{\calB}(\rho_0).
\end{split}
\end{equation}
Here we have used that $\nabla\eta_{\rho_0,\rho}=O(\br^{-1})$ and the fact that the integral on $A_{\rho,2\rho}$ is bounded independently of $\rho\geq \rho_0$ by assumption. This ends the proof of \eqref{ene-est-a-priori}. 

In order to prove \eqref{ene-est-a-priori-bis}, we let $\rho$ go to $+\infty$ in \eqref{never-ends-inequ-poly-spa}.
\end{proof}
 \begin{proof}[Proof of Theorem \ref{theo-dec-same-cone-ene-spa}]
 On the one hand, Remark \ref{rem-lemma-Ber-3-2} and Lemma \ref{lemm-9-3-Ber} applied to $R\leq r_1\leq r_2:=\rho$ gives for $\varepsilon >0$,
 \begin{align*}
\check B'(\rho)&=\frac{n-1}{\rho}\check B(\rho)-2 \check F(\rho)+O(\rho^{-3})\check B(\rho)\\
&=\frac{n-1}{\rho}\check B(\rho)-2\Phi_{-m}(\rho)\Psi_m(r_1)\check F(r_1)+O(\rho^{-3})\check B(\rho)\\
&\quad+2\Phi_{-m}(\rho)\left(\Psi_m(r_1)\check F(r_1)-\Psi_m(\rho)\check F(\rho)\right)\\
&\geq \frac{n-1}{\rho}\check B(\rho)-2\Phi_{-m}(\rho)\Psi_m(r_1)\check F(r_1)+O(\rho^{-3})\check B(\rho)\\
&\quad-\varepsilon\frac{\Phi_{-m}(\rho)\Psi_m(\rho)}{\rho}\check B(\rho)-C_{\varepsilon}\Phi_{-m}(\rho)\int_{S_{\rho}}\br|\calB|^2\Psi_m\\
&=\frac{n-1- \varepsilon}{\rho}\check B(\rho)-2\Phi_{-m}(\rho)\Psi_m(r_1)\check F(r_1)+O(\rho^{-3})\check B(\rho)-C_{\varepsilon}\int_{S_{\rho}}\br|\calB|^2.
\end{align*}
In particular, there is a positive constant $A_0$ such that:
\begin{equation}
\left(\frac{e^{-A_0\rho^{-2}} \check B}{\rho^{n-1-\varepsilon}}\right)'(\rho)\geq \frac{A_0}{\rho^3}\frac{e^{-A_0\rho^{-2}} \check B(\rho)}{\rho^{n-1-\varepsilon}} -C(h,r_1)\Phi_{-m-(n-1-\varepsilon)}(\rho)-\frac{C_{\varepsilon}}{\rho^{n-2-\varepsilon}}B_{\calB}(\rho).\label{ratio-der-B-eps}
\end{equation}
On the other hand, observe that Lemma \ref{lemma-bianchi-level-set-var} ensures that
 \begin{equation*}
\begin{split}
\left(\frac{\rho^2B_{\calB}}{\rho^{n-1-\varepsilon}}\right)'(\rho)&=\frac{n+1-(n-3-\varepsilon)}{\rho}\left(\frac{\rho^2B_{\calB}(\rho)}{\rho^{n-1-\varepsilon}}\right)+O(\rho^{-2})\left(\frac{\rho^2B_{\calB}(\rho)}{\rho^{n-1-\varepsilon}}\right)\\
&\quad+O(\rho^{-4})\frac{\check B(\rho)}{\rho^{n-1-\varepsilon}}+O(\rho^{-2})\int_{S_{\rho}}\br^{-(n-1-\varepsilon)}|\nabla h|^2\\
&=(4+\varepsilon)\frac{B_{\calB}(\rho)}{\rho^{n-2-\varepsilon}}+O(\rho^{-2})\left(\frac{\rho^2B_{\calB}(\rho)}{\rho^{n-1-\varepsilon}}\right)\\
&\quad+O(\rho^{-4})\frac{\check B(\rho)}{\rho^{n-1-\varepsilon}}+O(\rho^{-2})\int_{S_{\rho}}\br^{-(n-1-\varepsilon)}|\nabla h|^2.
\end{split}
\end{equation*}
This in turn implies that there is a positive constant $A_1$ such that:
 \begin{equation}
\begin{split}
\left(\frac{e^{-A_1\rho^{-1}}\rho^2B_{\calB}}{\rho^{n-1-\varepsilon}}\right)'(\rho)&\geq e^{-A_1\rho^{-1}}(4+\varepsilon)\frac{B_{\calB}(\rho)}{\rho^{n-2-\varepsilon}}\\
&\quad-C\rho^{-4}\frac{\check B(\rho)}{\rho^{n-1-\varepsilon}}-\frac{C}{\rho^2}\int_{S_{\rho}}\br^{-(n-1-\varepsilon)}|\nabla h|^2.\label{first-abs-bianchi}
\end{split}
\end{equation}
 Therefore, combining (\ref{ratio-der-B-eps}) with (\ref{first-abs-bianchi}) gives for any $A_2>0$:
 \begin{equation*}
\begin{split}
\left(\frac{e^{-A_0\rho^{-2}}\check B}{\rho^{n-1-\varepsilon}}+A_2\frac{e^{-A_1\rho^{-1}}\rho^2B_{\calB}}{\rho^{n-1-\varepsilon}}\right)'(\rho)&\geq  -C(h,r_1)\Phi_{-m-(n-1-\varepsilon)}(\rho)-\frac{A_2C}{\rho^2}\int_{S_{\rho}}\br^{-(n-1-\varepsilon)}|\nabla h|^2\\
&+\left(A_2e^{-A_1\rho^{-1}}(4+\varepsilon)-C_{\varepsilon}\right)\frac{B_{\calB}(\rho)}{\rho^{n-2-\varepsilon}}\\
&\quad +\left(A_0e^{-A_0\rho^{-2}}\rho^{-3}-A_2C\rho^{-4}\right)\frac{\check B(\rho)}{\rho^{n-1-\varepsilon}}.
\end{split}
\end{equation*}
Now choose $R$ such that if $\rho\geq R$ then $e^{-A_1\rho^{-1}}\geq \frac{1}{2}$ and $A_2$ so large that $A_2\geq C_{\varepsilon}$. Increasing $A_0$ so that $A_0e^{-A_0\rho^{-2}}\rho^{-3}\geq A_2C\rho^{-4}$ on this region implies then:
\begin{equation}
\begin{split}
\bigg(\frac{e^{-A_0\rho^{-2}}\check B}{\rho^{n-1-\varepsilon}}+A_2\frac{e^{-A_1\rho^{-1}}\rho^2B_{\calB}}{\rho^{n-1-\varepsilon}}\bigg)'(\rho) &\geq -C(h,r_1)\Phi_{-m-(n-1-\varepsilon)}(\rho)\\
&\quad -\frac{A_2C}{\rho^2}\int_{S_{\rho}}\br^{-(n-1-\varepsilon)}|\nabla h|^2.\label{inequ-diff-sum-bdy-fct}
\end{split}
\end{equation}
By assumption, there exists $\varepsilon>0$ such that $$\lim_{\rho\rightarrow+\infty}\frac{e^{-A_0\rho^{-2}}\check B(\rho)}{\rho^{n-1-\varepsilon}}+A_2\frac{e^{-A_1\rho^{-1}}\rho^2B_{\calB}(\rho)}{\rho^{n-1-\varepsilon}}=0.$$
 Therefore, by integrating (\ref{inequ-diff-sum-bdy-fct}) between $\rho$ and $+\infty$ and using the co-area formula, one gets as an intermediate result:
 \begin{equation}
 \begin{split}\label{monst-ineq-prim}
\frac{e^{-A_0\rho^{-2}}\check B(\rho)}{\rho^{n-1-\varepsilon}}+A_2\frac{e^{-A_1\rho^{-1}}\rho^2B_{\calB}(\rho)}{\rho^{n-1-\varepsilon}}&\leq C(h,r_1)\int_{\rho}^{+\infty}\Phi_{-m-(n-1-\varepsilon)}(s)\,ds\\
&\quad +A_2C\int_{E_{\rho}}\br^{-(n+1-\varepsilon)}|\nabla h|^2\\
&\leq C(h,r_1)\frac{\Phi_{-m-1}(\rho)}{\rho^{n-1-\varepsilon}}+A_2C\int_{E_{\rho}}\br^{-(n+1-\varepsilon)}|\nabla h|^2.
\end{split}
\end{equation}
Notice that the integral on the righthand side is well-defined thanks to Lemma \ref{lemma-finite-h-bianchi-nabla-h} applied to $\alpha:=n+1-\varepsilon$ since by assumption (\ref{ass-prim-dec-h-B}) and the co-area formula,
\begin{equation*}
\int_{E_R}\left(|h|^2+|\calB|^2\right)\br^{-(n+1-\varepsilon)}= \int_{R}^{+\infty}\left(\int_{S_{\rho}}\frac{|h|^2+|\calB|^2}{|\nabla\br|}\right)\,\rho^{-(n+1-\varepsilon)}\,d\rho\leq C\int_{R}^{+\infty}\rho^{-2}\,d\rho<+\infty.
\end{equation*}
 Here, we have not used the whole strength of the assumption on $\calB$.
 
 According to [\eqref{ene-est-a-priori-bis}, Lemma \ref{lemma-finite-h-bianchi-nabla-h}] applied to $\alpha:=n+1-\varepsilon$ and $\gamma=\frac{1}{2}$,
  \begin{equation}
 \begin{split}\label{monst-ineq-3}
\frac{1}{2}\int_{E_{\rho}}\br^{-(n+1-\varepsilon)}|\nabla h|^2&\leq -\frac{1}{\rho^{n+1-\varepsilon}}\int_{S_{\rho}}\langle h,\nabla_{\bN}h\rangle+\int_{E_{\rho}}\left(\frac{1-\varepsilon}{4}+C\br^{-2}\right)\br^{-(n+1-\varepsilon)}|h|^2\\
&\quad+C\int_{E_{\rho}} \br^{-(n+1-\varepsilon)} |\calB|^2+\frac{C}{\rho^2}\frac{B_{\calB}(\rho)}{\rho^{n-1-\varepsilon}}.
 \end{split}
 \end{equation}
 Let us apply [\eqref{Bianchi-bis-pol-wei}, Lemma \ref{lemma-bianchi-bis}] to $\alpha:=n+1-\varepsilon$ and $r_1:=\rho$ and $r_2=+\infty$, legitimated by assumption (\ref{ass-prim-dec-h-B}), to get:
  \begin{equation}
 \begin{split}\label{monst-ineq-4}
 \int_{E_{\rho}} \br^{-(n+1-\varepsilon)} |\calB|^2&\leq C\int_{E_{\rho}}\left(\br^{-4}|h|^2+\br^{-2}|\nabla h|^2\right)\br^{-(n+1-\varepsilon)}.
  \end{split}
 \end{equation}
Once inequality (\ref{monst-ineq-4}) is injected into (\ref{monst-ineq-3}), one can absorb the integral of $|\nabla h|^2$ on $E_{\rho}$ to get:
\begin{equation}
 \begin{split}\label{monst-ineq-5}
\frac{1}{4}\int_{E_{\rho}}\br^{-(n+1-\varepsilon)}|\nabla h|^2&\leq -\frac{1}{\rho^{n+1-\varepsilon}}\int_{S_{\rho}}\langle h,\nabla_{\bN}h\rangle+\int_{E_{\rho}}\left(\frac{1-\varepsilon}{4}+C\br^{-2}\right)\br^{-(n+1-\varepsilon)}|h|^2\\
&\quad+\frac{C}{\rho^2}\frac{B_{\calB}(\rho)}{\rho^{n-1-\varepsilon}}.
 \end{split}
 \end{equation}
 To conclude, let us plug \eqref{monst-ineq-5} back to (\ref{monst-ineq-prim}) to get for $\rho\geq R$, $R$ being sufficiently large:
 \begin{equation}
 \begin{split}\label{monst-ineq-prim-bis}
\frac{\check B(\rho)}{\rho^{n-1-\varepsilon}}+A_2\frac{\rho^2B_{\calB}(\rho)}{\rho^{n-1-\varepsilon}}&\leq C(h,r_1)\frac{\Phi_{-m-1}(\rho)}{\rho^{n-1-\varepsilon}}-\frac{A_2C}{\rho^{n+1-\varepsilon}}\int_{S_{\rho}}\langle h,\nabla_{\bN}h\rangle\\
&\quad+\frac{A_2C}{2}\int_{E_{\rho}}\left(\frac{1-\varepsilon}{2}+O(\br^{-2})\right)\br^{-(n+1-\varepsilon)}|h|^2.
\end{split}
\end{equation}
 If $\varepsilon$ can be taken to be larger than $1$ then one can simplify (\ref{monst-ineq-prim-bis}) one step further to get:
  \begin{equation}
 \begin{split}\label{monst-ineq-prim-bis-bis}
\frac{\check B(\rho)}{\rho^{n-1}}+A_2\frac{\rho^2B_{\calB}(\rho)}{\rho^{n-1}}&\leq C(h,r_1)\frac{\Phi_{-m-1}(\rho)}{\rho^{n-1}}-\frac{A_2C}{\rho^{2+n-1}}\int_{S_{\rho}}\langle h,\nabla_{\bN}h\rangle.
\end{split}
\end{equation}
 Thanks to Remark \ref{rem-lemma-Ber-3-2}, we observe that 
 \begin{equation*}
 \begin{split}
 \frac{d}{d\rho}\left(e^{C\rho^{-2}}\rho^{-(n-1)}\check B\right) &=\frac{d}{d\rho}\bigg(e^{C\rho^{-2}}\int_{S_{\rho}}\br^{-(n-1)}|h|^2|\nabla\br|\bigg)\\
 &\leq 2 e^{C\rho^{-2}}\rho^{-(n-1)}\int_{S_{\rho}}\langle h,\nabla_{\bN}h\rangle,
 \end{split}
 \end{equation*}
 for some sufficiently large constant $C$ and we conclude with the help of (\ref{monst-ineq-prim-bis-bis}) that:
 \begin{equation*}
\frac{d}{d\rho}\left(e^{C\rho^{-2}}\rho^{-(n-1)}\check B\right)+\frac{\rho^2}{A_2C}e^{C\rho^{-2}}\rho^{-(n-1)}\check B(\rho)+C\frac{\rho^4B_{\calB}(\rho)}{\rho^{n-1}}\leq C(h,r_1)\frac{\Phi_{-m+1}(\rho)}{\rho^{n-1}},
\end{equation*}
where $C$ is a positive constant that may vary from line to line.
Gr\"onwall's inequality gives for $\rho\geq r_1$:
\begin{equation*}
\begin{split}
\rho^{-(n-1)}\check B(\rho)&+Ce^{-\frac{\rho^3}{3A_2C}}\int_{\rho_0}^{\rho}e^{\frac{s^3}{3A_2C}}\frac{s^4B_{\calB}(s)}{s^{n-1}}\,ds\\
&\leq C(h,r_1)e^{-\frac{\rho^3}{3A_2C}}+C(h,r_1)e^{-\frac{\rho^3}{3A_2C}}\int_{r_1}^{\rho}e^{\frac{s^3}{3A_2C}}\Phi_{-m-n+2}(s)\,ds.
\end{split}
\end{equation*}
It can be shown by integration by parts that the right-hand side of the previous inequality is equivalent to $\Phi_{-m-n}(\rho)$ as $\rho$ tends to $+\infty$, in particular, this shows that 
\begin{equation}
\int_{S_{\rho}}|h|^2|\nabla\br|\leq C(h,r_1)\Phi_{-m-1}(\rho),\quad \rho\geq r_1\geq R, \label{first-est-h-int-level}
\end{equation}
which yields the first part of Theorem \ref{theo-dec-same-cone-ene-spa}. We are also left with the following estimate on the Bianchi gauge:
\begin{equation}
\begin{split}
e^{-\frac{\rho^3}{3A_2C}}\int_{\rho_0}^{\rho}e^{\frac{s^3}{3A_2C}}\frac{s^4B_{\calB}(s)}{s^{n-1}}\,ds\leq C(h,r_1)e^{-\frac{\rho^3}{3A_2C}}+C(h,r_1)\Phi_{-m-n}(\rho).\label{first-est-bianchi-int-level}
\end{split}
\end{equation}

\begin{claim}\label{never-ending-inequ-claim}
For $\beta>0$, there exist $C>0$ and $\rho_0>0$ such that if $\rho> 2\rho_0$,
\begin{align*}\label{never-ending-inequ}
C^{-1}\frac{\rho^2B_{\calB}(\rho)}{\rho^{n-1}} &\leq e^{-\frac{\beta\rho^3}{3}}\int_{\rho_0}^{\rho}e^{\frac{\beta s^3}{3}}\frac{s^4B_{\calB}(s)}{s^{n-1}}\,ds+\Phi_{-m-n-6}(\rho)\\
&\quad +e^{-\frac{\beta\rho^3}{3}}\int_{A_{\rho_0,\rho}}|\nabla h|^2e^{\frac{\beta\br^3}{3}}\br^{-n+1}.
\end{align*}
\end{claim}
\begin{proof}[Proof of Claim \ref{never-ending-inequ-claim}]
Let us integrate by parts as follows:
\begin{equation}
\begin{split}\label{est-never-ending-B-B}
\int_{\rho_0}^{\rho}e^{\frac{\beta s^3}{3}}\frac{s^4B_{\calB}(s)}{s^{n-1}}\,ds&=\left(\int_{\rho_0}^{\rho}e^{\frac{\beta s^3}{3}}\,ds\right)\frac{\rho^4B_{\calB}(\rho)}{\rho^{n-1}}-\int_{\rho_0}^{\rho}\left(\int_{\rho_0}^se^{\frac{\beta u^3}{3}}\,du\right)\left(\frac{s^4B_{\calB}}{s^{n-1}}\right)'(s)\,ds\\
&\geq Ce^{\frac{\beta \rho^3}{3}}\frac{\rho^2B_{\calB}(\rho)}{\rho^{n-1}}-\int_{\rho_0}^{\rho}\left(\int_{\rho_0}^se^{\frac{\beta u^3}{3}}\,du\right)\left(\frac{s^4B_{\calB}}{s^{n-1}}\right)'(s)\,ds,
\end{split}
\end{equation}
where used that assumption that $\rho\geq 2 \rho_0$. Now, according to Lemma \ref{lemma-bianchi-level-set-var}, 
\begin{equation}
\begin{split}\label{one-more-time-B-rho-3}
\int_{\rho_0}^{\rho}\left(\int_{\rho_0}^se^{\frac{\beta u^3}{3}}\,du\right)\left(\frac{s^4B_{\calB}}{s^{n-1}}\right)'\!\!\!(s)\,ds&=\int_{\rho_0}^{\rho}\left(\int_{\rho_0}^se^{\frac{\beta u^3}{3}}\,du\right)\left(\frac{6}{s}+O(s^{-2})\right)\frac{s^4B_{\calB}(s)}{s^{n-1}}\,ds\\
&\quad+\int_{\rho_0}^{\rho}\left(\int_{\rho_0}^se^{\frac{\beta u^3}{3}}\,du\right)\bigg(O(s^{-2})\frac{\check B(s)}{s^{n-1}}\\
&\qquad +O(1)s^{-(n-1)}\int_{S_{s}}|\nabla h|^2\bigg)\\
&\leq\frac{C}{\rho_0^3}\int_{\rho_0}^{\rho}e^{\frac{\beta s^3}{3}}\frac{s^4B_{\calB}(s)}{s^{n-1}}\,ds\\
&\quad +C\int_{\rho_0}^{\rho}e^{\frac{\beta s^3}{3}}\Phi_{-m-n-4}(s)\,ds\\
&\quad+C\int_{\rho_0}^{\rho}e^{\frac{\beta s^3}{3}}s^{-n+1}\int_{S_{s}}|\nabla h|^2\,ds.
\end{split}
\end{equation}
Here we have used (\ref{first-est-h-int-level}) in the second inequality. Injecting \eqref{one-more-time-B-rho-3} into \eqref{est-never-ending-B-B} leads to the expected claim once we invoke the co-area formula.
\end{proof}

Thanks to Claim \ref{never-ending-inequ-claim} applied to $\beta:=\frac{1}{A_2C}$, estimate \eqref{first-est-bianchi-int-level} leads to the second part of Theorem \ref{theo-dec-same-cone-ene-spa}.
 \end{proof}
 
 \begin{coro}\label{coro-dont-think-mild-dec-anymore}
 Let $(h,\calB)$ satisfy \eqref{eq:main.1}--\eqref{est-remain-term}. Assume 
\begin{equation}
\fint_{S_{\rho}}\left(|h|^2+\rho^2|\calB|^2\right)=O(\rho^{-\varepsilon}),
\end{equation}
 for some $\varepsilon>1$. Then,
  \begin{equation}
\int_{E_{R}}\left(|h|^2+|\nabla h|^2\right)\,\Psi_{m}<+\infty,\label{good-to-go}
\end{equation}
for all $m\in \RR$. In particular, $\hat{h}\in C^2_{-2n,1}(E_R)$, i.e.
\begin{equation*}
\int_{E_{R}}\left(|\hat{h}|^2+|\nabla \hat{h}|^2\right)\,\Phi_{-2n}<+\infty.\label{good-to-go-bis}
\end{equation*}
 \end{coro}
 \begin{proof}
 According to Theorem \ref{theo-dec-same-cone-ene-spa} applied to $-m-2$ and the co-area formula, one gets half of the expected result, i.e.
 \begin{equation}\begin{split}
\int_{E_{R}}|h|^2\,\Psi_{m}&=\int_R^{+\infty}\int_{S_{\rho}}\frac{|h|^2}{|\nabla\br|}\,\Psi_{m}(\rho)\leq C\int_R^{+\infty}\Phi_{-m-2}(\rho)\Psi_m(\rho)\,d\rho\\
&\leq C\int_R^{+\infty}\rho^{-2}\,d\rho<+\infty.\label{first-vic-finite-h}
\end{split}
\end{equation}
Following Bernstein \cite[Proof of Theorem $9.1$]{Ber-Asym-Struct}, the second part of the estimate \eqref{good-to-go} is obtained with the help of a suitable cut-off function. Let $\eta_{\rho_0,\rho}$, $\rho\geq 4\rho_0\geq 8R$, be a radial smooth cut-off function which is non-increasing on $E_{\rho}$ such that $0\leq\eta_{\rho_0,\rho}\leq 1$, $\eta_{\rho_0,\rho}\equiv 1$ on $A_{2\rho_0,\rho}$, $\supp \eta_{\rho_0,\rho}\subset A_{\rho_0,2\rho}$ and $\rho^k|\nabla^k\eta_{\rho_0,\rho}|\leq C$ on $A_{\rho,2\rho}$ for $k=1,2$. Let us integrate by parts as follows by using (\ref{eq:main.1}):
\begin{equation*}
\begin{split}
\int_{E_{R}}(\calL^+_m\eta_{\rho_0,\rho}^2)|h|^2\,\Psi_m&=\int_{E_{R}}\eta_{\rho_0,\rho}^2\calL_{-2n}^+|h|^2\,\Psi_m\\
&=2\int_{E_{R}}\eta_{\rho_0,\rho}^2|\nabla h|^2\,\Psi_m+2\int_{E_{R}}\eta_{\rho_0,\rho}^2\langle\calL^+_mh,h\rangle\,\Psi_m\\
&=2\int_{E_{R}}\eta_{\rho_0,\rho}^2|\nabla h|^2\,\Psi_m+2\int_{E_{R}}\eta_{\rho_0,\rho}^2\left\langle\calL_{\calB}(g)+\frac{m}{\br}\nabla_{\partial_{\br}}h+R[h],h\right\rangle\,\Psi_m\\
&=2\int_{E_{R}}\eta_{\rho_0,\rho}^2|\nabla h|^2\,\Psi_m-4\int_{E_{R}}\langle\div\left(\Psi_m\eta_{\rho_0,\rho}^2h\right),\calB\rangle\\
&\quad+\int_{E_{R}}\eta_{\rho_0,\rho}^2\left\langle\frac{m}{\br}\nabla_{\partial_{\br}}h+R[h],h\right\rangle\,\Psi_m.
\end{split}
\end{equation*}
Thanks to [\eqref{est-remain-term}, Proposition \ref{system-diff-exp-bis}] and Young's inequality, one gets:
\begin{equation}
\begin{split}
\frac{1}{2}\int_{E_{R}}\eta_{\rho_0,\rho}^2|\nabla h|^2\,\Psi_m&\leq \int_{E_{R}}\left(\calL^+_m\eta_{\rho_0,\rho}^2+\left(\frac{m^2}{\br^2}+1\right)\eta_{\rho_0,\rho}^2+|\nabla\eta_{\rho_0,\rho}|^2\right)|h|^2\,\Psi_m\\
&\quad+C\int_{E_{R}}\br^2\eta_{\rho_0,\rho}^2|\calB|^2\,\Psi_m.\label{est-nabla-h-abso-antipen}
\end{split}
\end{equation}

We are left with estimating the last integral on the righthand side of the previous inequality in terms of the energy of $h$ with weight $\Psi_m$. To do so, we invoke Theorem \ref{theo-dec-same-cone-ene-spa} applied to $-m-2$ to observe with the help of the co-area formula that:
\begin{equation*}
\begin{split}
\int_{E_{R}}\br^2\eta_{\rho_0,\rho}^2|\calB|^2\,\Psi_m&\leq 2 \int_{R}^{+\infty}\eta_{\rho_0,\rho}^2s^2B_{\calB}(s)\Psi_m(s)\,ds\\
&\leq  C\int_{R}^{+\infty}\eta_{\rho_0,\rho}^2 \Phi_{-m-2}(s)\Psi_m(s)\,ds\\
&\quad+C\int_{R}^{+\infty}\eta_{\rho_0,\rho}^2s^{n-1}e^{-\frac{\beta s^3}{3}}\left(\int_{A_{R,s}}|\nabla h|^2e^{\frac{\beta\br^3}{3}}\br^{-n+1}\right)\Psi_m(s)\,ds\\
&\leq C\int_{R}^{+\infty}s^{-2}\,ds+C\int_{E_R}\left(\int_{\br}^{+\infty}s^{n-1}e^{-\frac{\beta s^3}{3}}\Psi_m(s)\eta_{\rho_0,\rho}^2\,ds\right)|\nabla h|^2e^{\frac{\beta\br^3}{3}}\br^{-n+1}.
\end{split}
\end{equation*}
Now, since $\eta_{\rho_0,\rho}$ is decreasing on $E_{\rho}$, one gets by a straightforward integration by parts that for $\br\geq R$:
\begin{equation*}
\begin{split}
\int_{\br}^{+\infty}s^{n-1}e^{-\frac{\beta s^3}{3}}\Psi_m(s)\eta_{\rho_0,\rho}^2\,ds\leq C \eta_{\rho_0,\rho}^2(\br)\br^{n-3}e^{-\frac{\beta \br^3}{3}}\Psi_m(\br)+C\chi_{A_{\rho_0,2\rho_0}}.
\end{split}
\end{equation*}
Therefore,
\begin{equation}
\begin{split}\label{est-nabla-B-abso-antipen}
\int_{E_{R}}\br^2\eta_{\rho_0,\rho}^2|\calB|^2\,\Psi_m
&\leq C(h,\rho_0)+C\int_{E_R}\eta_{\rho_0,\rho}^2|\nabla h|^2\br^{-2}\Psi_m.
\end{split}
\end{equation}
Pulling \eqref{est-nabla-B-abso-antipen} into \eqref{est-nabla-h-abso-antipen} leads to:
\begin{align*}
\frac{1}{2}\int_{E_{R}}\eta_{\rho_0,\rho}^2|\nabla h|^2\,\Psi_m&\leq\int_{E_{R}}\left(\calL^+_m\eta_{\rho_0,\rho}^2+\left(\frac{m^2}{\br^2}+1\right)\eta_{\rho_0,\rho}^2+|\nabla\eta_{\rho_0,\rho}|^2\right)|h|^2\,\Psi_m+C(h,\rho_0)\\
&\quad+C\int_{E_R}\eta_{\rho_0,\rho}^2|\nabla h|^2\br^{-2}\Psi_m\\
&\leq C\int_{E_R}|h|^2\,\Psi_m+\frac{C}{\rho_0^2}\int_{E_R}\eta_{\rho_0,\rho}^2|\nabla h|^2\,\Psi_m+C(h,\rho_0).
\end{align*}
Here we have used \eqref{first-vic-finite-h} in the second inequality. This gives us the integrability of $|\nabla h|^2$ with respect to the weight $\Psi_m$ by choosing $R$ large enough by absorption. 
\end{proof}
 
 \section{Preliminary integral bounds}\label{pre-int-bd}
 
Recall that Corollary \ref{coro-dont-think-mild-dec-anymore} ensures that $\hat{h}$ and $\calB$ are in weighted $L^2$ with respect to $\Phi_{m}$ for all $m\in \RR$ under the assumption that we recall here for the sake of clarity:
\begin{equation}
\fint_{S_{\rho}}\left(|h|^2+\rho^2|\calB|^2\right)=O(\rho^{-\varepsilon}),\label{ass-prim-dec-h-B-record}
\end{equation}
 for some $\varepsilon>1$.
We can then record the following result:
\begin{prop}{\rm (\cite[Proposition 3.1]{Ber-Asym-Struct})} \label{prop-3-1-Ber}There is an $R>0$ so that if $\rho \geq R$ and $\hat{h} \in C^2_{-2n,1}(\bar{E}_R)$ then there exists a positive constant $C$ such that:
$$\int_{E_{\rho}} |\hat{h}|^2 \Phi_{-2n} \leq \frac{C}{\rho^2} \hat{D}(\rho) + \frac{C}{\rho}\hat{B}(\rho) \, .$$
\end{prop}
For later reference we note that integration by parts gives for all $\rho \geq 1$
\begin{equation}\label{eq-Ber-3-2}
\begin{split}
\int_\rho^{\rho+1} \Phi_{-2n}(t)\, dt &= \frac{2}{\rho} \Phi_{-2n}(\rho) -\frac{2}{\rho+1} \Phi_{-2n}(\rho+1) - 2(2n+1) \int_\rho^{\rho+1} \Phi_{-2n-2}(t)\, dt\\
&= \frac{2}{\rho} \Phi_{-2n}(\rho) + \Phi_{-2n}(\rho) O(\rho^{-2})
\end{split}
\end{equation}
The following lemma starts estimating the $L^2$-norm of the Bianchi one-form $\calB$ on (super-)level sets of $\br$ in terms of $\hat{h}$ if $(h,\calB)\in C^2(E_R)$ satisfies \eqref{eq:main.1}--\eqref{est-remain-term}.

\begin{lemma}{(Integral estimates on the Bianchi gauge)}\label{lemma-Bianchi-int-est}
Let $(h,\calB)\in C^2(E_R)$ satisfy \eqref{eq:main.1}--\eqref{est-remain-term} as well as \eqref{ass-prim-dec-h-B-record}. Then we have for $\rho\geq R$,
\begin{equation}
\begin{split}
\int_{S_{\rho}}\br\left|\nabla\br\right||\calB|^2\,\Psi_0+\int_{E_{\rho}}\br^2|\calB|^2\,\Psi_0
&\leq C\int_{E_{\rho}}\left(\br^{-6}|\nabla\hat{h}|^2+\br^{-4}|\hat{h}|^2\right)\,\Phi_{-2n}\\
&\leq C\rho^{-6}\hat{D}(\rho)+C\rho^{-5}\hat{B}(\rho).\label{lovely-bianchi-inequ}
\end{split}
\end{equation}
\end{lemma}
\begin{proof}
By integration by parts,
\begin{equation}
\begin{split}
\int_{E_{\rho}}\left\langle \nabla \left(\frac{\br^2}{4}\right),\nabla|\calB|^2\right\rangle\,\Psi_0=\,&\int_{E_{\rho}}\left\langle \nabla \left(e^{\frac{\br^2}{4}}\right),\nabla|\calB|^2\right\rangle\\
=\,&-\int_{E_{\rho}}\Delta \left(e^{\frac{\br^2}{4}}\right)|\calB|^2-\int_{S_{\rho}}\left|\nabla\left(\frac{\br^2}{4}\right)\right||\calB|^2\,\Psi_0\\
=\,&-\int_{E_{\rho}} \left(\frac{\br^2}{4}+\frac{n}{2}\right)|\calB|^2\,\Psi_0-\int_{S_{\rho}}\left|\nabla\left(\frac{\br^2}{4}\right)\right||\calB|^2\,\Psi_0.\label{first-IBP-bianchi}
\end{split}
\end{equation}
Now, thanks to (\ref{eq:main.2}), the lefthand side of (\ref{first-IBP-bianchi}) can be estimated as follows:
\begin{equation}
\begin{split}\label{sec-IBP-bianchi}
\left|\int_{E_{\rho}}\left(\left\langle \nabla \left(\frac{\br^2}{4}\right),\nabla|\calB|^2\right\rangle-|\calB|^2\right)\,\Psi_0\right|& \leq C\int_{E_{\rho}}\left(\br^{-2}|\nabla\hat{h}|+\br^{-1}|\hat{h}|\right)|\calB|\,\br^{n}\\
&\leq C\int_{E_{\rho}}\left(\br^{-6}|\nabla\hat{h}|^2+\br^{-4}|\hat{h}|^2\right)\,\Phi_{-2n}\\
&\quad +\frac{1}{8}\int_{E_{\rho}}\br^{2}|\calB|^2\,\Psi_0,
\end{split}
\end{equation}
where we have used Young's inequality in the last line.

Combining (\ref{first-IBP-bianchi}) and (\ref{sec-IBP-bianchi}) leads to the first part of the desired inequality (\ref{lovely-bianchi-inequ}):
\begin{equation*}
\begin{split}
\int_{S_{\rho}}\left|\nabla\left(\frac{\br^2}{4}\right)\right||\calB|^2\,\Psi_0+&\int_{E_{\rho}}\left(\frac{\br^2}{8}+\frac{n}{2}+1\right)|\calB|^2\,\Psi_0\\
&\leq C\int_{E_{\rho}}\left(\br^{-6}|\nabla\hat{h}|^2+\br^{-4}|\hat{h}|^2\right)\,\Phi_{-2n}.
\end{split}
\end{equation*}
The second inequality in (\ref{lovely-bianchi-inequ}) is obtained from the previous one together with Proposition \ref{prop-3-1-Ber}.
\end{proof}

\begin{lemma}{\rm (\cite[Lemma 3.3]{Ber-Asym-Struct})}\label{lemma-3-3-Ber-a}
Let $(h,\calB)\in C^2(E_R)$ satisfy \eqref{eq:main.1}--\eqref{est-remain-term} and \eqref{ass-prim-dec-h-B-record}. Then there exist $C>0$ and $\rho_0\geq R$ such that if $\rho\geq \rho_0$,
\begin{eqnarray}
|\hat{L}(\rho)|&\leq\,&\frac{1}{8}\rho^{-2}\hat{D}(\rho)+C\rho^{-3}\hat{B}(\rho)\label{1-3-3-Ber}\\
&\leq&\frac{1}{4}\rho^{-2}\hat{F}(\rho)+C\rho^{-3}\hat{B}(\rho).\label{2-3-3-Ber}
\end{eqnarray}
In particular, if $B(\rho)>0$,
\begin{equation}
\left|N(\rho)-\hat{N}(\rho)\right|\leq\frac{1}{4}\rho^{-2}N(\rho)+C\rho^{-2}.\label{diff-fre-fct-N-N-m}
\end{equation}

\end{lemma}
\begin{proof}
Let us integrate by parts by recalling the definition of $\hat{h}:=\Psi_{n}\cdot h,$ 
\begin{align*}
\hat{L}(\rho) =\,& \int_{E_\rho} \langle \hat{h} , \calL_{-2n} \hat{h}\rangle \,\Phi_{-2n}\\
=\,&\int_{E_\rho} \langle \hat{h} ,  \Psi_n \Li_{\calB}(g)+R[\hat{h}]\rangle\, \Phi_{-2n}\\
=\,&\int_{E_\rho} \langle \hat{h}  ,   \Li_{\calB}(g)\rangle \,\br^{-n}+\int_{E_{\rho}}\langle \hat{h}, R[\hat{h}]\rangle\,\Phi_{-2n}\\
=\,&-2\int_{E_\rho} \langle \calB ,   \div\left(\br^{-n}\hat{h} \right)\rangle -2\int_{S_{\rho}}\hat{h} (\calB, \bN)\,\br^{-n}+\int_{E_\rho} \langle \hat{h}  ,R[\hat{h}]\rangle \,\Phi_{-2n}.
\end{align*}
In particular, for $\varepsilon\in(0,1)$,
\begin{equation}
\begin{split}
\bigg|\hat{L}(\rho)&-\int_{E_\rho} \langle \hat{h}  ,R[\hat{h}]\rangle \,\Phi_{-2n}\bigg|\\
\leq&\,2\int_{E_\rho} | \calB|  \left|\div\left(\br^{-n}\hat{h} \right)\right|+2\int_{S_{\rho}}|\hat{h}| |\calB|\,\br^{-n}\\
\leq\,&C\int_{E_\rho} | \calB|  \left(|\nabla\hat{h}|+\br^{-1}|\hat{h}| \right)\,\br^{-n}+2\int_{S_{\rho}}|\hat{h}| |\calB|\,\br^{-n}\\
\leq\,&C_{\varepsilon}\int_{E_\rho} \br^2| \calB|^2\,\Psi_0+\varepsilon\int_{E_{\rho}}\br^{-2}\left(|\nabla\hat{h}|^2+\br^{-2}|\hat{h}|^2 \right)\,\Phi_{-2n}\\
&+C\br^{-3}\int_{S_{\rho}}|\hat{h}|^2\,\Phi_{-2n}+\rho^2\int_{S_{\rho}}\br|\calB|^2\,\Psi_0\\
\leq\,&\left(\varepsilon+C\rho^{-4}\right)\int_{E_{\rho}}\br^{-2}|\nabla\hat{h}|^2\,\Phi_{-2n}+(C_{\varepsilon}+\varepsilon)\int_{E_{\rho}}\br^{-4}|\hat{h}|^2\,\Phi_{-2n}\\
&+C\br^{-3}\int_{S_{\rho}}|\hat{h}|^2\,\Phi_{-2n}+C\rho^2\int_{E_{\rho}}\left(\br^{-6}|\nabla\hat{h}|^2+\br^{-4}|\hat{h}|^2\right)\,\Phi_{-2n}\\
\leq\,&\left(\varepsilon+C\rho^{-2}\right)\int_{E_{\rho}}\br^{-2}|\nabla\hat{h}|^2\,\Phi_{-2n}+C_{\varepsilon}\int_{E_{\rho}}\br^{-2}|\hat{h}|^2\,\Phi_{-2n}+C\br^{-3}\int_{S_{\rho}}|\hat{h}|^2\,\Phi_{-2n},\label{est-L-hat-rho}
\end{split}
\end{equation}
where $C_{\varepsilon}$ is a positive constant that may vary from line to line.
Here we have used Lemma \ref{lemma-Bianchi-int-est} in the fourth inequality and Young's inequality in the third line.

Finally, by Corollary \ref{coro-syst-cons} and Proposition \ref{prop-3-1-Ber},
\begin{equation}
\begin{split}\label{est-L-hat-rho-remain}
\left|\int_{E_\rho} \langle \hat{h}  ,R[\hat{h}]\rangle \,\Phi_{-2n}\right|&\leq \int_{E_\rho} \br^{-2}|\hat{h}|\left(|\hat{h}|+|\nabla \hat{h}|\right)\,\Phi_{-2n}\\
&\leq \frac{C}{\rho^4}\hat{D}(\rho)+\frac{C}{\rho^3}\hat{B}(\rho)+\int_{E_\rho} \br^{-2}|\hat{h}||\nabla \hat{h}|\,\Phi_{-2n}\\
&\leq \frac{C}{\rho^4}\hat{D}(\rho)+\frac{C}{\rho^3}\hat{B}(\rho)+\int_{E_\rho}\left( C_{\varepsilon}\br^{-2}|\hat{h}|^2+\varepsilon\br^{-2}|\nabla \hat{h}|^2\right)\,\Phi_{-2n}\\
&\leq \left(\varepsilon+\frac{C}{\rho^2}\right)\rho^{-2}\hat{D}(\rho)+\frac{C}{\rho^3}\hat{B}(\rho),
\end{split}
\end{equation}
for any $\varepsilon>0$ and where $C$, $C_{\varepsilon}$ are positive constants that may vary from line to line.

To conclude, we use the definitions of the various quantities defined at the end of Section \ref{sec:setup} together with the previous estimates (\ref{est-L-hat-rho}) and \eqref{est-L-hat-rho-remain}:
\begin{equation*}
\begin{split}
|\hat{L}(\rho)|\leq\,&\left(\varepsilon+C\rho^{-2}\right)\rho^{-2}\hat{D}(\rho)+C_{\varepsilon}\rho^{-3}\hat{B}(\rho)\\
&+C_{\varepsilon}\rho^{-4} \hat{D}(\rho) +C_{\varepsilon}\rho^{-3}\hat{B}(\rho)\\
\leq\,&\left(\varepsilon+C_{\varepsilon}\rho^{-2}\right)\rho^{-2}\hat{D}(\rho)+C_{\varepsilon}\rho^{-3}\hat{B}(\rho),
\end{split}
\end{equation*}
where we have used Proposition \ref{prop-3-1-Ber} in the first line. This proves the expected result by choosing $\varepsilon$ small enough and by considering $\rho\geq \rho_0$ so large that $\varepsilon+C_{\varepsilon}\rho_0^{-2}\leq \frac{1}{8}$.

This establishes (\ref{1-3-3-Ber}): inequality (\ref{2-3-3-Ber}) is obtained thanks to identity (\ref{id-F-D-L}) together with (\ref{1-3-3-Ber}) and Young's inequality.
\end{proof}

\begin{coro}{\rm (\cite[Lemma 3.5]{Ber-Asym-Struct})}\label{lemma-3-5-Ber}
Let $(h,\calB)\in C^2(E_R)$ satisfy \eqref{eq:main.1}--\eqref{est-remain-term} and \eqref{ass-prim-dec-h-B-record}.  Then either $h\equiv0$ on $E_{\rho}$ for some $\rho\geq R$ or $B(\rho)>0$ for all $\rho\geq R$.
 \end{coro}
 
 \begin{proof}
 We follow Bernstein's arguments very closely. Assume that $B(\rho)=0$. Then $h\equiv0$ on $S_{\rho}$ which implies that $\hat{F}(\rho)=0$. In particular, [(\ref{2-3-3-Ber}), Lemma \ref{lemma-3-3-Ber-a}] ensures that $\hat{L}(\rho)=0$ too. Thanks to identity (\ref{id-F-D-L}), we then get that $\hat{h}$ is parallel on $E_{\rho}$. As a consequence, the norm $|\hat{h}|$ is constant on any connected component of $E_{\rho}$ and vanishes on $S_{\rho}$, so $h\equiv0$ on $E_{\rho}$. Alternatively, one can use Proposition \ref{prop-3-1-Ber} to reach the same conclusion.
 \end{proof}

 \begin{lemma}{\rm (\cite[Lemma 3.3]{Ber-Asym-Struct})}\label{lemma-3-3-Ber-b}
Let $(h,\calB)\in C^2(E_R)$ satisfy \eqref{eq:main.1}--\eqref{est-remain-term} and \eqref{ass-prim-dec-h-B-record}. Then,
\begin{equation*}
\begin{split}
\Bigg|\int_{E_{\rho}} \langle \nabla_\bX \hat{h},& \calL_{-2n}\hat{h} \rangle \,\Phi_{-2n}\Bigg|\\
&\leq  \frac{C}{\rho}\left(\int_{E_{\rho}}|\nabla_{\bN}\hat{h}|^2\,\Phi_{-2n}\right)^{\frac{1}{2}}\left(\hat{D}(\rho)+\frac{1}{\rho}\hat{B}(\rho)\right)^{\frac{1}{2}}\\
 &\quad +\frac{C}{\rho}\left(\hat{D}(\rho)+\frac{1}{\rho^{\frac{1}{2}}}\hat{D}^{\frac{1}{2}}(\rho)\hat{B}^{\frac{1}{2}}(\rho)\right)\\
&\quad+ C\left(\int_{S_{\rho}}|\nabla\hat{h}|^2\Phi_{-2n}\right)^{\frac{1}{2}}\left(\int_{S_{\rho}}\br^2|\calB|^2\,\Psi_0\right)^{\frac{1}{2}}.
\end{split}
\end{equation*}

\end{lemma}
\begin{proof}
Observe first that Corollary \ref{coro-syst-cons} implies:
\begin{equation}
\begin{split}\label{prelim-ineq-rad-der-lin-op}
\Bigg|\int_{E_{\rho}} \langle \nabla_\bX \hat{h},& \calL_{-2n}\hat{h} \rangle \,\Phi_{-2n}\Bigg|\leq \Bigg|\int_{E_{\rho}} \langle \nabla_\bX \hat{h}, \calL_{\calB}(g) \rangle \,\br^{-n}\Bigg|+\int_{E_{\rho}} |\nabla_\bX \hat{h}||R[\hat{h}]| \,\Phi_{-2n}.
\end{split}
\end{equation}
Let us handle the non-linear terms first as follows:
\begin{equation*}
\begin{split}
\int_{E_{\rho}} |\nabla_\bX \hat{h}||R[\hat{h}]| \,\Phi_{-2n}&\leq \left(\int_{E_{\rho}}|\nabla_{\bN}\hat{h}|^2\,\Phi_{-2n}\right)^{\frac{1}{2}} \left(\int_{E_{\rho}}\br^2|R[\hat{h}]|^2\,\Phi_{-2n}\right)^{\frac{1}{2}}\\
&\leq C\left(\int_{E_{\rho}}|\nabla_{\bN}\hat{h}|^2\,\Phi_{-2n}\right)^{\frac{1}{2}} \left(\int_{E_{\rho}}\br^{-2}\left(|\hat{h}|^2+|\nabla\hat{h}|^2\right)\,\Phi_{-2n}\right)^{\frac{1}{2}}\\
&\leq \frac{C}{\rho}\left(\int_{E_{\rho}}|\nabla_{\bN}\hat{h}|^2\,\Phi_{-2n}\right)^{\frac{1}{2}} \left(\hat{D}(\rho)+\frac{1}{\rho}\hat{B}(\rho)\right)^{\frac{1}{2}}.
\end{split}
\end{equation*}
Here we have used Cauchy-Schwarz inequality in the first line together with Corollary \ref{coro-syst-cons} in the second line and Proposition \ref{prop-3-1-Ber} in the last line.

To handle the first integral term on the righthand side of \eqref{prelim-ineq-rad-der-lin-op}, we first establish the following estimate:
\begin{equation}\label{eq:est.int.parts}
\begin{split}
\frac{1}{2}\langle \nabla_\bX\hat{h},\calL_{\calB}(g)\rangle \br^{-n} &= \div\left(\left(\nabla_\bX \hat{h}(\calB)-\langle\div\hat{h},\calB\rangle\bX\right)\br^{-n}\right)+n\br^{-n-1}\nabla_\bX \hat{h}(\calB,\nabla\br)\\
 &\quad+ O\left(  \br^{-2}|\calB|\left(\br^{-1} | \hat{h}| + |\nabla \hat{h}|\right)+\br^{-1}\left(|\hat{h}|+ \br^{-1}|\nabla \hat{h}|\right) |\nabla \hat{h}|\Phi_{-2n}\right).
\end{split}
\end{equation}
To prove \eqref{eq:est.int.parts} we recall that
\begin{equation}\label{eq:computation.1}
\bX = \br \frac{\nabla^g\br}{|\nabla^g\br|^2}\ \  \text{ and } \ \ g(\nabla_Y\bX, Z) = g(Y,Z) + O(\br^{-2} |Y||Z|),
\end{equation}
for any tangent vectors $Y,Z$ from the beginning of Section \ref{sec:setup}. We compute
\begin{equation}\label{eq:computation.2}
\begin{split}
 \langle \nabla_\bX\hat{h}, \nabla \calB\rangle \br^{-n} &=  \bX_l \nabla_l \hat{h}_{ij} \nabla_i\calB_j\br^{-n}\\
 &= \nabla_i\left(\bX_l \nabla_l \hat{h}_{ij} \calB_j\br^{-n}\right) - \nabla_i \hat{h}_{ij} \calB_j\br^{-n} - \bX_l \nabla_i\nabla_l \hat{h}_{ij} \calB_j\br^{-n} \\
 &\quad + n  \bX_l \nabla_l \hat{h}_{ij} \calB_j\br^{-n-1}\nabla_i\br + O(\br^{-n-2} |\nabla \hat{h}| |\calB|)\\
 &= \nabla_i\left(\bX_l \nabla_l \hat{h}_{ij} \calB_j\br^{-n}\right) - \nabla_i \hat{h}_{ij} \calB_j\br^{-n} - \bX_l \nabla_l\nabla_i \hat{h}_{ij} \calB_j\br^{-n} \\
 &\quad+\bX_l\left(\Rm(g)_{ili}^p\hat{h}_{pj}+\Rm(g)_{ilj}^p\hat{h}_{ip}\right)\calB_j\br^{-n}\\
 &\quad + n  \bX_l \nabla_l \hat{h}_{ij} \calB_j\br^{-n-1}\nabla_i\br + O(\br^{-n-2} |\nabla \hat{h}| |\calB|)\\
 &= \nabla_i\left(\bX_l \nabla_l \hat{h}_{ij} \calB_j\br^{-n}\right) - \nabla_l\left(\bX_l \nabla_i \hat{h}_{ij} \calB_j\br^{-n}\right) - \nabla_i \hat{h}_{ij} \calB_j\br^{-n}  \\
 &\quad-\langle\Ric(g)(\bX),\hat{h}(\calB)\rangle_g\br^{-n}+\Rm(g)_{i\bX\calB}^p\hat{h}_{ip}\br^{-n}\\
 &\quad + n   \nabla_\bX \hat{h}_{ij} \calB_j\br^{-n-1}\nabla_i\br +\div(\bX) \nabla_i \hat{h}_{ij} \calB_j\br^{-n} 
 +  \nabla_i \hat{h}_{ij} \nabla_\bX\calB_j\br^{-n}\\
 &\quad - n \nabla_i \hat{h}_{ij} \calB_j \br^{-n} + O(\br^{-n-2} |\nabla \hat{h}| |\calB|+ \br^{-n-3} | \hat{h}| |\calB|)\\
 &= \nabla_i\left(\bX_l \nabla_l \hat{h}_{ij} \calB_j\br^{-n}\right) - \nabla_l\left(\bX_l \nabla_i \hat{h}_{ij} \calB_j\br^{-n}\right)\\
 &\quad + \left(\div(\bX) -n -1\right) \div \hat{h}(\calB) \br^{-n} +  \div\hat{h} (\nabla_\bX\calB) \br^{-n}\\
 &\quad + n   \nabla_\bX \hat{h}(\calB,\nabla\br)\br^{-n-1} + O(\br^{-n-2} |\nabla \hat{h}| |\calB|+ \br^{-n-3} | \hat{h}| |\calB|).
\end{split}
\end{equation}
Here we have used commutation formulas in the third line together with identities \eqref{eq:2} and \eqref{eq:4} to ensure that $\Rm(g)(\bX,\cdot,\cdot,\cdot)=O(\br^{-3})$ in the fifth equality.

Now note that \eqref{eq:computation.1} implies that
$$ \div(\bX) = n + O(\br^{-2})$$
as well as together with \eqref{eq:6} and \eqref{eq:main.2} that
\begin{equation*}
\begin{split}
 \nabla_\bX \calB& = \frac{1}{|\nabla\br|^2_g}2\nabla_{\nabla f}\calB \\
 &=\frac{1}{|\nabla\br|^2_g}\left(\calB+   O(\br^{-3}) |h| + O(\br^{-2}) |\nabla h|)\right)\\
 &=\calB+O(\br^{-4})\calB+O(\br^{-3}) |h| + O(\br^{-2}) |\nabla h|.
 \end{split}
 \end{equation*}
Combining this with \eqref{eq:computation.2} we see that
\begin{align*}
 \langle \nabla_\bX\hat{h}, \nabla \calB\rangle \br^{-n} &= \nabla_i\left(\bX_l \nabla_l \hat{h}_{ij} \calB_j\br^{-n}\right) - \nabla_l\left(\bX_l \nabla_i \hat{h}_{ij} \calB_j\br^{-n}\right) + n   \nabla_\bX \hat{h}_{ij} \calB_j\br^{-n-1}\nabla_i\br\\
 &\quad + O(\br^{-n-2} |\nabla \hat{h}| |\calB|+ \br^{-n-3} | \hat{h}| |\calB| +  \br^{-n-3} |\nabla \hat{h}| |h|\\
 &\qquad+ \br^{-n-2} |\nabla \hat{h}| |\nabla h|)\\
 &=\div\left(\left(\nabla_\bX \hat{h}(\calB)-\langle\div\hat{h},\calB\rangle_g\bX\right)\br^{-n}\right)+n\br^{-n-1}\nabla_\bX \hat{h}(\calB,\nabla\br)\\
 &\quad+ O(\br^{-n-2} |\nabla \hat{h}| |\calB|+ \br^{-n-3} | \hat{h}| |\calB| +  \br^{-n-3} |\nabla \hat{h}| |h|\\
 &\qquad + \br^{-n-2} |\nabla \hat{h}| |\nabla h|)\\
 &=\div\left(\left(\nabla_\bX \hat{h}(\calB)-\langle\div\hat{h},\calB\rangle_g\bX\right)\br^{-n}\right)+n\br^{-n-1}\nabla_\bX \hat{h}(\calB,\nabla\br)\\ &\quad+ O\left( \br^{-2} |\calB|\left(\br^{-1} | \hat{h}| + |\nabla \hat{h}|\right)+\br^{-1}\left(|h|+ \br^{-1}|\nabla h|\right) |\nabla \hat{h}|\right)\br^{-n}.
\end{align*}
Recall that $\hat{h}:=\Psi_n\cdot h$. In particular, $\left(|h|+\br^{-1}|\nabla h|\right)\br^{-n}=O\left(|\hat{h}|+\br^{-1}|\nabla\hat{h}|\right)\Phi_{-2n}$.
This establishes \eqref{eq:est.int.parts}.

Now, recalling \eqref{eq:eqhhat}, the divergence theorem applied to \eqref{eq:est.int.parts} leads to:
\begin{align*}
\left|\int_{E_{\rho}} \langle \nabla_\bX \hat{h}, \calL_{-2n}\hat{h} \rangle \Phi_{-2n}\right|
 &\leq  C\int_{S_{\rho}}|\nabla\hat{h}||\calB|\br^{1-n} +C\int_{E_{\rho}}|\nabla_{\bN}\hat{h}||\calB|\br^{-n}\\
 &\quad+  \frac{C}{\rho^2}\int_{E_{\rho}}(\br^{-1}|\hat{h}|+|\nabla\hat{h}|)|\calB|\br^{-n}\\
 &\quad +C\int_{E_{\rho}} \br^{-1}\left(|\hat{h}|+\br^{-1}|\nabla \hat{h}|\right) | \nabla \hat{h}|\,\Phi_{-2n}\\ 
 &\leq C\left(\int_{S_{\rho}}|\nabla\hat{h}|^2\Phi_{-2n}\right)^{\frac{1}{2}}\left(\int_{S_{\rho}}\br^2|\calB|^2\,\Psi_0\right)^{\frac{1}{2}}\\
 &\quad+\frac{C}{\rho}\left(\int_{E_{\rho}}|\nabla_{\bN}\hat{h}|^2\,\Phi_{-2n}\right)^{\frac{1}{2}}\left(\int_{E_{\rho}}\br^2|\calB|^2\,\Psi_0\right)^{\frac{1}{2}}\\
 &\quad+\frac{C}{\rho^2}\left(\int_{E_{\rho}}\br^{-2}(\br^{-2}|\hat{h}|^2+|\nabla\hat{h}|^2)\,\Phi_{-2n}\right)^{\frac{1}{2}}\left(\int_{E_{\rho}}\br^2|\calB|^2\,\Psi_0\right)^{\frac{1}{2}}\\
 &\quad +\frac{C}{\rho^2}\int_{E_{\rho}}|\nabla\hat{h}|^2\,\Phi_{-2n}+\frac{C}{\rho}\left(\int_{E_{\rho}}|\hat{h}|^2\,\Phi_{-2n}\right)^{\frac{1}{2}}\left(\int_{E_{\rho}}|\nabla\hat{h}|^2\,\Phi_{-2n}\right)^{\frac{1}{2}}\\
 &\leq C\left(\int_{S_{\rho}}|\nabla\hat{h}|^2\Phi_{-2n}\right)^{\frac{1}{2}}\left(\int_{S_{\rho}}\br^2|\calB|^2\,\Psi_0\right)^{\frac{1}{2}}\\
 &\quad+\frac{C}{\rho}\left(\int_{E_{\rho}}|\nabla_{\bN}\hat{h}|^2\,\Phi_{-2n}\right)^{\frac{1}{2}}\left(\hat{D}(\rho)+\frac{1}{\rho}\hat{B}(\rho)\right)^{\frac{1}{2}}\\
 &\quad +\frac{C}{\rho}\left(\hat{D}(\rho)+\frac{1}{\rho^{\frac{1}{2}}}\hat{D}^{\frac{1}{2}}(\rho)\hat{B}^{\frac{1}{2}}(\rho)\right)
 \end{align*}
Here we have used Proposition \ref{prop-3-1-Ber} and Lemma \ref{lemma-Bianchi-int-est} in the third inequality.
 \end{proof}
 
 \section{Frequency bounds}\label{sec-fre-bds}
We adapt \cite[Proposition $4.2$]{Ber-Asym-Struct} and \cite[Corollary $4.3$]{Ber-Asym-Struct} that hold true in general:
 \begin{prop}\label{prop-mix-4-2-4-3}
Let $(h,\calB)\in C^2(E_R)$ satisfy \eqref{eq:main.1}--\eqref{est-remain-term} and \eqref{ass-prim-dec-h-B-record}.  Then if $\rho\geq R$,
 \begin{equation}
 \begin{split}\label{comp-formula-der-D}
\hat{D}'(\rho)&=-\frac{2}{\rho}\int_{E_{\rho}}\langle\nabla_{\bX}\hat{h},\calL_{-2n}\hat{h}\rangle\,\Phi_{-2n}-2\int_{S_{\rho}}\frac{|\nabla_{\bN}\hat{h}|^2}{|\nabla \br|}\,\Phi_{-2n}-\left(\frac{n+2}{\rho}+\frac{\rho}{2}\right)\hat{D}(\rho)\\
&\quad-\frac{1}{\rho}\int_{\rho}^{\infty}t\hat{D}(t)\,dt+O(\rho^{-3})\hat{D}(\rho)+O(\rho^{-\frac{9}{2}})\hat{D}^{\frac{1}{2}}(\rho)\hat{B}^{\frac{1}{2}}(\rho).
\end{split}
\end{equation}
In particular, if $B(\rho)>0$,
\begin{equation}
\begin{split}\label{comp-formula-der-N-m}
\hat{N}'(\rho)&=-\frac{2}{\hat{B}(\rho)}\int_{E_{\rho}}\langle\nabla_{\bX}\hat{h}+N(\rho)\hat{h},\calL_{-2n}\hat{h}\rangle\,\Phi_{-2n}-\frac{1}{\hat{B}(\rho)}\int_{\rho}^{\infty}t\hat{D}(t)\,dt\\
&\quad-\frac{2}{\rho\hat{B}(\rho)}\int_{S_{\rho}}|\nabla_{\bX}\hat{h}+N(\rho)\hat{h}|^2|\nabla \br|\,\Phi_{-2n}+O(\rho^{-3})\hat{N}(\rho)+O(\rho^{-4})\hat{N}^{\frac{1}{2}}(\rho).
\end{split}
\end{equation}

\end{prop}
 \begin{proof}
 The proof of Proposition \ref{prop-mix-4-2-4-3} is a straightforward adaptation of the aforementioned results due to Bernstein by using the corresponding Rellich-Necas identity for a $C^2_{loc}$ symmetric $2$-tensor $\hat{h}$:
 \begin{equation*}
\begin{split}\label{rellich-necas-id}
\div\left(\langle\nabla_{\bX}\hat{h},\nabla_{\cdot}\hat{h}\rangle\Phi_{-2n}\right)&:=\nabla_k\left(\nabla_{\bX}\hat{h}_{ij}\nabla_k\hat{h}_{ij}\Phi_{-2n}\right)\\
&=\nabla_k\left(\bX_l\nabla_l\hat{h}_{ij}\right)\nabla_k\hat{h}_{ij}\Phi_{-2n}+\bX_l\nabla_l\hat{h}_{ij}\nabla_k\left(\nabla_k\hat{h}_{ij}\Phi_{-2n}\right)\\
&=\left(\nabla_k\bX_l\right)\nabla_l\hat{h}_{ij}\nabla_k\hat{h}_{ij}\Phi_{-2n}+\bX_l\nabla_k\nabla_l\hat{h}_{ij}\nabla_k\hat{h}_{ij}\Phi_{-2n}\\
&\quad+\bX_l\nabla_l\hat{h}_{ij}\nabla_k\left(\nabla_k\hat{h}_{ij}\Phi_{-2n}\right)\\
&=\left(1+O(\br^{-2})\right)|\nabla \hat{h}|^2\Phi_{-2n}+\langle\nabla_X\nabla\hat{h},\nabla\hat{h}\rangle\Phi_{-2n}\\
&\quad+\Rm(\bX,\cdot,\cdot,\cdot)\ast \hat{h}\ast\nabla \hat{h}\,\Phi_{-2n}+\langle\nabla_X\hat{h},\calL_{-2n}\hat{h}\rangle\Phi_{-2n}\\
&=|\nabla \hat{h}|^2\Phi_{-2n}+\langle\nabla_X\nabla\hat{h},\nabla\hat{h}\rangle\Phi_{-2n}+\langle\nabla_X\hat{h},\calL_{-2n}\hat{h}\rangle\Phi_{-2n}\\
&\quad+\left(O(\br^{-2})|\nabla\hat{h}|^2+O(\br^{-3})|\hat{h}||\nabla\hat{h}|\right)\Phi_{-2n}.
\end{split}
\end{equation*}
Here, we have used commutation formulas in the fourth line together with identity \eqref{eq:4} to ensure that $\Rm(g)(\bX,\cdot,\cdot,\cdot)=O(\br^{-3})$ in the fifth equality.

By the very definition of $\hat{D}(\rho)$, one gets by integration by parts:
\begin{equation*}
\begin{split}
\hat{D}'(\rho)&=-\Phi_{-2n}(\rho)\int_{S_{\rho}}\frac{|\nabla \hat{h}|^2}{|\nabla\br|}=\frac{1}{\rho}\int_{E_{\rho}}\div\left(|\nabla \hat{h}|^2\Phi_{-2n}\bX\right)\\
&=\frac{1}{\rho}\int_{E_{\rho}}\left(\nabla_{\bX}|\nabla \hat{h}|^2+\left(-n-\frac{\br^2}{2}\right)|\nabla\hat{h}|^2\right)\,\Phi_{-2n}+O(\rho^{-3})\hat{D}(\rho).
\end{split}
\end{equation*}

Now, the co-area formula together with Fubini's theorem, the divergence theorem and identity (\ref{rellich-necas-id}) as in the proof of \cite[Proposition $4.2$]{Ber-Asym-Struct} give:
\begin{equation*}
\begin{split}
\hat{D}'(\rho)&=\frac{1}{\rho}\int_{E_{\rho}}-2\langle\nabla_X\hat{h},\calL_{-2n}\hat{h}\rangle\Phi_{-2n}-\left(\frac{n+2}{\rho}+\frac{\rho}{2}\right)\hat{D}(\rho)-2\int_{S_{\rho}}\frac{|\nabla_{\bN}\hat{h}|^2}{|\nabla\br|}\,\Phi_{-2n}\\
&\quad-\frac{1}{\rho}\int_{\rho}^{\infty}t\hat{D}(t)\,dt+O(\rho^{-3})\hat{D}(\rho)+O(\rho^{-\frac{9}{2}})\hat{D}^{\frac{1}{2}}(\rho)\hat{B}^{\frac{1}{2}}(\rho),
\end{split}
\end{equation*}
where we have used the estimate from Proposition \ref{prop-3-1-Ber} to estimate the integral term involving $|\hat{h}|^2$. This ends the proof of (\ref{comp-formula-der-D}).

The proof of (\ref{comp-formula-der-N-m}) is word for word that of \cite[Corollary $4.3$]{Ber-Asym-Struct} based on (\ref{comp-formula-der-D}).
 \end{proof}
 
 We start with computing the derivative of the "frequency" function associated to a vector field $\calB$ satisfying \eqref{eq:main.2}, defined by
 $$N_{\calB}(\rho):=\frac{\rho^2\int_{S_{\rho}}|\calB|^2|\nabla\br|\,\Psi_0}{\hat{B}(\rho)}.$$
 Note that
 $$N_{\calB}(\rho)=\Psi_{2n+2}(\rho)\Psi_0(\rho)\frac{B_{\calB}(\rho)}{B(\rho)}=\Psi_2(\rho)\frac{B_{\calB}(\rho)}{\hat{B}(\rho)}.$$
 \begin{prop}\label{prop-never-stops}
Let $(h,\calB)\in C^2(E_R)$ satisfy \eqref{eq:main.1}--\eqref{est-remain-term} and \eqref{ass-prim-dec-h-B-record}. Assume $h$ is non-trivial. Then, for $\rho\geq R$,
 \begin{equation}
N_{\calB}(\rho)\leq C\left(\rho^{-6}N(\rho)+\rho^{-4}\right),\label{comp-fre-fcts}
\end{equation}
and
 \begin{equation}
 \begin{split}
 N'_{\calB}(\rho)&=\left(\rho+\frac{2n+4}{\rho}+O(\rho^{-2})\right)N_{\calB}(\rho)+2\frac{N(\rho)}{\rho}N_{\calB}(\rho)\\
 &\quad+\left(O(\rho^{-4})\frac{\int_{S_{\rho}}|\nabla \hat{h}|^2\,\Phi_{-2n}}{\hat{B}(\rho)}+O(\rho^{-2})\right)^{\frac{1}{2}}N_{\calB}^{\frac{1}{2}}(\rho).\label{comp-fre-fcts-der}
 \end{split}
 \end{equation}
 
 \end{prop}
 
 \begin{proof}
 We begin with the proof of \eqref{comp-fre-fcts}. From Lemma \ref{lemma-Bianchi-int-est} we have
 \begin{equation*}
 \begin{split}
 N_{\calB}(\rho)&=\frac{\rho^2\Psi_0(\rho)B_{\calB}(\rho)}{\hat{B}(\rho)}\leq \frac{C}{\hat{B}(\rho)}\left(\rho^{-5}\hat{D}(\rho)+\rho^{-4}\hat{B}(\rho)\right)\\
 &=C\left(\rho^{-6}\hat{N}(\rho)+\rho^{-4}\right)\\
 &\leq C\left(\rho^{-6}N(\rho)+\rho^{-4}\right),
 \end{split}
 \end{equation*}
 where we have used [\eqref{diff-fre-fct-N-N-m}, Lemma \ref{lemma-3-3-Ber-a}] in the last line.
 
 Now, by using Lemmata \ref{lemma-Ber-3-2} and \ref{lemma-bianchi-level-set-var},
 \begin{align*}
\hat{B}^2(\rho)N'_{\calB}(\rho)&=\rho^2\Psi_0(\rho)B_{\calB}'(\rho)\hat{B}(\rho)-\rho^2\Psi_0(\rho)B_{\calB}(\rho)\hat{B}'(\rho)\\
&\quad+\left(\frac{\rho}{2}+\frac{2}{\rho} \right)\rho^2\Psi_0(\rho)B_{\calB}(\rho)\hat{B}(\rho)\\
&=\left(\rho+\frac{2n+4}{\rho}+O(\rho^{-2})\right)\rho^2\Psi_0(\rho)B_{\calB}(\rho)\hat{B}(\rho)+2\rho^2\Psi_0(\rho)\hat{F}(\rho)B_{\calB}(\rho)\\
&\quad+\rho^2\Psi_0(\rho)\left(\int_{S_{\rho}}O(\br^{-8})|h|^2+O(\br^{-6})|\nabla h|^2\right)^{\frac{1}{2}}B_{\calB}^{\frac{1}{2}}(\rho)\hat{B}(\rho)\\
&=\left(\rho+\frac{2n+4}{\rho}+O(\rho^{-2})\right)\rho^2\Psi_0(\rho)B_{\calB}(\rho)\hat{B}(\rho)+2\rho^2\Psi_0(\rho)\hat{F}(\rho)B_{\calB}(\rho)\\
&\quad+\left(\int_{S_{\rho}}\left(O(\br^{-2})|\hat{h}|^2+O(\br^{-4})|\nabla \hat{h}|^2\right)\Phi_{-2n}\right)^{\frac{1}{2}}\left(\rho^2\Psi_0(\rho)B_{\calB}(\rho)\right)^{\frac{1}{2}}\hat{B}(\rho).
 \end{align*}
 Dividing the previous estimate by $\hat{B}^2(\rho)$ (legitimated by Corollary \ref{lemma-3-5-Ber}) leads to the expected result.
 \end{proof}
 
 In order to estimate the derivative of the frequency function $N$, we need an additional lemma that handles the integral involving the Lie derivative of the Bianchi gauge on the level sets of $\br$:
 \begin{lemma}\label{lemma-IBP-Lie-level-set}
Let $T$ be a $C^1_{loc}(E_R)$ symmetric $2$-tensor and let $Y$ be a $C^1_{loc}(E_R)$ vector field. Denote by $g_{\br}$ the metric on $S_{\br}$ induced by $g$ and let $T^{\top}$ (respectively $Y^{\top}$) the tangential part of $T$ (respectively $Y$).
 Then,
\begin{equation*}
\begin{split}
\langle T,\calL_Y(g)\rangle_g&=2\left[\bN\cdot\langle Y,T(\bN)\rangle_g-\langle Y,\nabla_{\bN}\bN\rangle_g T(\bN,\bN)-\bN\cdot\left(T(\bN,\bN)\right)\langle Y,\bN\rangle_g\right]\\
&\quad+2\left[-\langle Y,\nabla_{\bN}T(\bN)^{\top}\rangle_g-\langle Y^{\top},\nabla_{T(\bN)^{\top}}\bN\rangle_g+T(\bN)^{\top}\cdot\langle Y,\bN\rangle_g\right]\\
&\quad+\langle Y,\bN\rangle_g\langle T^{\top},\calL_{\bN}(g)\rangle_g+\langle T^{\top},\calL_{Y^{\top}}(g_{\br})\rangle_{g_{\br}}.
\end{split}
\end{equation*}

In particular, there is a positive constant $C$ such that on $E_R$,
\begin{equation*}
\begin{split}
\Bigg|\int_{S_{\br}}\frac{\langle T,\calL_Y(g)\rangle_g-2\bN\cdot\langle Y,T(\bN)\rangle_g}{|\nabla\br|}\Bigg|&\leq C\int_{S_{\br}}|Y|_g\left(\br^{-1}|T|_g+|\nabla T|_g\right).
\end{split}
\end{equation*}
 \end{lemma}
 \begin{proof}
 Decompose $T$ as follows:
 \begin{equation*}
 T=T(\bN,\bN)\frac{d\br\otimes d\br}{|\nabla\br|^2}+\langle
 T(\bN)^{\top},\cdot\rangle_g\otimes \frac{d\br}{|\nabla\br|}+\frac{d\br}{|\nabla\br|}\otimes \langle T(\bN)^{\top},\cdot\rangle_g+T^{\top}.
\end{equation*}
In particular,
\begin{equation}
\begin{split}
\langle T,\calL_{Y}(g)\rangle_g&= T(\bN,\bN)\calL_Y(g)(\bN,\bN)+2\calL_{Y}(g)(\bN,T(\bN)^{\top})+\langle T^{\top},\calL_Y(g)\rangle_g\\
&=2T(\bN,\bN)\langle\nabla_{\bN}Y,\bN\rangle_g+2\langle \nabla_{\bN}Y,T(\bN)^{\top}\rangle_g+2\langle \nabla_{T(\bN)^{\top}}Y,\bN\rangle_g\\
&\quad +\langle T^{\top},\calL_Y(g)\rangle_g\\
&=2\left[\bN\cdot\left( T(\bN,\bN)\langle Y, \bN\rangle_g\right)-\langle Y,\nabla_{\bN}\bN\rangle_g T(\bN,\bN)-\bN\cdot\left(T(\bN,\bN)\right)\langle Y,\bN\rangle_g\right]\\
&\quad+2\Big[\bN\cdot\langle Y,T(\bN)^{\top}\rangle_g-\langle Y,\nabla_{\bN}T(\bN)^{\top}\rangle_g+T(\bN)^{\top}\cdot\langle Y,\bN\rangle_g\\&\qquad 
-\langle Y^{\top},\nabla_{T(\bN)^{\top}}\bN\rangle_g\Big] +\langle T^{\top},\calL_Y(g)\rangle_g\\
&=2\left[\bN\cdot\langle Y,T(\bN)\rangle_g-\langle Y,\nabla_{\bN}\bN\rangle_g T(\bN,\bN)-\bN\cdot\left(T(\bN,\bN)\right)\langle Y,\bN\rangle_g\right]\\
&\quad+2\left[-\langle Y,\nabla_{\bN}T(\bN)^{\top}\rangle_g+T(\bN)^{\top}\cdot\langle Y,\bN\rangle_g-\langle Y^{\top},\nabla_{T(\bN)^{\top}}\bN\rangle_g\right]\\
&\quad+\langle T^{\top},\calL_Y(g)\rangle_g.\label{comput-1-T-lie}
\end{split}
\end{equation}
Now, by decomposing $Y=\langle Y,\bN\rangle_g\bN+Y^{\top}$ and $g=\frac{1}{|\nabla\br|^2}d\br\otimes d\br+g_{\br}$,
\begin{equation}
\begin{split}
\langle T^{\top},\calL_Y(g)\rangle_g&=\langle T^{\top},\calL_{Y^{\top}}(g)\rangle_g+\langle T^{\top},\calL_{\langle Y,\bN\rangle_g\bN}(g)\rangle_g\\
&=\frac{1}{|\nabla\br|^2}\langle T^{\top},\calL_{Y^{\top}}(d\br)\otimes d\br+d\br\otimes\calL_{Y^{\top}}(d\br)\rangle_g+\langle T^{\top},\calL_{Y^{\top}}(g_{\br})\rangle_{g_{\br}}\\
&\quad+\langle Y,\bN\rangle_g\langle T^{\top},\calL_{\bN}(g)\rangle_g+\langle T^{\top},d(\langle Y,\bN\rangle_g)\otimes \langle\bN,\cdot\rangle_g\\
&\quad +\langle\bN,\cdot\rangle_g\otimes d(\langle Y,\bN\rangle_g)\rangle_g\\
&=\langle T^{\top},\calL_{Y^{\top}}(g_{\br})\rangle_{g_{\br}}+\langle Y,\bN\rangle_g\langle T^{\top},\calL_{\bN}(g)\rangle_g.\label{comput-2-T-lie}
\end{split}
\end{equation}
Estimates (\ref{comput-1-T-lie}) and (\ref{comput-2-T-lie}) lead to the expected identity.

As for the integral estimate, notice that $|\nabla(|\nabla\br|^{-1})|_g=O(\br^{-1})$ and similarly, $|\calL_{\bN}(g)|_g=O(\br^{-1})$ and $|\nabla_{\bN}\bN|=O(\br^{-1})$. Use integration by parts with respect to the induced metric $g_{\br}$ to obtain:
\begin{equation*}
\begin{split}
\Bigg|\int_{S_{\br}}\frac{\langle T,\calL_Y(g)\rangle_g-2\bN\cdot\langle Y,T(\bN)\rangle_g}{|\nabla\br|}\Bigg|&\leq C\int_{S_{\br}}|Y|_g\left(\br^{-1}|T|_g+|\nabla T|_g\right)\\
&\quad+\Bigg|\int_{S_{\br}}\frac{\langle T^{\top},\calL_{Y^{\top}}(g_{\br})\rangle_{g_{\br}}}{|\nabla\br|}\Bigg|+2\Bigg|\int_{S_{\br}}\frac{T(\bN)^{\top}\cdot\langle Y,\bN\rangle_g}{|\nabla\br|}\Bigg|\\
&\leq C\int_{S_{\br}}|Y|_g\left(\br^{-1}|T|_g+|\nabla T|_g\right)\\
&\quad+2\Bigg|\int_{S_{\br}}\langle\div_{g_{\br}}\left(|\nabla\br|^{-1}T^{\top}\right),Y^{\top}\rangle_{g_{\br}}\Bigg|\\
&\quad+2\Bigg|\int_{S_{\br}}\div_{g_{\br}}\left(|\nabla\br|^{-1}T(\bN)^{\top}\right)\cdot\langle Y,\bN\rangle_g\Bigg|\\
&\leq C\int_{S_{\br}}|Y|_g\left(\br^{-1}|T|_g+|\nabla T|_g\right)\\
&\quad+C\int_{S_{\br}}\left(\left|\div_{g_{\br}}T^{\top}\right|_{g_{\br}}+\left|\div_{g_{\br}}\left(T(\bN)^{\top}\right)\right|\right)|Y|_g.
\end{split}
\end{equation*}
Now, since $|\nabla^{g_{\br}}_UV-\nabla^g_UV|\leq C\br^{-1}|U|_g|V|_g$ for any two vector fields $U,V$ tangent to $S_{\br}$, one gets the expected integral estimate.
 \end{proof}
 
 We are in a position to state and prove the first main result of this section: the following proposition establishes an a priori growth on the frequency function $N(\rho)$. More precisely:
 
  \begin{prop}{\rm (\cite[Proposition 4.4]{Ber-Asym-Struct})}\label{prop-4-4-Ber}
Let $(h,\calB)\in C^2(E_R)$ satisfy \eqref{eq:main.1}--\eqref{est-remain-term} and \eqref{ass-prim-dec-h-B-record}. Assume $h$ is non-trivial. \\
Then there exist positive constants $C$, $\tilde{C}$ and $\tilde{R}\geq R$ such that if $\rho\geq \tilde{R}$,
 \begin{equation}
\left(N-\tilde{C}N_{\calB}\right)'(\rho)\leq\frac{1}{4} \rho^{-1} N(\rho) +C\rho^{-2} N(\rho)+C.\label{good-enough-inequ-N}
\end{equation}
In particular, $|N(\rho)|\leq C\rho$  and $N_{\calB}(\rho)\leq C\rho^{-4}$ for $\rho\geq R$.

\end{prop}
 
  \begin{proof}
  Thanks to (\ref{id-F-D-L}) together with Corollary \ref{coro-syst-cons}, the co-area formula gives us:
  \begin{equation}
\begin{split}\label{prelim-F-est-der}
\hat{F}'(\rho)&=\hat{D}'(\rho)-\int_{S_{\rho}}\frac{\langle\hat{h},\calL_{-2n}\hat{h}\rangle}{|\nabla \br|}\,\Phi_{-2n}\\
&=\hat{D}'(\rho)-\int_{S_{\rho}}\frac{\langle\hat{h},\calL_{\calB}(g)\rangle}{|\nabla \br|}\,\br^{-n}-\int_{S_{\rho}}\frac{\langle\hat{h},R[\hat{h}]\rangle}{|\nabla \br|}\,\Phi_{-2n}.
\end{split}
\end{equation}
Observe that Cauchy-Schwarz inequality applied to the last integral on the righthand side of \eqref{prelim-F-est-der} gives:
  \begin{equation}
\begin{split}\label{remain-est-der-N}
\left|\int_{S_{\rho}}\frac{\langle\hat{h},R[\hat{h}]\rangle}{|\nabla \br|}\,\Phi_{-2n}\right|&\leq \frac{C}{\rho^2}\hat{B}(\rho)+\frac{C}{\rho^2}\left(\hat{B}(\rho)\int_{S_{\rho}}|\nabla\hat{h}|^2\,\Phi_{-2n}\right)^{\frac{1}{2}}.
\end{split}
\end{equation}
Now, according to [\eqref{comp-formula-der-D}, Proposition \ref{prop-mix-4-2-4-3}] and Lemma \ref{lemma-3-3-Ber-b},
\begin{equation}
\begin{split}
\hat{D}'(\rho)&\leq -2\int_{S_{\rho}}\frac{|\nabla_{\bN}\hat{h}|^2}{|\nabla \br|}\,\Phi_{-2n}-\left(\frac{n+2}{\rho}+\frac{\rho}{2}\right)\hat{D}(\rho)\\
&\quad+\frac{C}{\rho^2}\left(\hat{D}(\rho)+\rho^{-1}\hat{B}(\rho)\right)\\
&\quad+ \frac{C}{\rho}\left(\int_{S_{\rho}}|\nabla\hat{h}|^2\Phi_{-2n}\right)^{\frac{1}{2}}\left(\int_{S_{\rho}}\br^2|\calB|^2\,\Psi_0\right)^{\frac{1}{2}}\\
&\leq-2\int_{S_{\rho}}\frac{|\nabla_{\bN}\hat{h}|^2}{|\nabla \br|}\,\Phi_{-2n}-\left(\frac{n+2}{\rho}+\frac{\rho}{2}\right)\hat{F}(\rho)\\
&\quad+\frac{1}{8\rho}\left(\hat{F}(\rho)+C\rho^{-1}\hat{B}(\rho)\right)+\frac{C}{\rho^2}\left(\hat{F}(\rho)+\rho^{-1}\hat{B}(\rho)\right)\\
&\quad+ \frac{C}{\rho}\left(\int_{S_{\rho}}|\nabla\hat{h}|^2\Phi_{-2n}\right)^{\frac{1}{2}}\left(\rho^2\Psi_0(\rho)B_{\calB}(\rho)\right)^{\frac{1}{2}}.\label{est-der-D-inequ}
\end{split}
\end{equation}
Here we have used \eqref{id-F-D-L}, together with [\eqref{2-3-3-Ber}, Lemma \ref{lemma-3-3-Ber-a}] in the second inequality.
  
  Plugging \eqref{remain-est-der-N} and \eqref{est-der-D-inequ} into \eqref{prelim-F-est-der} gives:
  \begin{equation*}
  \begin{split}
  \hat{F}'(\rho)&\leq -2\int_{S_{\rho}}\frac{|\nabla_{\bN}\hat{h}|^2}{|\nabla \br|}\,\Phi_{-2n}-\left(\frac{n+2}{\rho}+\frac{\rho}{2}\right)\hat{F}(\rho)\\
&\quad+\frac{1}{8\rho}\left(\hat{F}(\rho)+C\rho^{-1}\hat{B}(\rho)\right)+\frac{C}{\rho^2}\left(\hat{F}(\rho)+\rho^{-1}\hat{B}(\rho)\right)\\
&\quad+ \frac{C}{\rho}\left(\int_{S_{\rho}}|\nabla\hat{h}|^2\Phi_{-2n}\right)^{\frac{1}{2}}\left[\left(\rho^2\Psi_0(\rho)B_{\calB}(\rho)\right)^{\frac{1}{2}}+\frac{\hat{B}^{\frac{1}{2}}(\rho)}{\rho}\right]-\int_{S_{\rho}}\frac{\langle\hat{h},\calL_{\calB}(g)\rangle}{|\nabla \br|}\,\br^{-n}.
  \end{split}
  \end{equation*}
  Now, thanks to Lemma \ref{lemma-Ber-3-2},
  
  \begin{align*}
\hat{B}(\rho)N'(\rho)&=\rho\hat{F}'(\rho)-\rho\hat{F}(\rho)\frac{\hat{B}'(\rho)}{\hat{B}(\rho)}+\hat{F}(\rho)\\
&\leq -2\rho\int_{S_{\rho}}\frac{|\nabla_{\bN}\hat{h}|^2}{|\nabla \br|}\,\Phi_{-2n}+2\rho\frac{\hat{F}^2(\rho)}{\hat{B}(\rho)}+O(\rho^{-2})\hat{F}(\rho)\\
&\quad+\frac{1}{8}\left(\hat{F}(\rho)+C\rho^{-1}\hat{B}(\rho)\right)+\frac{C}{\rho}\left(\hat{F}(\rho)+\rho^{-1}\hat{B}(\rho)\right)\\
&\quad+ C\left(\int_{S_{\rho}}|\nabla\hat{h}|^2\Phi_{-2n}\right)^{\frac{1}{2}}\left[\left(\rho^2\Psi_0(\rho)B_{\calB}(\rho)\right)^{\frac{1}{2}}+\frac{\hat{B}^{\frac{1}{2}}(\rho)}{\rho}\right]-\rho\int_{S_{\rho}}\frac{\langle\hat{h},\calL_{\calB}(g)\rangle}{|\nabla \br|}\,\br^{-n}\\
&\leq\frac{1}{8}\left(\hat{F}(\rho)+C\rho^{-1}\hat{B}(\rho)\right)+\frac{C}{\rho}\left(\hat{F}(\rho)+\rho^{-1}\hat{B}(\rho)\right)\\
&\quad+ C\left(\int_{S_{\rho}}|\nabla\hat{h}|^2\Phi_{-2n}\right)^{\frac{1}{2}}\left[\left(\rho^2\Psi_0(\rho)B_{\calB}(\rho)\right)^{\frac{1}{2}}+\frac{\hat{B}^{\frac{1}{2}}(\rho)}{\rho}\right]-\rho\int_{S_{\rho}}\frac{\langle\hat{h},\calL_{\calB}(g)\rangle}{|\nabla \br|}\,\br^{-n}.
\end{align*}
In particular, dividing the previous estimate by $\hat{B}(\rho)$ gives us:
 \begin{align*}
N'(\rho)&\leq \frac{1}{8}\frac{N(\rho)}{\rho}+C\frac{N(\rho)}{\rho^2}+C\left(\frac{\int_{S_{\rho}}|\nabla\hat{h}|^2\Phi_{-2n}}{\hat{B}(\rho)}\right)^{\frac{1}{2}}\left(N_{\calB}^{\frac{1}{2}}(\rho)+\frac{1}{\rho}\right)+\frac{C}{\rho}\\
&\quad-\frac{\rho}{\hat{B}(\rho)}\int_{S_{\rho}}\frac{\langle\hat{h},\calL_{\calB}(g)\rangle}{|\nabla \br|}\,\br^{-n}.
\end{align*}
Invoking Lemma \ref{lemma-IBP-Lie-level-set} applied to the symmetric $2$-tensor $\hat{h}$ and the vector field $\calB$, one gets:
 \begin{equation}
\begin{split}\label{comput-der-N-intermezzo}
N'(\rho)&\leq \frac{1}{8}\frac{N(\rho)}{\rho}+C\frac{N(\rho)}{\rho^2}+C\left(\frac{\int_{S_{\rho}}|\nabla\hat{h}|^2\Phi_{-2n}}{\hat{B}(\rho)}\right)^{\frac{1}{2}}\left(N_{\calB}^{\frac{1}{2}}(\rho)+\frac{1}{\rho}\right)+\frac{C}{\rho}\\
&\quad+C\frac{\rho^{1-n}}{\hat{B}(\rho)}\int_{S_{\br}}|\calB|_g\left(\br^{-1}|\hat{h}|_g+|\nabla \hat{h}|_g\right)
-\frac{\rho^{1-n}}{\hat{B}(\rho)}\int_{S_{\rho}}\frac{2\bN\cdot\langle \calB,\hat{h}(\bN)\rangle_g}{|\nabla\br|}.
\end{split}
\end{equation}
We note that using \eqref{eq:main.2} we can write
\begin{equation*}
\begin{split}\label{comput-der-N-intermezzo.2}
 \bN\cdot\langle \calB,\hat{h}(\bN)\rangle_g &= \frac{2}{\br|\nabla^g\br|} \langle \nabla_{\nabla f}\calB,\hat{h}(\bN)\rangle_g + \langle \calB,\nabla_\bN\hat{h}(\bN)\rangle_g + \langle \calB,\hat{h}(\nabla_\bN\bN)\rangle_g\\
 &= O(\br^{-1}|\hat{h}| +   |\nabla \hat{h}|) |\calB|+ O( \br^{-4} |h|  + \br^{-3}|\nabla h|)|\hat{h}|\\
 &=O(\br^{-1}|\hat{h}| +   |\nabla \hat{h}|) |\calB|+ O( \br^{-2} |\hat{h}|  + \br^{-3}|\nabla \hat{h}|)|\hat{h}|\Phi_{-2n}.
\end{split}
\end{equation*}
We can thus estimate
\begin{equation}
\begin{split}\label{comput-der-N-intermezzo.3}
\frac{\rho^{1-n}}{\hat{B}(\rho)} \left| \int_{S_{\rho}}\frac{2\bN\cdot\langle \calB,\hat{h}(\bN)\rangle_g}{|\nabla\br|} \right| &\leq C\rho^{-n-1} + C \rho^{-2-n}\left(\frac{\int_{S_{\rho}}|\nabla\hat{h}|^2\Phi_{-2n}}{\hat{B}(\rho)}\right)^{\frac{1}{2}}\\
&\quad +\frac{C\rho^{1-n}}{\hat{B}(\rho)}\int_{S_{\br}}|\calB|_g\left(\br^{-1}|\hat{h}|+|\nabla \hat{h}|\right)\\
&\leq C\rho^{-1} + C \rho^{-2}\left(\frac{\int_{S_{\rho}}|\nabla\hat{h}|^2\Phi_{-2n}}{\hat{B}(\rho)}\right)^{\frac{1}{2}}\\
&\quad +C\left(\frac{\int_{S_{\rho}}|\nabla \hat{h}|^2\,\Phi_{-2n}}{\hat{B}(\rho)}+O(\rho^{-2})\right)^{\frac{1}{2}}N_{\calB}^{\frac{1}{2}}(\rho).
\end{split}
\end{equation}
This implies, using \eqref{comput-der-N-intermezzo},
\begin{align*}
N'(\rho)&\leq \frac{1}{8}\frac{N(\rho)}{\rho}+C\frac{N(\rho)}{\rho^2}+C\left(\left(\frac{\int_{S_{\rho}}|\nabla\hat{h}|^2\Phi_{-2n}}{\hat{B}(\rho)}\right)^{\frac{1}{2}}+O(\rho^{-1})\right)N_{\calB}^{\frac{1}{2}}(\rho)\\
&\quad + C\rho^{-1} + C \rho^{-1}\left(\frac{\int_{S_{\rho}}|\nabla\hat{h}|^2\Phi_{-2n}}{\hat{B}(\rho)}\right)^{\frac{1}{2}} ,
\end{align*}
where we have applied Cauchy-Schwarz inequality in the second line of \eqref{comput-der-N-intermezzo}. 

Let us assume that 
\begin{equation}
\hat{B}(\rho)\int_{S_{\rho}}|\nabla\hat{h}|^2\,\Phi_{-2n}\leq 4\left(\hat{F}(\rho)+\frac{\rho}{4}\hat{B}(\rho)\right)^2.\label{ass-1-hard}
\end{equation}
Using this assumption we arrive at
\begin{equation}
\begin{split}\label{first-ass-hard-N-prime}
N'(\rho) &\leq \frac{1}{8} \rho^{-1} N(\rho) +C\rho^{-2} N(\rho)+ C \rho^{-1} + C \left( \rho^{-1}N(\rho) + \rho\right) \left(N_\calB^\frac{1}{2}(\rho)+C\rho^{-1}\right)\\
&\leq\frac{1}{4} \rho^{-1} N(\rho) + C  + C  \rho^{-1}N(\rho)  N_\calB(\rho)+C\rho N_\calB^\frac{1}{2}(\rho)\\
&\leq\frac{1}{4} \rho^{-1} N(\rho) + C + C  \rho^{-1}N(\rho)  N_\calB(\rho)+C\rho \left(\rho^{-6}N(\rho)+\rho^{-4}\right)^{\frac{1}{2}}\\
&\leq\frac{1}{4} \rho^{-1} N(\rho) +C\rho^{-2} N(\rho)+C  \rho^{-1}N(\rho)  N_\calB(\rho)+C ,
\end{split}
\end{equation}
where $C$ is a positive constant that may vary from line to line.

Here we have used Young's inequality in the second and last line, the bound \eqref{comp-fre-fcts} is used in the third line and we have made constant use of the fact that $N(\rho)\geq -C\rho^{-2}$ thanks to \eqref{diff-fre-fct-N-N-m}.
Now, under the assumption \eqref{ass-1-hard} and according to \eqref{comp-fre-fcts-der} and the bound \eqref{comp-fre-fcts},
 \begin{equation}
 \begin{split}\label{first-ass-hard-N-Bianchi-prime}
 N'_{\calB}(\rho)&\geq-C\rho^{-2}N_{\calB}(\rho)+2\frac{N(\rho)}{\rho}N_{\calB}(\rho)\\
 &\quad-C\left(\rho^{-4}\left(\frac{N(\rho)}{\rho}+\frac{\rho}{4}\right)^2+\rho^{-2}\right)^{\frac{1}{2}}N_{\calB}^{\frac{1}{2}}(\rho)\\
 &\geq-C\rho^{-2}N_{\calB}(\rho)+2\frac{N(\rho)}{\rho}N_{\calB}(\rho)-C\left(\rho^{-3}N(\rho)+\rho^{-1}\right)N_{\calB}^{\frac{1}{2}}(\rho)\\
 &\geq -C\rho^{-1}N_{\calB}(\rho)+\frac{N(\rho)}{\rho}N_{\calB}(\rho)-C\rho^{-3}N(\rho)-C\rho^{-1}\\
 &\geq -C\rho^{-2} N(\rho)+\frac{N(\rho)}{\rho}N_{\calB}(\rho)-C\rho^{-1}.
 \end{split}
 \end{equation}
 Combining \eqref{first-ass-hard-N-prime} together with \eqref{first-ass-hard-N-Bianchi-prime} leads to:
 \begin{equation}
\left(N-CN_{\calB}\right)'(\rho)\leq\frac{1}{4} \rho^{-1} N(\rho) +C\rho^{-2} N(\rho)+C.\label{good-enough-inequ-ass-hard}
\end{equation}
Assume now that \eqref{ass-1-hard} does not hold, i.e.
\begin{equation}
\hat{B}(\rho)\int_{S_{\rho}}|\nabla\hat{h}|^2\,\Phi_{-2n}\geq 4\left(\hat{F}(\rho)+\frac{\rho}{4}\hat{B}(\rho)\right)^2.\label{ass-2-easy}
\end{equation}
Going back to the identity \eqref{id-F-D-L} and using [\eqref{eq:eqhhat}, Corollary \ref{coro-syst-cons}]:
\begin{equation*}
\begin{split}
\hat{F}'(\rho)=-\int_{S_{\rho}}\frac{|\nabla\hat{h}|^2}{|\nabla\br|}\,\Phi_{-2n}-\int_{S_{\rho}}\frac{\langle\hat{h},\calL_{\calB}(g)\rangle}{|\nabla\br|}\,\br^{-n}-\int_{S_{\rho}}\frac{\langle\hat{h},R[\hat{h}]\rangle}{|\nabla \br|}\,\Phi_{-2n}.
\end{split}
\end{equation*}
Applying Lemma \ref{lemma-IBP-Lie-level-set} to the symmetric $2$-tensor $\hat{h}$ and the vector field $\calB$, and using \eqref{remain-est-der-N}, \eqref{comput-der-N-intermezzo.3},
 \begin{align*}
\rho\frac{\hat{F}'(\rho)}{\hat{B}(\rho)}&\leq-\frac{\rho}{\hat{B}(\rho)}\int_{S_{\rho}}\frac{|\nabla\hat{h}|^2}{|\nabla\br|}\,\Phi_{-2n}+C\rho^{-1} + C \rho^{-1}\left(\frac{\int_{S_{\rho}}|\nabla\hat{h}|^2\Phi_{-2n}}{\hat{B}(\rho)}\right)^{\frac{1}{2}}\\
&\quad +C\left(\frac{\int_{S_{\rho}}|\nabla \hat{h}|^2\,\Phi_{-2n}}{\hat{B}(\rho)}+O(\rho^{-1})\right)^{\frac{1}{2}}N_{\calB}^{\frac{1}{2}}(\rho)\\
&\leq-\frac{7\rho}{8\hat{B}(\rho)}\int_{S_{\rho}}\frac{|\nabla\hat{h}|^2}{|\nabla\br|}\,\Phi_{-2n}+C\rho^{-1}N_{\calB}(\rho)+C\rho^{-\frac{1}{2}}N_{\calB}^{\frac{1}{2}}(\rho)+C\rho^{-1}.
\end{align*}
Here we have used Young's inequality in the second line.
In particular, thanks to Lemma \ref{lemma-Ber-3-2},
\begin{equation}
\begin{split}\label{prelim-N-diff-inequ-case-2}
N'(\rho)&\leq-\frac{7\rho}{8\hat{B}(\rho)}\int_{S_{\rho}}\frac{|\nabla\hat{h}|^2}{|\nabla\br|}\,\Phi_{-2n}+ C\rho^{-1}(1+ \rho^{1/2} N_{\calB}^{\frac{1}{2}}(\rho) + N_{\calB}(\rho))\\
& \quad+\rho\frac{\hat{F}(\rho)}{\hat{B}(\rho)}\left(\frac{\rho}{2}+\frac{n+2}{\rho}+2\frac{\hat{F}(\rho)}{\hat{B}(\rho)}+ O(\rho^{-3})\right)\\
&\leq -3\rho\left(\frac{\hat{F}(\rho)}{\hat{B}(\rho)}+\frac{\rho}{4}\right)^2-\frac{\rho}{8\hat{B}(\rho)}\int_{S_{\rho}}\frac{|\nabla\hat{h}|^2}{|\nabla\br|}\,\Phi_{-2n}\\
&\quad+C\rho^{-1}(1+ \rho^{1/2} N_{\calB}^{\frac{1}{2}}(\rho)+ N_{\calB}(\rho))\\
&\quad+\rho\frac{\hat{F}(\rho)}{\hat{B}(\rho)}\left(\frac{\rho}{2}+\frac{n+2}{\rho}+2\frac{\hat{F}(\rho)}{\hat{B}(\rho)}+ O(\rho^{-3}) \right)\\
&=-\rho\left(\frac{\hat{F}(\rho)}{\hat{B}(\rho)}\right)^2-\frac{\hat{F}(\rho)}{\hat{B}(\rho)}\left(\rho^2-(n+2)\right) - \frac{3}{16} \rho^3 + O(\rho^{-3}) N(\rho)\\
&\quad-\frac{\rho}{8\hat{B}(\rho)}\int_{S_{\rho}}\frac{|\nabla\hat{h}|^2}{|\nabla\br|}\,\Phi_{-2n}+C\rho^{-1}(1+ \rho^{1/2} N_{\calB}^{\frac{1}{2}}(\rho)+ N_{\calB}(\rho))\\
&\leq-\frac{\rho}{2}\left(\frac{\hat{F}(\rho)}{\hat{B}(\rho)}\right)^2-\frac{\rho}{8\hat{B}(\rho)}\int_{S_{\rho}}\frac{|\nabla\hat{h}|^2}{|\nabla\br|}\,\Phi_{-2n}+C\rho^{-2}N(\rho).
\end{split}
\end{equation}
Here we have invoked assumption \eqref{ass-2-easy} in the second line and we have made constant use of the fact that $N(\rho)\geq -C\rho^{-2}$ thanks to \eqref{diff-fre-fct-N-N-m}. Finally the bound \eqref{comp-fre-fcts} is used in the last line.

Therefore, for any positive constant $\tilde{C}$, one has thanks to \eqref{prelim-N-diff-inequ-case-2} and \eqref{comp-fre-fcts-der}:
\begin{align}
\left(N-\tilde{C}N_{\calB}\right)'(\rho)&\leq-\frac{\rho}{2}\left(\frac{\hat{F}(\rho)}{\hat{B}(\rho)}\right)^2-\frac{\rho}{8\hat{B}(\rho)}\int_{S_{\rho}}\frac{|\nabla\hat{h}|^2}{|\nabla\br|}\,\Phi_{-2n}-2\tilde{C}\rho^{-1}N(\rho)N_{\calB}(\rho)\notag \\ \label{good-enough-inequ-ass-easy}
 &\quad+C\rho^{-2}N(\rho)+\tilde{C}\left(C \rho^{-4}\frac{\int_{S_{\rho}}|\nabla \hat{h}|^2\,\Phi_{-2n}}{\hat{B}(\rho)}+C\rho^{-2}\right)^{\frac{1}{2}}N_{\calB}^{\frac{1}{2}}(\rho)\\
 &\quad +C\rho^{-1} \leq C\rho^{-2}N(\rho)+C\rho^{-1}, \notag
\end{align}
where $C$ is a positive constant that may vary from line to line.

Combining differential inequalities \eqref{good-enough-inequ-ass-hard} and \eqref{good-enough-inequ-ass-easy} give the first expected result \eqref{good-enough-inequ-N}.

Integrating  \eqref{good-enough-inequ-N} between $\rho_0$ and $\rho$ gives in turn:
\begin{align*}
N(\rho)&\leq \tilde{C}N_{\calB}(\rho)+C(\hat{h},\calB,\rho_0)+\frac{1}{4} \int_{\rho_0}^{\rho}\frac{N(s)}{s}\,ds+C\int_{\rho_0}^{\rho}\frac{N(s)}{s^2}\,ds+C\rho\\
&\leq \frac{1}{4} \int_{\rho_0}^{\rho}\frac{N(s)}{s}\,ds+C\frac{N(\rho)}{\rho^6}+C(\hat{h},\calB,\rho_0)+C\int_{\rho_0}^{\rho}\frac{N(s)}{s^2}\,ds+C\rho,
\end{align*}
where we have used [\eqref{comp-fre-fcts}, Proposition \ref{prop-never-stops}].

In particular, if $\rho_0$ is sufficiently large,
\begin{equation*}
\begin{split}
N(\rho)\leq C(\hat{h},\calB,\rho_0)+\frac{1}{4}\int_{\rho_0}^{\rho}\frac{N(s)}{s}\,ds+C\int_{\rho_0}^{\rho}\frac{N(s)}{s^2}\,ds+C\rho.
\end{split}
\end{equation*}
Gr\"onwall's inequality applied to the function $\int_{\rho_0}^{\rho}\frac{N(s)}{s}\,ds$ give the expected a priori growth on the frequency function $N(\rho)\leq C\rho$ for $\rho\geq R$. Using [\eqref{comp-fre-fcts}, Proposition \ref{prop-never-stops}] again proves the a priori bound on $N_{\calB}$.
  \end{proof}
  
   \begin{coro}{\rm (\cite[Corollary 4.5]{Ber-Asym-Struct})}\label{coro-4-5-Ber}
Let $(h,\calB)\in C^2(E_R)$ satisfy \eqref{eq:main.1}--\eqref{est-remain-term} and \eqref{ass-prim-dec-h-B-record}. Assume $h$ is non-trivial. \\
Then there exist positive constants $C$ and $\tilde{R}\geq R$ such that if $\rho\geq \tilde{R}$,
\begin{equation}
\begin{split}\label{beautiful-diff-inequ-N-hat}
\hat{N}'(\rho)\leq-\frac{1}{2}\frac{\int_{\rho}^{\infty}t\hat{D}(t)\,dt}{\hat{B}(\rho)}+C\rho^{-2}\left(\hat{N}^{\frac{1}{2}}(\rho)+\hat{N}(\rho)+1\right).
\end{split}
\end{equation}
Moreover, 
\begin{equation}
\begin{split}\label{more-beautiful-diff-inequ-N-hat}
 \hat{N}'(\rho)+\frac{2}{\rho}\hat{N}(\rho)&\leq \frac{1}{1+C\rho^{-3}\hat{N}(\rho)^{-\frac{1}{2}}}\left(\frac{2}{\rho}\hat{N}(\rho)-\frac{1}{\hat{B}(\rho)}\int_{\rho}^{\infty}t\hat{D}(t)\,dt\right)\\
&\quad + \frac{C}{\rho^2}\left(\hat{N}(\rho)+\hat{N}^{\frac{1}{2}}(\rho)+\rho^{-3}\right).
 \end{split}
\end{equation}
  \end{coro}
  \begin{rk}
  The reason of stating \eqref{more-beautiful-diff-inequ-N-hat} is made clearer in the proof of Proposition \ref{prop-5-4-Ber} on the decay of the frequency function $\hat{N}$. Roughly speaking, one must keep track of the coefficient in front of the good term $\int_{\rho}^{\infty}t\hat{D}(t)\,dt$: the coefficient $\frac{1}{2}$ in \eqref{beautiful-diff-inequ-N-hat} is not enough for our purpose.
  \end{rk}
  \begin{proof}
  According to [\eqref{comp-formula-der-N-m}, Proposition \ref{prop-mix-4-2-4-3}],
  \begin{align*}
\hat{N}'(\rho)&\leq-\frac{1}{\hat{B}(\rho)}\int_{\rho}^{\infty}t\hat{D}(t)\,dt+\frac{2}{\hat{B}(\rho)}\Bigg|\int_{E_{\rho}}\langle\nabla_{\bX}\hat{h},\calL_{-2n}\hat{h}\rangle\,\Phi_{-2n}\Bigg|\\
&\quad+\frac{2|N(\rho)|}{\hat{B}(\rho)}\Bigg|\int_{E_{\rho}}\langle\hat{h},\calL_{-2n}\hat{h}\rangle\,\Phi_{-2n}\Bigg|+C\rho^{-3}\hat{N}(\rho)+C\rho^{-4}\hat{N}^{\frac{1}{2}}(\rho).
\end{align*}
Thanks to [\eqref{1-3-3-Ber}, Lemma \ref{lemma-3-3-Ber-a}] and Lemma \ref{lemma-3-3-Ber-b},
 \begin{align}
\hat{N}'(\rho)&\leq -\frac{1}{\hat{B}(\rho)}\int_{\rho}^{\infty}t\hat{D}(t)\,dt+\frac{C}{\rho^2}\left(\hat{N}(\rho)+\hat{N}^{\frac{1}{2}}(\rho)\right)\notag \\
&\quad+C\left(-\frac{\hat{D}'(\rho)}{\hat{B}(\rho)}\right)^{\frac{1}{2}}N^{\frac{1}{2}}_{\calB}(\rho)
+C\rho^{-3}|N(\rho)|\hat{N}(\rho) \notag \\
&\quad+C\rho^{-3}|N(\rho)|+C\rho^{-3}\hat{N}(\rho)+C\rho^{-4}\hat{N}^{\frac{1}{2}}(\rho) \label{diff-inequ-N-hat-prelim}\\
&\leq -\frac{1}{\hat{B}(\rho)}\int_{\rho}^{\infty}t\hat{D}(t)\,dt+ \frac{C}{\rho^2}\left(\hat{N}(\rho)+\hat{N}^{\frac{1}{2}}(\rho)+\rho^{-3}\right)\notag \\
&\quad+C\left(-\frac{\hat{D}'(\rho)}{\hat{B}(\rho)}\right)^{\frac{1}{2}}N^{\frac{1}{2}}_{\calB}(\rho).\notag
\end{align}
Here we have used that due to \eqref{diff-fre-fct-N-N-m} it holds that $|N(\rho)| \leq \hat{N}(\rho) + O(\rho^{-2})$. Note that by the very definition of $\hat{N}$:
\begin{equation}
\begin{split}\label{disguise-ineq-N-N-prime}
\left(-\frac{\hat{D}'(\rho)}{\hat{B}(\rho)}\right)&=-\frac{\left(\rho^{-1}\hat{B}\hat{N}\right)'(\rho)}{\hat{B}(\rho)}\\
&=\rho^{-2}\hat{N}(\rho)-\rho^{-1}\hat{N}'(\rho)-\rho^{-1}\frac{\hat{B}'(\rho)}{\hat{B}(\rho)}\hat{N}(\rho)\\
&\leq-\rho^{-1}\hat{N}'(\rho)+ \rho^{-1}\left(\frac{n+2}{\rho}+\frac{\rho}{2}+\frac{2}{\rho}N(\rho)+C\rho^{-1}\right)\hat{N}(\rho)\\
&\leq -\rho^{-1}\hat{N}'(\rho)+ C\hat{N}(\rho).
\end{split}
\end{equation}
Here we have used Lemma \ref{lemma-Ber-3-2} in the third inequality together with Proposition \ref{prop-4-4-Ber}.

Now, thanks to [\eqref{comp-fre-fcts}, Proposition \ref{prop-never-stops}], for $\varepsilon >0$ we have
\begin{equation*}
\begin{split}\label{alg-manip-D-N}
\left(-\frac{\hat{D}'(\rho)}{\hat{B}(\rho)}\right)^{\frac{1}{2}}N^{\frac{1}{2}}_{\calB}(\rho)&\leq C\left(-\frac{\hat{D}'(\rho)}{\hat{B}(\rho)}\right)^{\frac{1}{2}}\left(\rho^{-6}N(\rho)+\rho^{-4}\right)^{\frac{1}{2}}\\
&\leq \varepsilon\left(-\frac{\hat{D}'(\rho)}{\hat{B}(\rho)}\right)+C\varepsilon^{-1}\rho^{-4}\\
&\leq C\varepsilon\hat{N}(\rho)+C\varepsilon^{-1}\rho^{-4}-\varepsilon\rho^{-1}\hat{N}'(\rho),\end{split}
\end{equation*}
where we used \eqref{disguise-ineq-N-N-prime} in the last line. Plugging this back into \eqref{diff-inequ-N-hat-prelim} we see that
\begin{equation}\label{diff-inequ-N-hat-intermed}
\begin{split}
(1+C \varepsilon\rho^{-1}) \hat{N}'(\rho)&\leq -\frac{1}{\hat{B}(\rho)}\int_{\rho}^{\infty}t\hat{D}(t)\,dt\\
&\quad + \frac{C}{\rho^2}\left(\hat{N}(\rho)+\hat{N}^{\frac{1}{2}}(\rho)+\rho^{-3}\right)\\
&\quad+C\varepsilon\hat{N}(\rho)+C\varepsilon^{-1}\rho^{-4}\, .
\end{split}
\end{equation}
As a first result, let us choose $\varepsilon:=\rho^{-2}$ in \eqref{diff-inequ-N-hat-intermed} to get:
\begin{equation*}
\begin{split}
\hat{N}'(\rho)&\leq -\frac{1}{2\hat{B}(\rho)}\int_{\rho}^{\infty}t\hat{D}(t)\,dt\\
&\quad + \frac{C}{\rho^2}\left(\hat{N}(\rho)+\hat{N}^{\frac{1}{2}}(\rho)+1\right).
\end{split}
\end{equation*}
This proves \eqref{beautiful-diff-inequ-N-hat}.

Now, by introducing the term $2\rho^{-1}\hat{N}(\rho)$ artificially and by dividing the inequality \eqref{diff-inequ-N-hat-intermed} by $1+C \varepsilon\rho^{-1}$, one gets:
\begin{equation*}
\begin{split}
 \hat{N}'(\rho)+\frac{2}{\rho+C \varepsilon}\hat{N}(\rho)&\leq \frac{1}{1+C \varepsilon\rho^{-1}}\left(\frac{2}{\rho}\hat{N}(\rho)-\frac{1}{\hat{B}(\rho)}\int_{\rho}^{\infty}t\hat{D}(t)\,dt\right)\\
&\quad + \frac{C}{\rho^2}\left(\hat{N}(\rho)+\hat{N}^{\frac{1}{2}}(\rho)+\rho^{-3}\right)\\
&\quad+C\varepsilon\hat{N}(\rho)+C\varepsilon^{-1}\rho^{-4}\, .
\end{split}
\end{equation*}
Observe that:
\begin{align*}
 \hat{N}'(\rho)+\frac{2}{\rho}\hat{N}(\rho)&\leq \frac{1}{1+C \varepsilon\rho^{-1}}\left(\frac{2}{\rho}\hat{N}(\rho)-\frac{1}{\hat{B}(\rho)}\int_{\rho}^{\infty}t\hat{D}(t)\,dt\right)\\
&\quad + \frac{C}{\rho^2}\left(\hat{N}(\rho)+\hat{N}^{\frac{1}{2}}(\rho)+\rho^{-3}\right)+\frac{2}{\rho}\left(1-\frac{1}{1+C\varepsilon\rho^{-1}}\right)\hat{N}(\rho)\\
&\quad+C\varepsilon\hat{N}(\rho)+C\varepsilon^{-1}\rho^{-4}\\
&\leq \frac{1}{1+C \varepsilon\rho^{-1}}\left(\frac{2}{\rho}\hat{N}(\rho)-\frac{1}{\hat{B}(\rho)}\int_{\rho}^{\infty}t\hat{D}(t)\,dt\right)\\
&\quad + \frac{C}{\rho^2}\left(\hat{N}(\rho)+\hat{N}^{\frac{1}{2}}(\rho)+\rho^{-3}\right)+\frac{2C\varepsilon}{\rho^2}\hat{N}(\rho)\\
&\quad+C\varepsilon\hat{N}(\rho)+C\varepsilon^{-1}\rho^{-4}.
\end{align*}
Finally, let us choose $\varepsilon:=\rho^{-2}\hat{N}^{-\frac{1}{2}}_{-2n}(\rho)$ (if $\hat{N}(\rho)=0$ then $\hat{N}(\rho')=0$ for $\rho'\geq\rho$ and there is nothing to prove) to obtain:
\begin{equation*}
\begin{split}
 \hat{N}'(\rho)+\frac{2}{\rho}\hat{N}(\rho)&\leq \frac{1}{1+C\rho^{-3}\hat{N}^{-\frac{1}{2}}(\rho)}\left(\frac{2}{\rho}\hat{N}(\rho)-\frac{1}{\hat{B}(\rho)}\int_{\rho}^{\infty}t\hat{D}(t)\,dt\right)\\
&\quad + \frac{C}{\rho^2}\left(\hat{N}(\rho)+\hat{N}^{\frac{1}{2}}(\rho)+\rho^{-3}\right).
 \end{split}
\end{equation*}
This proves \eqref{more-beautiful-diff-inequ-N-hat}.

\end{proof}

\section{Frequency decay}\label{sec-fre-dec}

We start this section with the proof of the main preliminary result on the limiting behavior of the frequency function $\hat{N}$:
   \begin{theo}{\rm (\cite[Theorem 4.1]{Ber-Asym-Struct})}\label{theo-4-1-Ber}
Let $(h,\calB)\in C^2(E_R)$ satisfy \eqref{eq:main.1}--\eqref{est-remain-term} and \eqref{ass-prim-dec-h-B-record}. Assume $h$ is non-trivial. Then, $$\lim_{\rho\rightarrow+\infty}\hat{N}(\rho)=\lim_{\rho\rightarrow+\infty}N(\rho)=0.$$
\end{theo}
\begin{proof}
In case $\lim_{\rho\rightarrow+\infty}\hat{N}(\rho)=0$ then Lemma \ref{lemma-3-3-Ber-a} implies $\lim_{\rho\rightarrow+\infty}N(\rho)=0$.

Now, [\eqref{beautiful-diff-inequ-N-hat}, Corollary \ref{coro-4-5-Ber}] implies that the function $\rho\rightarrow e^{\frac{C}{\rho}}(\hat{N}(\rho)+1)$ is non-increasing for some large positive constant $C$. This already ensures the finiteness of the limit of $\hat{N}(\rho)$ as $\rho$ tends to $+\infty$, that we denote here by $N_{\infty}$. By the non-negativity of $\hat{N}$, one has $N_{\infty}\geq 0$. Then to show that $N_{\infty}=0$ goes along the same lines as in the proof of \cite[Theorem 4.1]{Ber-Asym-Struct} using Lemma \ref{lemma-Ber-3-4}. 
\end{proof}

The next proposition discretizes the problem of showing some decay on the frequency function. 

\begin{prop}{\rm (\cite[Proposition 5.4]{Ber-Asym-Struct})}\label{prop-5-4-Ber}
Let $(h,\calB)\in C^2(E_R)$ satisfy \eqref{eq:main.1}--\eqref{est-remain-term} and \eqref{ass-prim-dec-h-B-record}. Assume $h$ is non-trivial and assume 
\begin{equation}\label{eq:decay-assume}
\hat{N}(\rho)\leq\eta\rho^{2\gamma},
\end{equation}
 for $\rho\geq R$ and some $\gamma\in[-1,0]$. Then there exist constants $R' \geq R$ and $C \geq 0$, depending on $h$, so that, for $\rho \geq R'$, either 
\begin{equation}
\begin{split}
&(1)\ \hat{N}(\rho+2) - \hat{N}(\rho) \leq - 4 \rho^{-1}\hat{N}(\rho), \text{ or}\\
&(2)\ \hat{N}(\rho+1) - \hat{N}(\rho) \leq  - 2 \rho^{-1}\hat{N}(\rho) + C \rho^{-2 + \gamma}.\label{eq:decay-assume-est}
\end{split}
\end{equation}
\end{prop}
\begin{proof} By \eqref{beautiful-diff-inequ-N-hat} and Theorem \ref{theo-4-1-Ber} and $\gamma \in [-1,0]$, there exists $R' \geq R$ and $\kappa \geq 0$ such that $s\geq R'$ together with \eqref{eq:decay-assume} implies that
$$ \hat{N}'(s) \leq  \frac{C}{s^2}\left(\hat{N}(s)+\hat{N}^{\frac{1}{2}}(s)+s^{-3}\right) \leq  \frac{1}{10}\kappa s^{-2+\gamma}\, .$$
As $\gamma \leq 0$, the mean value inequality gives for any $\tau \in [0,2]$ and $s\geq R'$,
\begin{equation}\label{eq-5-2-Ber}
\hat{N}(s+\tau) \leq \hat{N}(s) + \frac{1}{5}\kappa s^{-2 + \gamma}\, .
\end{equation}
Suppose case $(1)$ does not hold for a given $\rho$, that is 
$$\hat{N}(\rho+2) - \hat{N}(\rho) > - 4 \rho^{-1}\hat{N}(\rho)\, .$$
We first claim that for all $\tau \in [0,2]$
\begin{equation}\label{eq-5-3-Ber}
\hat{N}(\rho+\tau) - \hat{N}(\rho) \geq - 4 \rho^{-1}\hat{N}(\rho) - \kappa \rho^{-2 + \gamma}\, .
\end{equation}
Indeed if \eqref{eq-5-3-Ber} fails for some $\tau_0 \in [0,2]$, then by \eqref{eq-5-2-Ber} applied at $s = \rho+\tau_0$
\begin{equation*}
\begin{split}
\hat{N}(\rho+2) &= \hat{N}(\rho+\tau_0+(2-\tau_0))\leq \hat{N}(\rho+\tau_0) + \frac{1}{5} \kappa (\rho+\tau_0)^{-2+\gamma}\\
&\leq \hat{N}(\rho+\tau_0) + \frac{1}{5} \kappa \rho^{-2+\gamma} < \hat{N}(\rho) - 4 \rho^{-1}  \hat{N}(\rho)  + \frac{1}{5} \kappa \rho^{-2+\gamma}-\kappa \rho^{-2 + \gamma}\\
&< \hat{N}(\rho) - 4 \rho^{-1}  \hat{N}(\rho)\, .
\end{split}
\end{equation*}
This contradicts the assumption that case $(1)$ does not hold at $\rho$, proving the claim.

If $R'$ is sufficiently large, Lemma \ref{lemma-Ber-3-4} together with the fact that $\hat{N}(\rho)\leq \eta/2$ for all $\rho \geq R'$ guarantees that for $t\in [\rho,\rho+2]$,
$$(1-2\eta\rho^{-1}) B(\rho) \leq B(t)\, .$$
Thus by \eqref{eq-5-3-Ber}, if $t\in[\rho,\rho+2]$,
\begin{equation*}
\begin{split}
tD(t) &\geq \left(1-\frac{2\eta}{\rho}\right) \hat{N}(t) B(\rho) \geq \left(1-\frac{2\eta}{\rho}\right) \left(\left(1-\frac{4}{\rho}\right) \hat{N}(\rho) - \kappa \rho^{-2+\gamma}\right) B(\rho)\\
&\geq (1-(2\eta + 4) \rho^{-1})) \hat{N}(\rho) B(\rho) - \kappa \rho^{-2+\gamma} B(\rho) 
\end{split}
\end{equation*}
Hence, using \eqref{eq-Ber-3-2} and \eqref{eq:decay-assume} we see that for $C$ sufficiently large,
\begin{equation*}
\begin{split}
\int_{\rho}^\infty t \hat{D}(t)\, dt &\geq \int_\rho^{\rho+1} t D(t) \Phi_{-2n}(t)\, dt\\
& \geq \left(1-\frac{2\eta+4}{\rho}\right) \hat{N}(\rho) B(\rho) \int_{\rho}^{\rho+1} \Phi_{-2n}(t)\, dt\\
& \quad- \kappa \rho^{-2+\gamma} B(\rho) \int_{\rho}^{\rho+1} \Phi_{-2n}(t)\, dt \\
&\geq \frac{2}{\rho} \hat{N}(\rho) \hat{B}(\rho) -  C \rho^{-2+2\gamma}\hat{B}(\rho)\, .
\end{split}
\end{equation*}

The previous estimate together with [\eqref{more-beautiful-diff-inequ-N-hat}, Corollary \ref{coro-4-5-Ber}] and \eqref{eq:decay-assume} yields
\begin{equation*}
\hat{N}'(\rho)+\frac{2}{\rho}\hat{N}(\rho)\leq C \rho^{-2+\gamma},
\end{equation*}
for some positive constant $C$ that may vary from line to line.

Note that in the last estimate $\rho$ can be replaced by $\rho+\tau$ for any $\tau \in [0,1]$. Hence by \eqref{eq-5-3-Ber}, for any $\tau \in [0,1]$,
\begin{equation*}
\begin{split}
\hat{N}'(\rho+\tau) &\leq -2 (\rho+\tau)^{-1} \hat{N}(\rho+\tau) + C(\rho+\tau)^{-2+\gamma}\\
&\leq -2 \rho^{-1} \hat{N}(\rho+\tau) + C\rho^{-2+\gamma} \\
&\leq -2 \rho^{-1} \hat{N}(\rho) +8 \rho^{-2} \hat{N}(\rho)+ C\rho^{-2+\gamma}\\
&\leq -2 \rho^{-1} \hat{N}(\rho) + C \rho^{-2+\gamma},
\end{split}
\end{equation*}
which leads to the second estimate [$(2)$, \eqref{eq:decay-assume-est}] after integration between $0$ and $1$.
\end{proof}

We are now in a position to prove the main result of this section:
\begin{theo}{\rm (\cite[Theorem 5.1]{Ber-Asym-Struct})}\label{theo-5-1-Ber}
Let $(h,\calB)\in C^2(E_R)$ satisfy \eqref{eq:main.1}--\eqref{est-remain-term} and \eqref{ass-prim-dec-h-B-record}. Assume $h$ is non-trivial. Then for $\varepsilon>0$, there exists a positive constant $C_{\varepsilon}$ such that if $\rho\geq R$:
$$\hat{N}(\rho)\leq \frac{C_{\varepsilon}}{\rho^{2-\varepsilon}}.$$
\end{theo}
\begin{proof}
We proceed as in the proof of \cite[Theorem $5.1$]{Ber-Asym-Struct}. For this purpose, define the rescaled function $\calN(\rho):=\rho^2\hat{N}(\rho)$.

We assume that \eqref{eq:decay-assume} holds. Assume Case (1) of Proposition \ref{prop-5-4-Ber} holds true at $\rho$ then a straightforward computation shows:
\begin{equation*}
\calN(\rho+2)-\calN(\rho)\leq 0.
\end{equation*}
Now, if Case (2) of Proposition \ref{prop-5-4-Ber} holds then:
\begin{equation*}
\begin{split}
\calN(\rho+1)-\calN(\rho)&=(\rho+1)^2\hat{N}(\rho+1)-\rho^2\hat{N}(\rho)\\
&\leq -2\rho\hat{N}(\rho) + C \rho^{ \gamma}+2\rho \hat{N}(\rho+1)+\hat{N}(\rho+1)\\
&=2\rho\left(\hat{N}(\rho+1)-\hat{N}(\rho)\right)+\hat{N}(\rho+1) +C \rho^{ \gamma}\\
&\leq2\rho\left(- 2 \rho^{-1}\hat{N}(\rho) + C \rho^{-2 + \gamma}\right)+\hat{N}(\rho+1) +C \rho^{ \gamma}\\
&\leq C\rho^{\gamma},
\end{split}
\end{equation*}
where $C$ is a positive constant that may vary from line to line.

For each $i\geq 0$, let $$\min_{\rho\in[R+i,R+i+2]}\calN(\rho)=:\calN^{-}_i\leq \calN^{+}_i:=\max_{\rho\in[R+i,R+i+2]}\calN(\rho).$$
According to Theorem \ref{theo-4-1-Ber}, there is an $\eta>0$ such that \eqref{eq:decay-assume} holds with $\gamma=0$. Therefore, in either case of Proposition \ref{prop-5-4-Ber}, one gets:
\begin{equation*}
\calN^{+}_{i+1}\leq \calN^{+}_{i}+C.
\end{equation*}
In particular, if $\rho\in[R+i,R+i+2]$, then
\begin{equation*}
\rho^2\hat{N}(\rho)\leq  \calN^{+}_{i}\leq \calN^{+}_{0}+C\rho.
\end{equation*}
This shows \eqref{eq:decay-assume} holds true with $\gamma=-\frac{1}{2}$.

Injecting this improved decay back to Case (2) of Proposition \ref{prop-5-4-Ber}, the previous reasoning gives us:
\begin{equation*}
\calN^{+}_{i+1}\leq \calN^{+}_{i}+\frac{C}{\sqrt{R+i}},\quad i\geq 0.
\end{equation*}
This leads us to: $\rho^2\hat{N}(\rho)\leq \calN^{+}_{0}+C\rho^{\frac{1}{2}}$, for $\rho\geq R$, i.e. \eqref{eq:decay-assume} holds true with $\gamma=-\frac{3}{4}$.

Given $\varepsilon\in(0,1)$, one gets the expected decay on $\hat{N}$ by iterating finitely many times the previous arguments.
\end{proof}

\section{Decay estimates and traces at infinity}\label{sec-dec-est-tra-inf}
\subsection{Existence of traces at infinity}~~\\

The main result of this section is the following theorem which corresponds to {\rm (\cite[Theorem 6.1]{Ber-Asym-Struct})}.

\begin{theo}\label{theo-6-1-Ber}
Let $(h,\calB)\in C^2(E_R)$ satisfy \eqref{eq:main.1}--\eqref{est-remain-term} and \eqref{ass-prim-dec-h-B-record}. Then the following assertions hold true.
\begin{enumerate}
\item The limit $\lim_{\rightarrow+\infty}\rho^{1-n}B(\rho)$ exists and is zero if and only if $h\equiv 0$ on $E_R$.
\item There is a positive constant $C$ such that for $\rho\geq R$,
\begin{equation*}
\int_{E_{\rho}}|\hat{h}|^2\br^{-1-n}\leq \frac{C}{\rho^{n}}\int_{S_{\rho}}|\hat{h}|^2,
\end{equation*}
 and for all $\varepsilon\in(0,1)$, there is a positive constant $C_{\varepsilon}$ and a radius $R_{\varepsilon}\geq R$ such that the following decay estimates hold for $\rho\geq R_{\varepsilon}$:
\begin{equation}
\int_{E_{\rho}}\left(\br^2|\nabla \hat{h}|^2+\br^4|\nabla_{\partial_{\br}}\hat{h}|^2\right)\br^{-1-n}\leq \frac{C_{\varepsilon}}{\rho^{n-\varepsilon}}\int_{S_{\rho}}|\hat{h}|^2.\label{est-int-rad-der-der-hat-h}
\end{equation}
\item 
Moreover, the radial limit $\lim_{r\rightarrow+\infty}\hat{h}=:\tr_{\infty}\hat{h}$ exists in the $L^2(g_S)$-topology 
and:
\begin{equation}
\int_{E_{\rho}}\left(\br^2\left(|\hat{h}-\tr_{\infty}\hat{h}|^2+|\nabla \hat{h}|^2\right)+\br^4|\nabla_{\partial_{\br}}\hat{h}|^2\right)\br^{-2-n}\leq \frac{C_{\varepsilon}\|\tr_{\infty}\hat{h}\|_{L^2(S)}^2}{\rho^{2-\varepsilon}},\quad \rho\geq R_{\varepsilon}.\label{conv-rate-tr-inf}
\end{equation}
\end{enumerate}
\end{theo}

\begin{rk}
Notice that the integral estimate on $\hat{h}$ is optimal and the integral gradient estimate is sharp up to $\varepsilon$. This explains that we cannot guarantee the trace at infinity $\tr_{\infty}\hat{h}$ to be an element of $H^1(S)$ unlike in \cite{Ber-Asym-Struct}.
\end{rk}
\begin{rk}
The integral radial estimate does not seem to be sharp since $\hat{h}$ is asymptotically $0$-homogeneous, i.e. the $k$-th covariant derivatives of $\hat{h}$ are expected to decay like $\br^{-k}$ at least in the integral sense for $k\geq 0$ if the asymptotic cone is smooth outside the origin. Then by (\eqref{eq:main.1},\eqref{eq:main.2}), one can see that the radial derivative of $\hat{h}$ should decay like $\br^{-6}$. Nonetheless, estimate \eqref{est-int-rad-der-der-hat-h} will be sufficient for our applications.
\end{rk}
\begin{rk}
If $(C(S),dr^2+r^2g_S,o)$ is Ricci flat then it can be shown that the trace at infinity of $\hat{h}$ defines a smooth symmetric $2$-tensor on the link $S$ which is radially invariant: see \cite{Uni-Con-Egs-Der}.
\end{rk}
\begin{rk}\label{rem-6-1-Ber}
As a final remark, thanks to Section \ref{ode-pde-sec}, Theorem \ref{theo-6-1-Ber} can be applied to each metric $g_{\sigma}$ and potential function $f_{\sigma}$, $\sigma\in[1,2]$, and the estimates are uniform in $\sigma\in[1,2]$. This fact will be needed in Section \ref{sec-rel-ent}.
\end{rk}

\begin{proof}
According to Lemma \ref{lemma-Ber-3-2} and Theorem \ref{theo-5-1-Ber}, for any $\varepsilon\in(0,1)$,
\begin{equation*}
\frac{d}{d\rho}\rho^{1-n}B(\rho)=O_{\varepsilon}(\rho^{-3+\varepsilon})\rho^{1-n}B(\rho),\label{eqn-diff-B-concl}
\end{equation*}
where $O_{\varepsilon}$ depends on $\varepsilon$. Since the function $\rho\rightarrow \rho^{-3+\varepsilon}$ is integrable on the half-line $[R,+\infty)$, equation \eqref{eqn-diff-B-concl} tells us that the limit $\lim_{\rho\rightarrow+\infty}\rho^{1-n}B(\rho)=:B_{\infty}$ exists. Moreover, $B_{\infty}=0$ if and only if $B(\rho)=0$ for some $\rho\geq R$ if and only if $h\equiv0$ on $E_{R}$ thanks to Corollary \ref{lemma-3-5-Ber}.

From now on, we assume $B_{\infty}>0$, i.e. $h$ is non-trivial on $E_{R}$. By the co-area formula, if $R$ is large enough so that $2^{-1}B_{\infty}\leq\rho^{-n+1}B(\rho)\leq 2B_{\infty}$ for $\rho\geq R$:
\begin{equation*}
\int_{E_{\rho}}|\hat{h}|^2\br^{-1-n}\leq 2\int_{\rho}^{+\infty}r^{-1-n}B(r)\,dr\leq 4B_{\infty}\int_{\rho}^{+\infty}r^{-2}\,dr\leq 8\rho^{-n}B(\rho).
\end{equation*}
Let us estimate the covariant derivatives of $\hat{h}$ as follows: by Theorem \ref{theo-5-1-Ber},
\begin{equation}
D(\rho)=\rho^{-1}\hat{N}(\rho)B(\rho)\leq C_{\varepsilon}B_{\infty}\rho^{n-4+\varepsilon},\quad\rho\geq R.\label{opt-int-est-D-pol}
\end{equation}
Differentiating the function $\rho^{1-n}D(\rho)$ by using the co-area formula only:
\begin{equation}
\frac{d}{d\rho}\left(\rho^{1-n}D(\rho)\right)= \left(\frac{\rho}{2}+\frac{n+1}{\rho}\right)\rho^{1-n}D(\rho)-\rho^{1-n}\int_{S_{\rho}}\frac{|\nabla \hat{h}|^2}{|\nabla\br|}.\label{der-D-resc-est}
\end{equation}
In particular, by integrating the previous relation between $\rho\geq R$ and $+\infty$ and by invoking \eqref{opt-int-est-D-pol},
\begin{equation*}
\begin{split}
\int_{E_{\rho}}\br^{1-n}|\nabla\hat{h}|^2&\leq C_{\varepsilon}B_{\infty}\rho^{-1+\varepsilon}\\
&\leq C_{\varepsilon}\rho^{-n+\varepsilon}B(\rho),\quad \rho\geq R,\label{opt-int-est-cov-der-pol}
\end{split}
\end{equation*}
where $C_{\varepsilon}$ is a positive constant that may vary from line to line.

Let us now prove the integral estimate on the radial derivative of $\hat{h}$ in \eqref{est-int-rad-der-der-hat-h}. 

According to Proposition \ref{prop-mix-4-2-4-3},
\begin{align*}
(\rho^{3-n}D)'(\rho)&=-2\rho^{3-n}\int_{S_{\rho}}\frac{|\nabla_{\bN}\hat{h}|^2}{|\nabla\br|}-2\rho^{2-n}\Phi_{-2n}^{-1}(\rho)\int_{E_{\rho}}\langle\nabla_{X}\hat{h},\calL_{-2n}\hat{h}\rangle\,\Phi_{-2n}\\
&\quad+\rho^{2-n}D(\rho)-\rho^{2-n}\Phi_{-2n}^{-1}(\rho)\int_{\rho}^{+\infty}t\hat{D}(t)\,dt\\
&\quad+O(\rho^{-3})\rho^{3-n}D(\rho)+O\left(\rho^{-\frac{3}{2}-n}\Phi_{-2n}^{-1}(\rho)\right)\hat{D}^{\frac{1}{2}}(\rho)\hat{B}^{\frac{1}{2}}(\rho)\\
&\leq -2\rho^{3-n}\int_{S_{\rho}}\frac{|\nabla_{\bN}\hat{h}|^2}{|\nabla\br|}-2\rho^{2-n}\Phi_{-2n}^{-1}(\rho)\int_{E_{\rho}}\langle\nabla_{X}\hat{h},\calL_{-2n}\hat{h}\rangle\,\Phi_{-2n}\\
&\quad+C_{\varepsilon}B_{\infty}\rho^{-2+\varepsilon}+O\left(\rho^{-\frac{3}{2}-n}\Phi_{-2n}^{-1}(\rho)\right)\hat{D}^{\frac{1}{2}}(\rho)\hat{B}^{\frac{1}{2}}(\rho).
\end{align*}
Here we have used \eqref{opt-int-est-D-pol} in the last line.

Now, thanks to Lemma \ref{lemma-3-3-Ber-b},
\begin{equation}
\begin{split}
(\rho^{3-n}D)'(\rho)&\leq-2\rho^{3-n}\int_{S_{\rho}}\frac{|\nabla_{\bN}\hat{h}|^2}{|\nabla\br|}+C\rho^{1-n}\Phi_{-2n}^{-1}(\rho)\left(\hat{D}(\rho)+\frac{1}{\rho^{\frac{1}{2}}}\hat{D}^{\frac{1}{2}}_{-2n}(\rho)\hat{B}^{\frac{1}{2}}_{-2n}(\rho)\right)\\
&\quad +C\rho^{2-n}\Phi_{-2n}^{-1}(\rho)\left(\int_{S_{\rho}}|\nabla\hat{h}|^2\Phi_{-2n}\right)^{\frac{1}{2}}\left(\int_{S_{\rho}}\br^2|\calB|^2\,\Psi_0\right)^{\frac{1}{2}}\\
&\quad+C_{\varepsilon}B_{\infty}\rho^{-2+\varepsilon}+C\rho^{-\frac{3}{2}-n}\Phi_{-2n}^{-1}(\rho)\hat{D}^{\frac{1}{2}}(\rho)\hat{B}^{\frac{1}{2}}(\rho)\\
&\leq-2\rho^{3-n}\int_{S_{\rho}}\frac{|\nabla_{\bN}\hat{h}|^2}{|\nabla\br|}+C\rho^{-n}\left(\hat{N}(\rho)+\hat{N}^{\frac{1}{2}}(\rho)\right)B(\rho)\\
&\quad+ C\rho^{2-n}N_{\calB}^{\frac{1}{2}}(\rho)B^{\frac{1}{2}}(\rho)\left(\int_{S_{\rho}}|\nabla\hat{h}|^2\right)^{\frac{1}{2}}+C_{\varepsilon}B_{\infty}\rho^{-2+\varepsilon}\\
&\leq-2\rho^{3-n}\int_{S_{\rho}}\frac{|\nabla_{\bN}\hat{h}|^2}{|\nabla\br|}+C_{\varepsilon}B_{\infty}\rho^{-2+\varepsilon}+ C\rho^{-n}\left(\int_{S_{\rho}}|\nabla\hat{h}|^2\right)^{\frac{1}{2}}B^{\frac{1}{2}}(\rho),\label{it-never-ends}
\end{split}
\end{equation}
where we have used $N_{\calB}(\rho)=O(\rho^{-4})$ in the last line thanks to Proposition \ref{prop-4-4-Ber} together with Theorem \ref{theo-5-1-Ber} and the upper bound on $B(\rho)$. In order to handle the energy of $\hat{h}$ on the level set $S_{\rho}$, we use \eqref{der-D-resc-est} once more to deduce that:
\begin{equation*}
\begin{split}
\frac{1}{2}\int_{S_{\rho}}|\nabla \hat{h}|^2&\leq \rho D(\rho)-\rho^{n-1}(\rho^{1-n}D)'(\rho)\\
&\leq C_{\varepsilon}B_{\infty}\rho^{n-3+\varepsilon}-\rho^{n-3}(\rho^{3-n}D)'(\rho),
\end{split}
\end{equation*}
where we have used \eqref{opt-int-est-D-pol} in the second line. In particular, by injecting this estimate back to \eqref{it-never-ends} and using Young's inequality give us for any $\eta>0$, 
\begin{equation*}
\begin{split}
(\rho^{3-n}D)'(\rho)&\leq-2\rho^{3-n}\int_{S_{\rho}}\frac{|\nabla_{\bN}\hat{h}|^2}{|\nabla\br|}+C_{\varepsilon}B_{\infty}\rho^{-2+\varepsilon}\\
&\quad+ \eta\left(C_{\varepsilon}B_{\infty}\rho^{-3+\varepsilon}-C\rho^{-3}(\rho^{3-n}D)'(\rho)\right)+CB_{\infty}\eta^{-1}\rho^{-1}.
\end{split}
\end{equation*}
Therefore, by choosing $\eta:=\rho^{1-\frac{\varepsilon}{2}}$, we get for $\rho$ sufficiently large so that $1+C\rho^{-2-\frac{\varepsilon}{2}}\leq 2$:
\begin{equation}
\begin{split}
(\rho^{3-n}D)'(\rho)&\leq-\rho^{3-n}\int_{S_{\rho}}\frac{|\nabla_{\bN}\hat{h}|^2}{|\nabla\br|}+C_{\varepsilon}B_{\infty}\rho^{-2+\varepsilon}.\label{now-it-ends}
\end{split}
\end{equation}
Integrating \eqref{now-it-ends} between $\rho$ and $+\infty$ gives the expected integral decay on the radial derivative of $\hat{h}$ by noticing that $\lim_{\rho\rightarrow+\infty}\rho^{3-n}D(\rho)=0$ by \eqref{opt-int-est-D-pol} for $\varepsilon\in(0,1)$.

The proof of \eqref{conv-rate-tr-inf} and the fact that the radial limit of $\hat{h}$ exists in the $L^2(S)$-topology follow directly from the arguments in the proof of \cite[Appendix $A$]{Ber-Asym-Struct}, see also \cite[Theorem $6.1$]{Ber-Asym-Struct}.
\end{proof}
\begin{prop}\label{prop-bianchi-id-sequ-lim}
Let $(h,\calB)\in C^2(E_R)$ satisfy \eqref{eq:main.1}--\eqref{est-remain-term} and \eqref{ass-prim-dec-h-B-record}.
Then for any $\varepsilon\in(0,1)$, there is a positive constant $C_{\varepsilon}$ and a radius $R_{\varepsilon}\geq R$ such that the following decay estimate hold for $\rho\geq R_{\varepsilon}$:
\begin{equation}
\int_{E_{\rho}}\br^4\bigg|\hat{h}(\partial_{\br})-\frac{\tr \hat{h}}{2}\partial_{\br}\bigg|^2\br^{-1-n}\leq \frac{C_{\varepsilon}}{\rho^{n-\varepsilon}}\int_{S_{\rho}}|\hat{h}|^2.\label{est-int-bianchi-rad-h}
\end{equation}
In particular, for any $\varepsilon\in(0,1)$, there exist a positive constant $C_{\varepsilon}$ and a diverging sequence $(\rho_i)_i$ such that:
\begin{equation}
\begin{split}
\rho_i^{-n+1}\int_{S_{\rho_i}}\bigg|\hat{h}(\partial_{\br})-\frac{\tr \hat{h}}{2}\partial_{\br}\bigg|^2&\leq C_{\varepsilon}\rho_i^{-(4-\varepsilon)},\label{bianchi-id-sequ}
\end{split}
\end{equation}
and,
\begin{equation*}
\big(\tr_{\infty}\hat{h}\big)(\partial_{r})=\frac{\tr_{g_C}(\tr_{\infty}\hat{h})}{2}\partial_{r},\quad\text{on the link $S$.}\label{bianchi-id-at-inf}
\end{equation*}
\end{prop}

\begin{proof}
 Observe that
\begin{equation*}
\calB=\div h-\frac{\nabla \tr h}{2}=\Psi_n^{-1}\bigg(\div \hat{h}-\frac{\nabla \tr\hat{h}}{2}\bigg)-\frac{\Psi_n^{-1}}{2}\left(1+\frac{2n}{\br^2}\right)\bigg(\hat{h}(\br\partial_{\br})-\frac{\tr \hat{h}}{2}\br\partial_{\br}\bigg).
\end{equation*}
In particular, if $\rho\geq R$,
\begin{equation*}
\begin{split}
\int_{S_{\rho}}\bigg|\hat{h}(\partial_{\br})-\frac{\tr \hat{h}}{2}\partial_{\br}\bigg|^2&\leq C\int_{S_{\rho}}\br^{-2}|\calB|^2\Psi_n^2+C\rho^{-2}\int_{S_{\rho}}|\nabla \hat{h}|^2\\
&\leq C\rho^{-4}N_{\calB}(\rho)B(\rho)+C\rho^{-2}\int_{S_{\rho}}|\nabla \hat{h}|^2\\
&\leq C\rho^{-8}B(\rho)+C\rho^{-2}\int_{S_{\rho}}|\nabla \hat{h}|^2,
\end{split}
\end{equation*}
where $C$ is a positive constant that may vary from line to line. Here we have used Proposition \ref{prop-4-4-Ber} in the last line. Since $B(\rho)\leq 2B_{\infty} \rho^{n-1}$ for $\rho\geq R$, one gets immediately:
\begin{equation*}
\begin{split}
\rho^{-n+1}\int_{S_{\rho}}\bigg|\hat{h}(\partial_{\br})-\frac{\tr \hat{h}}{2}\partial_{\br}\bigg|^2&\leq CB_{\infty}\rho^{-8}+C\rho^{-2-(n-1)}\int_{S_{\rho}}|\nabla \hat{h}|^2,\quad \rho\geq R.
\end{split}
\end{equation*}
Multiplying this inequality by $\rho^2$ and integrating from $\rho$ to $+\infty$ leads to the proof of \eqref{est-int-bianchi-rad-h} once [\eqref{est-int-rad-der-der-hat-h}, Theorem \ref{theo-6-1-Ber}] is invoked.
Given $\varepsilon\in(0,1)$, estimate \eqref{est-int-rad-der-der-hat-h} guarantees the existence of a diverging sequence $(\rho_i)_i$ such that:
\begin{equation*}
\int_{S_{\rho_i}}|\nabla\hat{h}|^2\leq C_{\varepsilon}\rho_i^{n-3+\varepsilon}.
\end{equation*}
This proves \eqref{bianchi-id-sequ}. In particular, this ensures that
 \begin{equation*}
\lim_{\rho_i\rightarrow+\infty}\rho_i^{-n+1}\int_{S_{\rho_i}}\bigg|\hat{h}(\partial_{\br})-\frac{\tr \hat{h}}{2}\partial_{\br}\bigg|^2 =0.
\end{equation*}
Since the convergence of $\hat{h}$ to its trace at infinity $\tr_{\infty}\hat{h}$ holds in the $L^2(S)$-topology, the expected result follows.
\end{proof}

\begin{prop}\label{prop-div-free-tr-infty}
Let $(h,\calB)\in C^2(E_R)$ satisfy \eqref{eq:main.1}--\eqref{est-remain-term} and \eqref{ass-prim-dec-h-B-record}. Then,
\begin{equation*}
\div_{g_S}(\tr_{\infty}\hat{h})^{\top}=0,\quad\text{ in the weak sense.}
\end{equation*}
\end{prop}

\begin{proof}
Define the difference of the Ricci tensors by $\calR:=\Ric(g_2)-\Ric(g_1)$ and define the rescaled difference of the Ricci curvatures by $\hat{\calR}:=\Psi_{n-2}\calR.$ We will use the notation $g$ for $g_1$ as the reference metric in the sequel.

\begin{claim}\label{ident-claim-tr-Ric-met}
The radial limit $\lim_{r\rightarrow+\infty}\hat{\calR}=:\tr_{\infty}\hat{\calR}$ exists in the $L^2(S)$-topology and,
 $$\tr_{\infty}\hat{\calR}=-\frac{1}{8}\tr_{\infty}\hat{h}.$$
 Moreover, on $E_R$, there is some positive constant $C$ such that:
 \begin{equation}
|8\hat{\calR}+\hat{h}|\leq C\left(\br^{-2}|\hat{h}|+\br^{-1}|\nabla_{\partial_{\br}}\hat{h}|\right).\label{est-diff-ric-met-resc}
\end{equation}
\end{claim}
\begin{proof}[Proof of Claim \ref{ident-claim-tr-Ric-met}]
By using the soliton equation \eqref{eq:0-0} satisfied by both metrics $g_1=g$ and $g_2$:
\begin{equation}
\begin{split}
2\hat{\calR}&=2\Psi_{n-2}\left(\Ric(g_2)-\Ric(g_1)\right)=\Psi_{n-2}\left(-h+\calL_{\nabla f}h\right)\\
&=-\br^{-2}\hat{h}+\Psi_{n-2}\calL_{\nabla f}\left(\Phi_{-n}\hat{h}\right)=\br^{-2}\left(-\hat{h}+\calL_{\nabla f}\hat{h}\right)-\left(\frac{1}{4}+\frac{n}{2\br^2}\right)|\nabla \br|^2\hat{h}\\
&=-\frac{1}{4}\hat{h}+O(\br^{-2})\hat{h}+O(\br^{-1})\nabla_{\partial_{\br}}\hat{h},\label{bunch-equ-alg}
\end{split}
\end{equation}
where we have used \eqref{eq:5} and \eqref{eq:6} together with the fact that $$\calL_{\nabla f}\hat{h}=\nabla_{\nabla f}\hat{h}+\frac{1}{2}\left(\hat{h}\circ \calL_{\nabla f}(g)+\calL_{\nabla f}(g)\circ \hat{h}\right).$$

By Theorem \ref{theo-6-1-Ber} and \eqref{bunch-equ-alg}, we obtain the expected claim.
\end{proof}
In particular, Claim \ref{ident-claim-tr-Ric-met} ensures that it suffices to prove that $\div_{g_S}(\tr_{\infty}\hat{\calR})^{\top}=0$ in the weak sense.

To do so, we work on the end $E_R$ first and then we pass to the limit.

\begin{claim}\label{claim-div-resc-R}
\begin{equation*}
\begin{split}
\left|\div\hat{\calR}+\frac{(n-2)}{8\br}\hat{h}(\partial_{\br})\right|\leq C\left(\br^{-3}|\hat{h}|+\br^{-2}|\nabla_{\partial_{\br}}\hat{h}|+\br^{-4}|\nabla \hat h|\right)
\end{split}
\end{equation*}
\end{claim}

\begin{proof}[Proof of Claim \ref{claim-div-resc-R}]
By invoking \eqref{eq:2-0} in the following form $\div_{g}\Ric(g)+\Ric(g)(\nabla \br^2/4)=0$ valid for the metric $g_2$ as well, 
\begin{align*}
\div_g\hat{\calR}&=\div_g\left(\Psi_{n-2}(\Ric(g_2)-\Ric(g))\right)\\
&=\Psi_{n-2}\div_g\calR+\left(1+\frac{2(n-2)}{\br^2}\right)\hat{\calR}\left(\nabla\left(\frac{\br^2}{4}\right)\right)\\
&=\Psi_{n-2}\left(\div_{g}-\div_{g_2}\right)\Ric(g_2)+\frac{2(n-2)}{\br^2}\hat{\calR}\left(\nabla\left(\frac{\br^2}{4}\right)\right)\\
&=\frac{(n-2)}{\br}\hat{\calR}\left(\partial_{\br}\right)+\Psi_{n-2}\left(O(\br^{-3})|h|+O(\br^{-2})|\nabla h|\right)\\
&=\frac{(n-2)}{\br}\hat{\calR}\left(\partial_{\br}\right)+O(\br^{-3})|\hat{h}|+O(\br^{-4})|\nabla \hat{h}|.
\end{align*}
In particular, [\eqref{est-diff-ric-met-resc}, Claim \ref{ident-claim-tr-Ric-met}] gives the expected result.
\end{proof}

Let $V$ be a vector field on the link $S$ of the cone that we extend radially to the end $E_R$ so that $[\nabla f,V]=0$.
Let $g_{\br}$ be the metric on $S_{\br}$ induced by $g$. Notice that the norm of $V$ with respect to $g$ grows linearly as $\br$ tends to $+\infty$ and that  
\begin{equation}
\left<V,\bN\right>_{g}=(g-g_C)(V,\bN)=O(\br^{-2})|V|_{g}=O(\br^{-1}).\label{est-v-perp}
\end{equation}

\begin{claim}\label{claim-div-R-tan}
\begin{equation*}
\begin{split}
\left|\int_{S_{\rho}}\big\langle \div_{g_{\rho}}\hat{\calR}^{\top},V^{\top}\big\rangle_{g_{\rho}}+\bN\cdot\hat{\calR}(\bN,V^{\top})\right|\leq C\int_{S_{\rho}}\left(|\hat{h}^{\top}(\partial_{\br})|+\br^{-2}|\hat{h}|+\br^{-1}|\nabla_{\partial_{\br}}\hat{h}|+\br^{-3}|\nabla h|\right).
\end{split}
\end{equation*}
\end{claim}
\begin{proof}[Proof of Claim \ref{claim-div-R-tan}]
Let $(e_i)_{i=1,...,\, n-1}$ be an orthonormal frame for the metric $g_{\rho}$. By definition of the divergence of a symmetric $2$-tensor,
\begin{equation}
\begin{split}\label{est-sum-perp-calR-0}
\big\langle\div_{g}\hat{\calR},V^{\top}\big\rangle&= \big\langle\nabla_{\bN}\hat{\calR}_{\bN},V^{\top}\big\rangle+\sum_{e_i\perp\bN}\big\langle\nabla_{e_i}\hat{\calR}_{e_i},V^{\top}\big\rangle\\
&=\big\langle\nabla_{\bN}\hat{\calR}_{\bN},V^{\top}\big\rangle+\sum_{e_i\perp\bN}\big\langle\nabla_{e_i}\hat{\calR}^{\top}_{e_i},V^{\top}\big\rangle\\
&\quad+\big\langle\nabla_{e_i}\big(\hat{\calR}(\bN)\otimes \bN+\bN\otimes\hat{\calR}(\bN)\big)_{e_i},V^{\top}\big\rangle.
\end{split}
\end{equation} 
Now, since $\langle V^{\top},\bN\rangle_{g}=0$,
\begin{equation}\label{est-sum-perp-calR}
\bigg|\sum_{e_i\perp\bN}\Big\langle\nabla_{e_i}\left(\hat{\calR}(\bN)\otimes \bN+\bN\otimes\hat{\calR}(\bN)\right)_{e_i},V^{\top}\Big\rangle\bigg|\leq C\br^{-1}|\hat{\calR}(\bN)^{\top}||V|\leq C|\hat{\calR}(\bN)^{\top}|,
\end{equation}
and by using the fact that $\nabla^{g}_{e_i}e_j=\nabla^{g_{\br}}_{e_i}e_j-\br^{-1}\left(\delta_{ij}+O(\br^{-3})\right)\bN$ together with the definition of the divergence of a symmetric $2$-tensor:
\begin{equation}
\begin{split}\label{est-sum-perp-calR-bis}
\sum_{e_i\perp\bN}\big\langle\nabla_{e_i}\hat{\calR}^{\top}_{e_i},V^{\top}\big\rangle&=\sum_{e_i\perp\bN}e_i\cdot\hat{\calR}^{\top}(e_i,V^{\top})-\hat{\calR}^{\top}(\nabla^g_{e_i}e_i,V^{\top}) -\hat{\calR}^{\top}(e_i,\nabla^g_{e_i}V^{\top})\\
&=\sum_{e_i\perp\bN}e_i\cdot\hat{\calR}^{\top}(e_i,V^{\top})-\hat{\calR}^{\top}(\nabla^{g_{\br}}_{e_i}e_i,V^{\top})-\hat{\calR}^{\top}(e_i,\nabla^{g_{\br}}_{e_i}V^{\top})\\
&\quad+O(\br^{-1})|\hat{\calR}(\bN)^{\top}||V^{\top}|\\
&=\big\langle\div_{g_{\br}}\hat{\calR}^{\top},V^{\top}\big\rangle_{g_{\br}}+O(\br^{-1})|\hat{\calR}(\bN)^{\top}||V^{\top}|,
\end{split}
\end{equation}
Finally, since $[V,\nabla f]=0$, \eqref{eq:5} and \eqref{est-v-perp} give, 
\begin{equation*}
\begin{split}
\nabla_{\bN}V^{\top}&=\nabla_{\bN}\left(V-\langle V,\bN\rangle\bN\right)\\
&=\frac{1}{|\nabla f|}\nabla_V\nabla f-\bN\cdot\langle V,\bN\rangle\bN-\langle V,\bN\rangle\nabla_{\bN}\bN\\
&=\br^{-1}V+O(\br^{-2})=\br^{-1}V^{\top}+O(\br^{-2}),
\end{split}
\end{equation*}
and,
\begin{equation}
\begin{split}\label{est-sum-perp-calR-bis-bis}
\big\langle\nabla_{\bN}\hat{\calR}_{\bN},V^{\top}\big\rangle&=\bN\cdot\hat{\calR}(\bN,V^{\top})-\hat{\calR}(\nabla_{\bN}\bN,V^{\top})-\hat{\calR}(\bN,\nabla_{\bN}V^{\top})\\
&=\bN\cdot\hat{\calR}(\bN,V^{\top})+O(\br^{-2})|\hat{\calR}|+O(\br^{-1})|\hat{\calR}(\bN)^{\top}||V^{\top}|,
\end{split}
\end{equation}
where we have used \eqref{eq:5} in the second line.

Injecting \eqref{est-sum-perp-calR}, \eqref{est-sum-perp-calR-bis} and \eqref{est-sum-perp-calR-bis-bis} back to \eqref{est-sum-perp-calR-0} gives the expected claim, once we invoke [\eqref{est-diff-ric-met-resc}, Claim \ref{ident-claim-tr-Ric-met}] and Claim \ref{claim-div-resc-R}.
\end{proof}
We are in a position to conclude the proof of Proposition \ref{prop-div-free-tr-infty}.

On the one hand, by integration by parts on the level sets together with Claim \ref{ident-claim-tr-Ric-met}:
\begin{equation}
\begin{split}\label{easy-peasy-IBP-div-Lie}
\bigg|\int_{S_{\rho}}\big\langle \div_{g_{\rho}}\hat{\calR}^{\top},V^{\top}\big\rangle_{g_{\rho}}-\frac{1}{16}\big\langle \hat{h}^{\top},\calL_{V^{\top}}(g_{\rho})\big\rangle\bigg|&\leq C\int_{S_{\rho}}\br^{-2}|\hat{h}|+\br^{-1}|\nabla_{\partial_{\br}}\hat{h}|.
\end{split}
\end{equation}
On the other hand, by the co-area formula and integration by parts:
\begin{equation}
\begin{split}\label{est-int-rad-der-hat-R}
\int_{\rho_1}^{\rho_2}\rho^{-n+1}\int_{S_{\rho}}\bN\cdot\hat{\calR}(\bN,V^{\top})\,d\rho&=\int_{A_{\rho_1,\rho_2}}\big\langle\nabla\br,\nabla\hat{\calR}(\bN,V^{\top})\big\rangle\,\br^{-n+1}\\
&=\int_{S_{\rho}}|\nabla\br|\hat{\calR}(\bN,V^{\top})\,\br^{-n+1}\bigg|_{\rho_1}^{\rho_2}\\
&\quad-\int_{A_{\rho_1,\rho_2}}\div(\br^{-n+1}\nabla\br)\hat{\calR}(\bN,V^{\top}).
\end{split}
\end{equation}
In particular, by combining Claims \ref{ident-claim-tr-Ric-met} and \ref{claim-div-R-tan} together with \eqref{easy-peasy-IBP-div-Lie} and \eqref{est-int-rad-der-hat-R}:
\begin{equation}
\begin{split}\label{final-arg-div-free}
\int_{\rho_1}^{\rho_2}\rho^{-n+1}\bigg|\int_{S_{\rho}}\langle \hat{h}^{\top},\calL_{V^{\top}}(g_{\rho})\rangle\bigg|&\leq C\int_{A_{\rho_1,\rho_2}}\Big(|\hat{h}^{\top}(\partial_{\br})|+\br^{-2}|\hat{h}|+\br^{-1}|\nabla_{\partial_{\br}}\hat{h}|\\
&\qquad +\br^{-3}|\nabla h|\Big)\,\br^{-n+1}\\
&\quad +C\rho_1^{-n+2}\int_{S_{\rho_1}}|\hat{h}^{\top}(\partial_{\br})|+\br^{-2}|\hat{h}|+\br^{-1}|\nabla_{\partial_{\br}}\hat{h}|\\
&\quad+C\rho_2^{-n+2}\int_{S_{\rho_2}}|\hat{h}^{\top}(\partial_{\br})|+\br^{-2}|\hat{h}|+\br^{-1}|\nabla_{\partial_{\br}}\hat{h}|.
\end{split}
\end{equation}
Now, on the one hand, according to [\eqref{est-int-rad-der-der-hat-h}, Theorem \ref{theo-6-1-Ber}] and [\eqref{est-int-bianchi-rad-h}, \eqref{bianchi-id-sequ}, Proposition \ref{prop-bianchi-id-sequ-lim}], the righthand side of \eqref{final-arg-div-free} is bounded as $\rho_2$ tends to $+\infty$, at least sequentially. On the other hand, the integrand on the lefthand side converges to a finite value: more precisely,
\begin{equation*}
\lim_{\rho\rightarrow+\infty}\rho^{-n+1}\int_{S_{\rho}}\big\langle \hat{h}^{\top},\calL_{V^{\top}}(g_{\rho})\big\rangle=\int_{S}\big\langle \tr_{\infty}\hat{h}^{\top},\calL_{V}(g_{S})\big\rangle,
\end{equation*}
where $(S,g_S)$ is the link of the asymptotic cone of both expanders.
These two remarks together with \eqref{final-arg-div-free} imposes the vanishing of the previous limit. This ends the proof of Proposition \ref{prop-div-free-tr-infty}.
\end{proof}

As a corollary of Theorem \ref{theo-6-1-Ber}, we get the following non-optimal pointwise decay in space for higher derivatives of the difference of two expanding gradient Ricci solitons coming out of the same cone:
\begin{coro}\label{interpol-not-opt-dec}
Let $(h,\calB)\in C^2(E_R)$ satisfy \eqref{eq:main.1}--\eqref{est-remain-term} and \eqref{ass-prim-dec-h-B-record}. Then for any $\eta>0$ and $k\geq 1$, there exists a positive constant $C_{\eta,k}$ such:
\begin{equation*}
|\nabla^kh|\leq C_{\eta,k}e^{-(1-\eta)\frac{\br^2}{4}},\quad \text{on $E_R$}.
\end{equation*}
\end{coro}

\begin{proof}
The proof is a straightforward application of the following Gagliardo-Nirenberg interpolation inequalities. 

Let $j\in\mathbb{N}$ and $k\in \mathbb{N}$ such that $k>j$ and define $\alpha\in(0,1)$ by $\alpha:=\frac{2j+n}{2k+n}$. Then for any tensor $T$ compactly supported with support in a ball $B(x,2i_0)$ for some point $x\in M$ with $i_0$ strictly less than half the injectivity radius of the metric $g$,
\begin{equation}
\|\nabla^jT\|_{C^0}\leq C\|\nabla^kT\|_{C^0}^{\alpha}\|T\|_{L^2}^{1-\alpha},\label{gag-nir-int-inequ}
\end{equation}
for some positive constant $C$ independent of the point $x\in M$.

Let $x\in E_R$ such that $B(x,2i_0)$ is compactly contained in $E_R$ and let $\varphi:M\rightarrow[0,1]$ be a smooth cut-off function such that its support is in $B(x,2i_0)$ and such that $\varphi\equiv 1$ on $B(x,i_0)$. Then let us apply \eqref{gag-nir-int-inequ} to $T:=\varphi\cdot h$ to get:
\begin{equation*}
\begin{split}
\|\nabla^jh\|_{C^0(B(x,i_0))}&\leq C\|h\|_{C^k(B(x,2i_0))}^{\alpha}\|h\|_{L^2(B(x,2i_0))}^{1-\alpha}\label{gag-nir-int-inequ-bis}\\
&\leq C_{j,k}(h)\|h\|_{L^2(B(x,2i_0))}^{1-\alpha}.
\end{split}
\end{equation*}
Here we have used the fact that $h$ and its covariant derivatives are bounded on $E_R$ by Theorem \ref{ac-str-exp}.

Now, the fact that $B(x,2i_0)\subset E_{\br(x)-C}$ for some uniform positive constant $C=C(i_0)$ together with Theorem \ref{theo-6-1-Ber} give us for any positive $\varepsilon>0$:
\begin{equation}
\begin{split}
\|\nabla^jh\|_{C^0(B(x,i_0))}&\leq C_{j,k}(h)\bigg(\int_{E_{ \br(x)-C}}|h|^2\,\bigg)^{\frac{1-\alpha}{2}}\\
&= C_{j,k}(h)\bigg(\int_{E_{ \br(x)-C}}|\hat{h}|^2\,\br^{-2n}e^{-\frac{\br^2}{2}}\bigg)^{\frac{1-\alpha}{2}}\\
&\leq C_{j,k}(h) e^{-(1-\alpha)\frac{(\br(x)-C)^2}{4}}\\
&\leq C_{j,k,\varepsilon}(h)e^{-(1-\alpha-\varepsilon)\frac{\br(x)^2}{4}}\label{give-it-to-me}
\end{split}
\end{equation}
Given $\eta>0$, then \eqref{give-it-to-me} shows the expected estimate once $k$ (and hence $\alpha$) and $\varepsilon$ are fixed so that $1-\alpha-\varepsilon\geq 1-\eta$.
\end{proof}

\subsection{Asymptotically conical K\"ahler expanding gradient Ricci solitons}\label{subsec-kah-egs-tra}~~\\

We illustrate Theorem \ref{theo-6-1-Ber} in the context of K\"ahler expanding gradient Ricci solitons.

Recall that if $(M^{2n},J,\omega)$ is K\"ahler with K\"ahler form $\omega$ and complex structure $J$ and if $g$ denotes the Riemannian metric associated to $\omega$, then we say that $(M^{2n},J,\omega,\nabla^gf)$ is a gradient K\"ahler expanding Ricci soliton if $\nabla^gf$ is real holomorphic and 
\begin{equation*}
\rho_{\omega}-i\partial\overline{\partial}f=-\frac{\omega}{2},
\end{equation*}
where $\rho_{\omega}$ is the Ricci form of $\omega$, i.e. $\rho_{\omega}(u,v):=\Ric(g)(Ju,v)$ for $u,v\in TM$.

Now, recall that a K\"ahler cone is a Riemannian cone $C$ (we omit the reference to the link here) such that the cone metric $g_C$ is K\"ahler together with a choice of $g_C$-parallel complex structure $J_C$. We then have a K\"ahler form $\omega_C(u,v):=g_C(Ju,v)$ for $u,v\in T(C\setminus\{0\})$ and $\omega_C=\frac{i}{2}\partial\overline{\partial}r^2$ with respect to $J_C$. The vector field $r\partial_r$ is real holomorphic and the vector field $J_Cr\partial_r$, called the Reeb vector field, is real holomorphic and Killing.

Given a K\"ahler cone $(C,J_C,\omega_C=\frac{i}{2}\partial\overline{\partial}r^2)$, the level set $\{r=1\}$ is traditionally denoted by $S$ (standing for Sasakian) endowed with its induced Riemannian metric $g_S$. The restriction of the Reeb vector field to $S$ induces a non-zero vector field $\xi:=J_Cr\partial_r|_{r=1}$ on $S$. If $\eta$ denotes the $g_S$-dual one-form of $\xi$ then we get the $g_S$ orthogonal decomposition $TS=\mathcal{D}\oplus\RR\xi$, where $\mathcal{D}$ is the kernel of $\eta$. At the level of metrics, one gets $g_S=\eta\otimes\eta+g^T$ where $g^T:=g_S|_{\mathcal{D}}$. We call $g^T$ the transverse metric. With these definitions in hand,

\begin{coro}\label{coro-str-ric-inf-kahler}
Let $(M^n,J,\omega,\nabla^gf)$ and $(M^n,J,\omega',\nabla^{g'}f')$ be two K\"ahler expanding gradient Ricci solitons asymptotic to the same K\"ahler cone $(C,J_C,\omega_C,r\partial_r/2)$. Then the trace at infinity of the difference of the Ricci tensors $\hat{\calR}$ preserves both the radial vector field $\partial_r$ and $J\partial_r$. More precisely, 
\begin{equation*}
\tr_{\infty}\hat{\calR}=\frac{\tr_{g_C}(\tr_{\infty}\calR)}{2}\left(dr^2+r^2\eta\otimes\eta\right)+\Ric^{T}_{\infty},
\end{equation*}
where $\Ric^{T}_{\infty}$ is a symmetric tensor that preserves $\mathcal{D}$. In particular, $\Ric^{T}_{\infty}$ is trace free and divergence free, i.e. $\tr_{g_C}\Ric^{T}_{\infty}(g)=0$ and $\div_{g^T}\Ric^{T}_{\infty}=0$ in the weak sense.
\end{coro}
Notice that under the assumptions of Corollary \ref{coro-str-ric-inf-kahler}, it has been proved in \cite{Con-Der-Sun} that $(M^n,J,\omega,\nabla^gf)$ is the unique (up to pullback by biholomorphisms) complete expanding gradient K\"ahler-Ricci soliton asymptotic to the K\"ahler cone $(C,J_C,\omega_C,r\partial_r/2)$. Moreover, $M$ is the smooth canonical model of $C$, and there exists a resolution map $\pi:M\to C$ such that $d\pi(\nabla^gf)=\frac{1}{2}r\partial_{r}$: see Theorem $A$ and Corollary $B$ in \cite{Con-Der-Sun} for more precise statements.

Moreover, Corollary \ref{coro-str-ric-inf-kahler} is consistent with the examples of complete expanding gradient K\"ahler-Ricci solitons discovered by Feldman, Ilmanen and Knopf \cite{Fel-Ilm-Kno} on the total space of the tautological line bundles $L^{-k}$, $k>n$, over $\mathbb{CP}^{n-1}$. Indeed, these solutions on $L^{-k}$ are $U(n)$-invariant and asymptotic to the cone $C(\mathbb{S}^{2n-1}/\mathbb{Z}_k)$ endowed with the Euclidean metric $\frac{1}{2}i\partial\overline{\partial}\, |\cdot|^2$, where $\mathbb{Z}_k$ acts on $\mathbb{C}^n$ diagonally. As noticed in \cite[Example $3.3.3$]{Sie-phd}, the eigenvalues $(\lambda_i)_{1\leq i\leq n}$ of the Ricci tensor of these solitons are of the form: 
\begin{equation*}
\lambda_i=-(k-n)^ne^{(k-n)}\phi^{-n}e^{-\phi},\quad  i=1,...,n,\quad  \lambda_n=(k-n)^ne^{(k-n)}\phi^{-n}e^{-\phi}(\phi+(n-1)),
\end{equation*}
 where $\phi$ is a smooth positive function that behaves like $|z|^2$ at infinity. Notice that this is consistent with the (integral) Ricci curvature decay obtained in Theorem \ref{theo-6-1-Ber}. Moreover, the scalar curvature of these metrics is equal to $(k-n)^ne^{(k-n)}\phi^{1-n}e^{-\phi}$ which is positive everywhere on $L^{-k}$. Finally, the trace at infinity of the Ricci tensor of these metrics is diagonal and is universally proportional to $(k-n)^ne^{k-n}\left(dr\otimes dr+r^2\eta\otimes\eta\right)$ with the notations introduced above. In particular, the transverse component of the trace at infinity of the Ricci tensor vanishes identically.\newline

 \section{A relative entropy for Ricci gradient expanders}\label{sec-rel-ent}
 We begin this section with a technical lemma that estimates various integrals involving Gaussian weights.
\begin{lemma}\label{lemma-est-int-param}
Let $\alpha>0$, $\beta, \gamma \in \RR$. There exists $\lambda_0=\lambda_0(\alpha,\beta,\gamma)>0$ such that if $\tfrac{r^2}{4t}\geq \lambda_0$, then
\begin{equation*}
\int_{0}^{t}\tau^{\gamma}\left(\frac{r^2}{4\tau}\right)^{\beta}e^{-\alpha\frac{r^2}{4\tau}}\,d\tau\leq C(\alpha,\beta)t^{\gamma+1}\left(\frac{r^2}{4t}\right)^{\beta-1}e^{-\alpha\frac{r^2}{4t}}\, .
\end{equation*}

\end{lemma}

\begin{proof}
By a change of variable, one gets:
\begin{equation*}
\begin{split}
\int_{0}^{t}\tau^{\gamma}\left(\frac{r^2}{4\tau}\right)^{\beta}e^{-\alpha\frac{r^2}{4\tau}}\,d\tau&=\left(\frac{r^2}{4}\right)^{\gamma}\int_{0}^{t}\left(\frac{r^2}{4\tau}\right)^{\beta-\gamma}e^{-\alpha\frac{r^2}{4\tau}}\,d\tau\\
&=\left(\frac{r^2}{4}\right)^{\gamma+1}\int_{\frac{r^2}{4t}}^{+\infty}s^{\beta-\gamma-2}e^{-\alpha s}\,ds.
\end{split}
\end{equation*}
Now, by integrating by parts,
\begin{equation*}
\begin{split}
\int_{\frac{r^2}{4t}}^{+\infty}s^{\beta-\gamma-2}e^{-\alpha s}\,ds&=\frac{1}{\alpha}\left(\frac{r^2}{4t}\right)^{\beta-\gamma-2}e^{-\alpha\frac{r^2}{4t}}+\frac{\beta-\gamma-2}{\alpha}\int_{\frac{r^2}{4t}}^{+\infty}s^{\beta-\gamma-3}e^{-\alpha s}\,ds\\
&\leq \frac{1}{\alpha}\left(\frac{r^2}{4t}\right)^{\beta-\gamma-2}e^{-\alpha\frac{r^2}{4t}}+\frac{4t|\beta-\gamma-2|}{\alpha r^2}\int_{\frac{r^2}{4t}}^{+\infty}s^{\beta-\gamma-2}e^{-\alpha s}\,ds.
\end{split}
\end{equation*}
In particular, if we choose $\lambda_0=\lambda_0(\alpha,\beta,\gamma)>0$ such that $(2|\beta-\gamma-2|)\cdot 4t\leq\alpha r^2$, then one absorbs the last integral on the righthand side of the previous estimate to get the expected result.
\end{proof}

Let $(M_i^n,g_i,\nabla^{g_i}f_i)_{i=1,2}$ be two normalized asymptotically conical expanding gradient Ricci solitons coming out of the same cone $(C(S),dr^2+r^2g_S,r\partial_r/2)$. 

We define the following quantity for $t>0$, providing the limit exists:
\begin{equation}
\W(g_2(t),g_1(t)):=\lim_{R\rightarrow+\infty}\left(\int_{f_2(t)\leq R}\frac{e^{f_2(t)}}{(4\pi t)^{\frac{n}{2}}}\,d\mu_{g_2(t)}-\int_{f_1(t)\leq R}\frac{e^{f_1(t)}}{(4\pi t)^{\frac{n}{2}}}\,d\mu_{g_1(t)}\right)\, .\label{rel-ent-def}
\end{equation}
The following theorem proves this quantity is well-defined.
\begin{theo}\label{theo-rel-ent-def}
Let $(M_i^n,g_i,\nabla^{g_i}f_i)_{i=1,2}$ be two asymptotically conical gradient Ricci expanders coming out of the same cone $(C(S),dr^2+r^2g_S,r\partial_r/2)$ normalized as in Section \ref{sol-id-sec}. Then the limit in (\ref{rel-ent-def}) exists and is finite for all time $t>0$ and is constant in time:
\begin{equation*}
-\infty<\W(g_2(t),g_1(t))=\W(g_2(s),g_1(s))<+\infty,\quad 0<s<t.
\end{equation*}

\end{theo}

\begin{proof}
Let us prove that the relative entropy between $(g_2(t))_{t>0}$ and $(g_1(t))_{t>0}$ is independent of time provided it is well-defined.

By using the change of variable $y=\varphi_t^i(x)$, $(\varphi_t^i)_{t>0}$ being the flow generated by $-\nabla^{g_i}f_i/t$, $i=1,2$, in the following integral shows immediately that
 \begin{equation*}
 \int_{f_i(t)\leq R}\frac{e^{f_i(t)}}{(4\pi t)^{\frac{n}{2}}}\,d\mu_{g_i(t)}=\int_{f_i\leq R}\frac{e^{f_i}}{(4\pi)^{\frac{n}{2}}}\,d\mu_{g_i},\label{change-var-trivial-exp}
 \end{equation*}
 which implies that $\W(g_2(t),g_1(t))$ does not depend on time provided it is well-defined.  
 
 We decide to give another proof which is more in the spirit of the general case.
 
We start by computing the evolution equation of the weights $e^{f_i(t)}/(4\pi t)^{n/2}$ where $f_i(t):=(\varphi^i_t)^*f_i$. By using the soliton equations, see \eqref{eq:1-0} -- \eqref{eq:3-0},
\begin{equation*}
\begin{split}
\left(\partial_t+\Delta_{g_i(t)}\right)\left(\frac{e^{f_i(t)}}{(4\pi t)^{\frac{n}{2}}}\right)&=\left(\partial_tf_i(t)-\frac{n}{2t}+\Delta_{g_i(t)}f_i(t)+|\nabla^{g_i(t)}f_i(t)|^2_{g_i(t)}\right)\left(\frac{e^{f_i(t)}}{(4\pi t)^{\frac{n}{2}}}\right)\\
&=\R_{g_i(t)}\frac{e^{f_i(t)}}{(4\pi t)^{\frac{n}{2}}}.
\end{split}
\end{equation*}

In particular, for a fixed positive radius $R$,
\begin{equation*}
\begin{split}
\partial_t\int_{f_i(t)\leq R}\frac{e^{f_i(t)}}{(4\pi t)^{\frac{n}{2}}}d\mu_{g_i(t)}&=\int_{f_i(t)\leq R}\partial_t\left(\frac{e^{f_i(t)}}{(4\pi t)^{\frac{n}{2}}}d\mu_{g_i(t)}\right)\\
&\quad -\int_{f_i(t)=R}\frac{\partial_tf_i(t)}{|\nabla^{g_i(t)}f_i(t)|_{g_i(t)}}\left(\frac{e^{f_i(t)}}{(4\pi t)^{\frac{n}{2}}}\right)d\sigma_{g_i(t)}\\
&=-\int_{f_i(t)\leq R}\Delta_{g_i(t)}\left(\frac{e^{f_i(t)}}{(4\pi t)^{\frac{n}{2}}}\right)d\mu_{g_i(t)}\\
&\quad +\int_{f_i(t)=R}|\nabla^{g_i(t)}f_i(t)|_{g_i(t)}\left(\frac{e^{f_i(t)}}{(4\pi t)^{\frac{n}{2}}}\right)d\sigma_{g_i(t)}\\
&=0,
\end{split}
\end{equation*}
by Stokes formula, if we choose $R$ sufficiently large so that $\{f_i(t)=R\}$ is a smooth hypersurface. 

Therefore, by integrating with respect to time first and by letting $R$ go to $+\infty$ gives the result: $\W(g_2(t),g_1(t))=\W(g_2,g_1)$, for all positive time $t$.\\

Let us prove that $\W(g_2(t),g_1(t))$ is well-defined for all $t>0$. Recall by \eqref{eq:1-0} -- \eqref{eq:3-0} that if $(\varphi^i_t)_{t>0}$ denotes the flow generated by $-\nabla^{g_i}f_i/t$ such that $\varphi^i_t|_{t=1}=\Id_M$ and if $f_i(t):=(\varphi^i_t)^*f_i$,
\begin{equation}
\partial_t(tf_i)=t\R_{g_i(t)},\quad t>0,\quad \text{or equivalently},\quad\partial_t(tf_i)=\Delta_{g_i(t)}(tf_i)-\frac{n}{2}\label{sol-id-egs-time-dep}.
\end{equation}
Let $R>0$ and $\delta\in(0,1]$. Let $\varphi_{R,\delta}:[0,+\infty)\rightarrow [0,1]$ be a smooth cut-off function such that $\varphi_{R,\delta}(r)=1$ on $[0,R]$, $\varphi_{R,\delta}=0$ on $[(1+\delta)R,+\infty)$ and such that $|\varphi_{R,\delta}'|\leq c/(\delta R)$ for some positive universal constant $c$. We claim that the following limit
\begin{equation}\label{def:rel_entropy}
\lim_{R\rightarrow+\infty}\lim_{\delta\rightarrow 0}\left(\int_{M}\varphi_{R,\delta}(f_2(t))\,\frac{e^{f_2(t)}}{(4\pi t)^{\frac{n}{2}}}d\mu_{g_2(t)}-\int_{M}\varphi_{R,\delta}(f_1(t))\,\frac{e^{f_1(t)}}{(4\pi t)^{\frac{n}{2}}}d\mu_{g_1(t)}\right),
\end{equation}
exists, is finite and is equal to $\W(g_2(t),g_1(t))$ for all $t>0$.
Observe that it suffices to prove this assertion at time $t=1$ by considering parabolic rescalings of the solutions: $g_i^{\alpha}(t):=\alpha^{-1}g_i(\alpha t)$, $i=1,2$ and $\alpha>0$.

Let us compare the weighted volume forms at infinity first by linearizing one with respect to the other:
\begin{equation}
\begin{split}\label{lin-exp}
e^{f_2-f_1}&=e^{t\left(f_2-f_1\right)(t)}\biggr\rvert_{t=1}\\
&=1+\int_0^1\partial_{\tau}e^{\tau(f_2-f_1)(\tau)}\,d\tau\\
&=1+\int_0^1\partial_{\tau}(\tau(f_2-f_1))e^{\tau(f_2-f_1)(\tau)}\,d\tau\\
&=1+\int_0^1\tau(\R_{g_2(\tau)}-\R_{g_1(\tau)})e^{\tau(f_2-f_1)(\tau)}\,d\tau\\
&=1+\int_0^1\tau(\R_{g_2(\tau)}-\R_{g_1(\tau)})\,d\tau+Q_0(f_2-f_1),
\end{split}
\end{equation}
where,
\begin{equation*}
Q_0(f_2-f_1):=\int_0^1\tau(\R_{g_2(\tau)}-\R_{g_1(\tau)})\left(e^{\tau(f_2-f_1)(\tau)}-1\right)\,d\tau.
\end{equation*}
Here we have used (\ref{sol-id-egs-time-dep}) in the fourth line. 

Thanks to Theorem \ref{theo-6-1-Ber} and Corollary \ref{interpol-not-opt-dec}, we see that for any $\eta>0$, we can estimate for any $t\in(0,1]$,
\begin{equation}\label{eq:interpolation.1}
t|\R_{g_2(t)}-\R_{g_1(t)}| \leq C_\eta e^{-(1-\eta)f_1(t)},\quad \text{on $\{f_1(t)\geq R\}$,}
\end{equation}
where $C$ is a time-independent positive constant. Observe that \eqref{eq:interpolation.1} is in particular true on $\{f_1\geq R\}\subset \{f_1(t)\geq R\}$.

We can thus estimate the error term $Q_0(f_2-f_1)$ outside $K\subset M$ as follows:
\begin{equation*}
\begin{split}
|Q_0(f_2-f_1)|&\leq C\int_0^1\tau|\R_{g_2(\tau)}-\R_{g_1(\tau)}|\left|\tau(f_2-f_1)(\tau)\right|\,d\tau\\
&\leq \int_0^1\tau|\R_{g_2(\tau)}-\R_{g_1(\tau)}|\int_0^{\tau}s\big|\!\R_{g_2(s)}-\R_{g_1(s)}\!\big|\,ds\,d\tau\\
&\leq C\int_0^1e^{- (1-\eta)f_1(\tau)}\int_0^{\tau}e^{-(1-\eta)f_1(s)}\,ds\,d\tau\\
&\leq C\int_0^1e^{-2(1-\eta)f_1(\tau)}\,d\tau\\
&\leq C e^{-2 (1-\eta)f_1}.
\end{split}
\end{equation*}
Here, we have used repeatedly Lemma \ref{lemma-est-int-param} in the fourth and the fifth lines and we have invoked \eqref{sol-id-egs-time-dep} in the second line.

Now, consider $g_{\sigma}:=(\sigma-1)\cdot g_2+(2-\sigma)\cdot g_1$ for $\sigma\in[1,2]$. It is straightforward to check that $g_{\sigma}$ is a metric on $M\setminus K$. Then,
\begin{equation*}
\begin{split}
\left(\frac{d\mu_{g_2}}{d\mu_{g_1}}\right)^2&=\left(\frac{d\mu_{g_{\sigma}}|_{\sigma=2}}{d\mu_{g_{\sigma}}|_{\sigma=1}}\right)^2=1+\tr_{g_1}(g_2-g_1)-\int_0^1\int_0^s|g_2-g_1|^2_{g_{\sigma}}\,d\sigma ds,
\end{split}
\end{equation*}
where we have used the fact that if $F(\sigma):=\left(\frac{d\mu_{g_{\sigma}}}{d\mu_{g_1}}\right)^2$ then 
\begin{equation*}
\frac{d^2}{d\sigma^2}F(\sigma)=\frac{d}{d\sigma}\tr_{g_{\sigma}}(g_2-g_1)=-|g_2-g_1|^2_{g_{\sigma}}.
\end{equation*}
In particular,
\begin{equation}
\begin{split}\label{lin-vol-forms}
\frac{d\mu_{g_2}}{d\mu_{g_1}}&=1+\frac{1}{2}\tr_{g_1}(g_2-g_1)+Q_1(g_2-g_1),
\end{split}
\end{equation}
where we can estimate due to Theorem \ref{theo-6-1-Ber} and Corollary \ref{interpol-not-opt-dec} that for $\eta>0$
\begin{equation*}
|Q_1(g_2-g_1)|\leq C_\eta e^{-2(1-\eta) f_1} .
\end{equation*}
Next, observe that,
\begin{equation}
\begin{split}\label{lin-tr-int}
\tr_{g_1}(g_2-g_1)&=\int_0^1\partial_{\tau}\left(\tau\cdot\tr_{g_1(\tau)}(g_2(\tau)-g_1(\tau))\right)\,d\tau\\
&=\int_0^1\tau\left(\tr_{g_1(\tau)}(\partial_{\tau}g_2-\partial_{\tau}g_1)\right)\,d\tau\\
&\quad+\int_0^1\left(\partial_{\tau}g_1\ast g_1(\tau)^{-1}\ast g_1(\tau)^{-1}\ast (g_2-g_1)(\tau)\right)\,d\tau\\
&\quad+\int_0^1\tr_{g_1(\tau)}(g_2(\tau)-g_1(\tau))\,d\tau.
\end{split}
\end{equation}
Note that by Theorem \ref{ac-str-exp}, we have that $\partial_{\tau}g_i(\tau)$ is uniformly bounded on $\{ f_i(\tau)\geq R\}$ for  all $\tau \in [0,1]$ and $i=1,2$. Alternatively  this follows from the quadratic curvature decay of $g_i(1)$.
The last two integrals on the righthand side of (\ref{lin-tr-int}) can then be estimated as follows:
\begin{equation}
\begin{split}\label{est-RHS-lin-tr-int}
\left|\int_0^1\left(\partial_{\tau}g_1\ast g_1(\tau)^{-1}\ast g_1(\tau)^{-1}\ast (g_2-g_1)(\tau)\right)\,d\tau\right|&\leq C \int_0^1 |g_2(\tau) -g_1(\tau)|_{g_1(\tau)}\, d\tau \\
\left|\int_0^1\tr_{g_1(\tau)}(g_2(\tau)-g_1(\tau))\,d\tau\right|&\leq C \int_0^1 |g_2(\tau) -g_1(\tau)|_{g_1(\tau)}\, d\tau \, .
\end{split}
\end{equation}
Now, by the definition of a solution to the Ricci flow, one gets,
\begin{equation}
\begin{split}\label{est-int-tr-egs}
\int_0^1\tau\tr_{g_1(\tau)}(\partial_{\tau}g_2(\tau)-\partial_{\tau}g_1(\tau))\,d\tau&=-2\int_0^1\tau\left(\R_{g_2(\tau)}-\R_{g_1(\tau)}\right)\,d\tau\\
&+\int_0^1g_2(\tau)^{-1}\ast(g_2(\tau)-g_1(\tau))\ast \Ric(g_2(\tau))\,d\tau.
\end{split}
\end{equation}
The second integral on the righthand side of (\ref{est-int-tr-egs}) can be estimated as in (\ref{est-RHS-lin-tr-int}):
\begin{equation}
\label{est-RHS-int-tr-egs}
\left|\int_0^1g_2(\tau)^{-1}\ast(g_2(\tau)-g_1(\tau))\ast \Ric(g_2(\tau))\,d\tau\right|\leq C \int_0^1 |g_2(\tau) -g_1(\tau)|_{g_1(\tau)}\, d\tau.
\end{equation}
To sum it up, (\ref{lin-exp}), (\ref{lin-vol-forms}) together with \eqref{lin-tr-int}, (\ref{est-RHS-lin-tr-int}) and (\ref{est-RHS-int-tr-egs}) lead to:
\begin{equation}\label{crucial-est-lin-weigted-vol-forms}
\left|e^{f_2-f_1}\frac{d\mu_{g_2}}{d\mu_{g_1}}-1\right|\leq C \int_0^1 |g_2(\tau) -g_1(\tau)|_{g_1(\tau)}\, d\tau  + C_\eta e^{-2(1-\eta) f_1}\,.
\end{equation}
As an intermediate step, notice that:
\begin{equation}
\begin{split}
\varphi_{R,\delta}(f_2)\,e^{f_2-f_1}\frac{d\mu_{g_2}}{d\mu_{g_1}}-\varphi_{R,\delta}(f_1)&= \varphi_{R,\delta}(f_2)\,\left(e^{f_2-f_1}\frac{d\mu_{g_2}}{d\mu_{g_1}}-1\right) \\
&\quad +\left(\varphi_{R,\delta}(f_2)-\varphi_{R,\delta}(f_1)\right) \\
&= \varphi_{R,\delta}(f_2)\,\left(e^{f_2-f_1}\frac{d\mu_{g_2}}{d\mu_{g_1}}-1\right) \\
&\quad +\left(\int_0^1\varphi'_{R,\delta}(f_{\sigma})\,d\sigma\right)(f_2-f_1) \label{last-est-wei-vol-form}
\end{split}
\end{equation}
where  $f_{\sigma}=(\sigma-1)f_2+(2-\sigma)f_1$ for $\sigma\in[1,2]$. Therefore, in view of (\ref{last-est-wei-vol-form}) together with \eqref{crucial-est-lin-weigted-vol-forms}, the following two claims establish that the limit in \eqref{def:rel_entropy} exists.
\begin{claim}\label{lim-tr-delic-van-infinity-claim} For $R > R_0$, it holds
\begin{equation*}\label{lim-tr-delic-van-infinity}
\limsup_{\delta\rightarrow0}\int_0^1\int_{E_{2\sqrt{R_0}}}\varphi_{R,\delta}(f_2)\,  |g_2(\tau) -g_1(\tau)|_{g_1(\tau)}\, \,e^{f_1}d\mu_{g_1}d\tau\leq C R_0^{-\frac{1}{4}}\, .
\end{equation*}
\end{claim}

\begin{proof}[Proof of Claim \ref{lim-tr-delic-van-infinity-claim}]
For $\tau>0$, define $h(\tau):=g_2(\tau)-g_1(\tau)$ and correspondingly $\hat{h}(\tau):=f_1^{\frac{n}{2}}(\tau)e^{f_1(\tau)}h(\tau).$ Then, by the Cauchy-Schwarz inequality, for $R>R_0$,
\begin{align*}
\int_0^1\!\!&\int_{E_{2\sqrt{R_0}}}\!\!\!\!\!\!\varphi_{R,\delta}(f_2)\,  |g_2(\tau) -g_1(\tau)|_{g_1(\tau)}\,e^{f_1}d\mu_{g_1}\,d\tau\\
&\leq\int_0^1\!\! \int_{E_{2\sqrt{R_0}}}|\hat{h}(\tau)|_{g_1(\tau)}e^{f_1-f_1(\tau)}f_1^{-\frac{n}{2}}(\tau)\,d\mu_{g_1}\, d\tau\\
&\leq\int_0^1\!\!\bigg(\int_{E_{2\sqrt{R_0}}}\!\!\!|\hat{h}(\tau)|_{g_1(\tau)}^2f_1^{-\frac{n+1}{2}}(\tau)\,d\mu_{g_1}\bigg)^{\frac{1}{2}}\bigg(\int_{E_{2\sqrt{R_0}}}\!\!\!e^{2(f_1-f_1(\tau))}f_1^{-\frac{n-1}{2}}(\tau)\,d\mu_{g_1}\bigg)^{\frac{1}{2}}\!\! d\tau.
\end{align*}
Now, since $g_1$ and $g_1(\tau)$ are uniformly-in-time equivalent on $E_{2\sqrt{R_0}}$ by Theorem \ref{ac-str-exp}, Theorem \ref{theo-6-1-Ber} and the change of variable theorem lead to:
\begin{align*}
\int_{E_{2\sqrt{R_0}}}|\hat{h}(\tau)|_{g_1(\tau)}^2f_1^{-\frac{n+1}{2}}(\tau)\,d\mu_{g_1}&\leq C\int_{E_{2\sqrt{R_0}}}|\hat{h}(\tau)|_{g_1(\tau)}^2f_1^{-\frac{n+1}{2}}(\tau)\,d\mu_{g_1(\tau)}\\
&=C\tau^{\frac{n}{2}}\int_{\varphi_{\tau}(E_{2\sqrt{R_0}})}|\hat{h}|_{g_1}^2f_1^{-\frac{n+1}{2}}\,d\mu_{g_1}\\
&\leq C\tau^{\frac{n}{2}}\int_{E_{\frac{2\sqrt{R_0}}{\sqrt{\tau}}}}|\hat{h}|_{g_1}^2f_1^{-\frac{n+1}{2}}\,d\mu_{g_1}\\
&\leq C\frac{\tau^{\frac{n+1}{2}}}{\sqrt{R_0}}.
\end{align*}
Finally, by invoking Fubini's theorem and Lemma \ref{lemma-est-int-param},
\begin{equation*}
\begin{split}
\int_0^1\int_{E_{2\sqrt{R_0}}}\!\!\!\!\!\!e^{2(f_1-f_1(\tau))}f_1^{-\frac{n-1}{2}}(\tau)\,d\mu_{g_1}\,d\tau&=\int_{E_{2\sqrt{R_0}}}\left(\int_0^1e^{-2f_1(\tau)}f_1^{-\frac{n-1}{2}}(\tau)\,d\tau\right)\,e^{2f_1}d\mu_{g_1}\\
&\leq C\int_{E_{2\sqrt{R_0}}}f_1^{-\frac{n+1}{2}}\,d\mu_{g_1}\leq C(R_0)<+\infty.
\end{split}
\end{equation*}
This ends the proof of Claim \ref{lim-tr-delic-van-infinity-claim}.
\end{proof}

The following claim takes care of the second term on the righthand side of \eqref{last-est-wei-vol-form}:

\begin{claim}\label{lim-first-var-van-infinity-claim}
\begin{equation*}
\lim_{R\rightarrow+\infty}\lim_{\delta\rightarrow0}\int_M\int_1^2\varphi'_{R,\delta}(f_{\sigma})\,(f_2-f_1)\, \,d\sigma e^{f_1}d\mu_{g_1},
\end{equation*}
exists and is finite.
\end{claim}
\begin{proof}[Proof of Claim \ref{lim-first-var-van-infinity-claim}]
For reasons that will be made be clear below, it is more convenient to express the difference $f_2-f_1$ in terms of the Morse flow $(\psi^{\sigma}_{\tau})_{\tau>0}$ associated to each function $\br_{\sigma}:=2\sqrt{f_{\sigma}}$, $\sigma\in[1,2]$ instead of the flow generated by $-\nabla^{g_i}f_i/{\tau}$, $i=1,2$ as we did previously. Indeed, consider for $\sigma\in[1,2]$,
\begin{equation}
\partial_{\tau}\psi^{\sigma}_{\tau}=-\tau^{-1}\frac{\nabla^{g_{\sigma}}f_{\sigma}}{|\nabla^{g_{\sigma}}\br_{\sigma}|^2_{g_{\sigma}}}\circ\psi^{\sigma}_{\tau}=:-\tau^{-1}X^{\sigma}\circ\psi^{\sigma}_{\tau},\quad \tau>0,\quad \psi_{\tau=1}^{\sigma}=\Id_M.\label{def-morse-flow-ad-hoc}
\end{equation}
The advantage of this flow lies in the following basic property on the action of this one-parameter family of diffeomorphisms on the level sets of $\br_{\sigma}$: by using the very definition of the flow given in \ref{def-morse-flow-ad-hoc}, one gets for $\tau\in(0,1]$ and $\rho>0$ large enough,
\begin{equation*}
\psi^{\sigma}_{\tau}\left(S^{\sigma}_{\rho}\right)=S^{\sigma}_{\frac{\rho}{\sqrt{\tau}}},
\end{equation*}
where $S^{\sigma}_{\rho}:=\{\br_{\sigma}=\rho\}$.

Moreover, observe that [\eqref{eq:3}, Lemma \ref{sol-id-qual}] and \eqref{eq:6} give:
\begin{equation*}
|X^{\sigma}-X|_{g_{\sigma}}\leq C\left(\br_{\sigma}|h|_{g_{\sigma}}+\br_{\sigma}^{-3}\right)\leq C\br_{\sigma}^{-3},\end{equation*}
where $X:=\nabla^{g_1}f_1=\nabla^{g_2}f_2$ and $C$ is a positive constant uniform in $\sigma\in[1,2]$ that may vary from line to line. Here we have used Theorem \ref{ac-str-exp} ensuring that $|h|_{g_{\sigma}}=O(\br_{\sigma}^{-4})$ in the second inequality.

By mimicking the Taylor expansions done in \eqref{lin-exp} and \eqref{lin-tr-int}--\eqref{est-RHS-int-tr-egs}, one gets on the one hand,
\begin{equation}
\begin{split}\label{est-pot-fct-diff-morse}
f_2-f_1&=\int_0^1\partial_{\tau}\left(\tau(\psi^{\sigma}_{\tau})^{\ast}(f_2-f_1)\right)\,d\tau\\
&=\int_0^1(\psi^{\sigma}_{\tau})^{\ast}\left((f_2-f_1)+X^{\sigma}\cdot(f_1-f_2)\right)\,d\tau\\
&=\int_0^1(\psi^{\sigma}_{\tau})^{\ast}\left((f_2-|\nabla^{g_2}f_2|^2_{g_2}-(f_1-|\nabla^{g_1}f_1|^2_{g_1})\right)\,d\tau\\
&\quad+\int_0^1(\psi^{\sigma}_{\tau})^{\ast}\left((X^{\sigma}-\nabla^{g_1}f_1)\cdot f_1-(X^{\sigma}-\nabla^{g_2}f_2)\cdot f_2\right)\,d\tau\\
&=\int_0^1(\psi^{\sigma}_{\tau})^{\ast}\left(\R_{g_2}-\R_{g_1}\right)\,d\tau+\int_0^1(\psi^{\sigma}_{\tau})^{\ast}\left((X^{\sigma}-X)\cdot(f_1-f_2)\right)\,d\tau,
\end{split}
\end{equation}
where we have used \eqref{eq:3-0} in the last line.

Now on the other hand, observe that:
\begin{equation*}
(X^{\sigma}-X)\cdot(f_1-f_2)=\frac{1}{|\nabla^{g_{\sigma}}\br_{\sigma}|^2_{g_{\sigma}}}\left(\nabla^{g_{\sigma}}f_{\sigma}-X\right)\cdot (f_2-f_1)+\left(\frac{1}{|\nabla^{g_{\sigma}}\br_{\sigma}|^2_{g_{\sigma}}}-1\right)(g_1-g_2)(X,X).
\end{equation*}

Thanks to [\eqref{eq:3}, Lemma \ref{sol-id-qual}] and \eqref{eq:6}, there exists a positive constant $C$ independent of $\sigma\in[1,2]$ that may vary from line to line such that:
\begin{equation}
\begin{split}\label{first-term-diff-pot-fct}
|(X^{\sigma}-X)\cdot(f_1-f_2)|&\leq C\br_{\sigma}|h|_{g_{\sigma}}|\nabla^{g_{\sigma}}(f_1-f_2)|_{g_{\sigma}}+C\br_{\sigma}^{-4}|h|_{g_{\sigma}}|X|^2_{g_{\sigma}}\\
&\leq C\br_{\sigma}^2|h|^2_{g_{\sigma}}+C\br_{\sigma}^{-2}|h|_{g_{\sigma}}\\
&\leq C\br_{\sigma}^{-2}|h|_{g_{\sigma}},
\end{split}
\end{equation}
where we have used [\eqref{eq:3}, Lemma \ref{sol-id-qual}] once more in the second line together with Theorem \ref{ac-str-exp} ensuring that $|h|_{g_{\sigma}}=O(\br_{\sigma}^{-4})$.

It remains to relate the difference of the scalar curvatures on the righthand side of \eqref{est-pot-fct-diff-morse} to the difference of metrics $h$ as we did in estimates \eqref{lin-tr-int}--\eqref{est-RHS-int-tr-egs}. To do so, define the one-parameter family of metrics $g_i^{\sigma}(\tau):=\tau(\psi^{\sigma}_{\tau})^{*}g_i$ for $i=1,2$, $\sigma\in[1,2]$ and $\tau>0$ which satisfies the Ricci flow equation approximately in the following way:
\begin{equation}
\begin{split}
\partial_{\tau}g_i^{\sigma}&=\frac{g_i^{\sigma}(\tau)}{\tau}-\tau^{-1}\calL_{(\psi^{\sigma}_{\tau})^{*}X^{\sigma}}(g_i^{\sigma}(\tau))\\
&=-2\Ric(g_i^{\sigma}(\tau))-\tau^{-1}\calL_{(\psi^{\sigma}_{\tau})^{*}(X^{\sigma}-X)}(g_i^{\sigma}(\tau)),\label{app-ricci-flow-morse}
\end{split}
\end{equation}
where we have used the soliton equation \eqref{eq:0-0} in the last line. In particular, by tracing \eqref{app-ricci-flow-morse} for $i=1,2$ gives:
\begin{equation}
\tau\tr_{g^{\sigma}_{i}(\tau)}(\partial_{\tau}g_i^{\sigma})=-2\tau\R_{g_i^{\sigma}(\tau)}-2\div_{g_i^{\sigma}(\tau)}((\psi^{\sigma}_{\tau})^{*}(X^{\sigma}-X)). \label{tr-app-ricci-flow-morse}
\end{equation}
Based on \eqref{tr-app-ricci-flow-morse}, the same reasoning that led to estimates \eqref{lin-tr-int}--\eqref{est-RHS-int-tr-egs} gives:
\begin{equation}
\begin{split}\label{bis-diff-scal-curv-morse}
\bigg|\int_0^1(\psi^{\sigma}_{\tau})^{\ast}&\left(\R_{g_2}-\R_{g_1}+\div_{g_2}(X^{\sigma}-X)-\div_{g_1}(X^{\sigma}-X)\right)\,d\tau\bigg|\leq \\
&C|\tr_{g_1}h|+C\int_0^1|g^{\sigma}_2(\tau) -g^{\sigma}_1(\tau)|_{g^{\sigma}_1(\tau)}\, d\tau.
\end{split}
\end{equation}
As for the last term on the lefthand side of \eqref{bis-diff-scal-curv-morse}, we linearize carefully the divergence of a vector field $V$ as follows:
\begin{equation}
\div_{g_2}V=\div_{g_1}V+\frac{1}{2}V\cdot \tr_{g_1}h+O(|V|_{g_1}|h|^2_{g_1}).\label{love-lin-div-morse}
\end{equation}
This essentially comes from \cite[Section $5$, Chapter $2$]{Cho-Lu-Ni}.

Applying \eqref{love-lin-div-morse} to $V:=X^{\sigma}-X$ together with [\eqref{eq:3}, Lemma \ref{sol-id-qual}], as well as Corollary \ref{interpol-not-opt-dec} to estimate $\nabla^{g_\sigma} \tr_{g_1}h$, leads to an improvement of \eqref{bis-diff-scal-curv-morse}:
\begin{equation}
\begin{split}\label{bis-diff-scal-curv-morse-bianchi}
\bigg|\int_0^1(\psi^{\sigma}_{\tau})^{\ast}&\left(\R_{g_2}-\R_{g_1}+\frac{1}{2}\left(\frac{1}{|\nabla^{g_{\sigma}}\br_{\sigma}|^2_{g_{\sigma}}}-1\right)\nabla^{g_{\sigma}}f_{\sigma}\cdot \tr_{g_1}h\right)\,d\tau\bigg|\leq\\
& C|\tr_{g_1}h|+C\int_0^1|g^{\sigma}_2(\tau) -g^{\sigma}_1(\tau)|_{g^{\sigma}_1(\tau)}\, d\tau.
\end{split}
\end{equation}
Combining estimates \eqref{est-pot-fct-diff-morse}, \eqref{first-term-diff-pot-fct} and \eqref{bis-diff-scal-curv-morse-bianchi} finally leads to:
\begin{equation}
\begin{split}\label{est-pot-fct-diff-morse-2}
\bigg|\int_M\int_1^2&\varphi'_{R,\delta}(f_{\sigma})\,(f_2-f_1)\, \,d\sigma e^{f_1}d\mu_{g_1}\\
&-\frac{1}{2}\int_M\int_1^2\varphi'_{R,\delta}(f_{\sigma})\int_0^1(\psi^{\sigma}_{\tau})^{\ast}\left[\left(\frac{1}{|\nabla^{g_{\sigma}}\br_{\sigma}|^2_{g_{\sigma}}}-1\right)\nabla^{g_{\sigma}}f_{\sigma}\cdot \tr_{g_1}h\right]\,d\tau d\sigma e^{f_1}d\mu_{g_1}\bigg|\\
&\leq C\int_M\int_1^2\int_0^1|\varphi'_{R,\delta}(f_{\sigma})|\left(|h|_{g_{\sigma}}+|h^{\sigma}(\tau)|_{g^{\sigma}_1(\tau)}\right)\,d\tau d\sigma e^{f_1}d\mu_{g_1},
\end{split}
\end{equation}
for some positive constant $C$ uniform in $\sigma\in[1,2]$ where $h^{\sigma}(\tau):=g_2^{\sigma}(\tau)-g_1^{\sigma}(\tau).$

In order to prove Claim \ref{lim-first-var-van-infinity-claim}, we  first prove that the righthand side of \eqref{est-pot-fct-diff-morse-2} converges to $0$ as $\delta$ tends to $0$ and $R$ goes to $+\infty$.

By Cauchy-Schwarz inequality and Theorem \ref{theo-6-1-Ber} applied to each metric $g_{\sigma}$ and potential function $f_{\sigma}$ (see Remark \ref{rem-6-1-Ber}), observe that:
 \begin{equation}
 \begin{split}\label{inequ-proof-lim-first-var-van-infinity-bis}
\int_M|\varphi'_{R,\delta}(f_{\sigma})|&\,|h|_{g_{\sigma}} \, e^{f_{\sigma}}d\mu_{g_{\sigma}}\leq \frac{C}{R\delta}\int_{R\leq f_{\sigma}\leq R(1+\delta)}|\hat{h}|_{g_{\sigma}}f_{\sigma}^{-\frac{n}{2}}\,d\mu_{g_{\sigma}}\\
&\leq \frac{C}{R\delta} \int_{2\sqrt{R}}^{2\sqrt{R(1+\delta)}} \rho^{-n} \int_{S^{\sigma}_{\rho}} |\hat{h}|_{g_{\sigma}}\, d\rho\\
&\leq \frac{C}{R\delta} \int_{2\sqrt{R}}^{2\sqrt{R(1+\delta)}} \rho^{-\frac{n+1}{2}} \left(\int_{S^{\sigma}_{\rho}} |\hat{h}|_{g_{\sigma}}^2 \right)^\frac{1}{2} d\rho\\
&\leq \frac{C}{R\delta}  \int_{2\sqrt{R}}^{2\sqrt{R(1+\delta)}} \rho^{-1} \, d\rho \leq \frac{C}{R} \frac{\log{(1+\delta)}}{\delta}
\end{split}
\end{equation}
where $C$ is a uniform positive constant which does not depend on $\sigma\in[1,2]$, $\delta\in(0,1]$ and $R\geq R_0>0$ and which may vary from line to line. Here we have used the co-area formula in the second line. 

Let us prove that:
\begin{equation}
\limsup_{R\rightarrow+\infty}\limsup_{\delta\rightarrow0}\int_1^2\int_M\int_0^1|\varphi'_{R,\delta}(f_{\sigma})|\,|h^{\sigma}(\tau)|_{g_1^{\sigma}(\tau)}\, d\tau\,e^{f_{\sigma}}d\mu_{g_{\sigma}}d\sigma =0.\label{hello-ugly-est}
\end{equation}
 
 We proceed analogously as before by defining $f_{\sigma}(\tau):=(\psi_{\tau}^{\sigma})^{*}f_{\sigma}$ and $\hat{h}^{\sigma}(\tau):=f^{\frac{n}{2}}_{\sigma}(\tau)e^{f_{\sigma}(\tau)}h^{\sigma}(\tau)$ and by observing that:
 \begin{equation*}
 \begin{split}
\int_{M}|\varphi'_{R,\delta}(f_{\sigma})|\int_0^1|h^{\sigma}(\tau)|_{g_1(\tau)}&\, d\tau\,e^{f_{\sigma}}d\mu_{g_{\sigma}}\leq\\
 &\leq \frac{C}{R\delta}\int_0^1\int_{R\leq f_{\sigma}\leq R(1+\delta)}|\hat{h}^{\sigma}(\tau)|_{g_1(\tau)}e^{f_{\sigma}-f_{\sigma}(\tau)}f_{\sigma}^{-\frac{n}{2}}(\tau)\,d\mu_{g_{\sigma}}d\tau\\
 &\leq\frac{C}{R\delta}\int_0^1\int_{R\leq f_{\sigma}\leq R(1+\delta)}|\hat{h}^{\sigma}(\tau)|_{g_1(\tau)}f_{\sigma}^{-\frac{n}{2}}(\tau)\,d\mu_{g_{\sigma}(\tau)}d\tau\\
 &\leq \frac{C}{R\delta}\int_0^1\int_{\frac{R}{\tau}\leq f_{\sigma}\leq \frac{R(1+\delta)}{\tau}}|\hat{h}|_{g_1}f_{\sigma}^{-\frac{n}{2}}\,d\mu_{g_{\sigma}}\tau^{\frac{n}{2}}d\tau,
 \end{split}
 \end{equation*}
 where we have used the fact that $f_{\sigma}\leq f_{\sigma}(\tau)$ for $\tau\in(0,1]$ and a change of variable in the last line. We notice that the use of the Morse flow $(\psi^{\sigma}_{\tau})_{\tau>0}$ is crucially entering the proof of this claim here.

Now, reasoning as in \eqref{inequ-proof-lim-first-var-van-infinity-bis}, 
 \begin{equation*}
 \begin{split}
 \int_{M}|\varphi'_{R,\delta}(f_{\sigma})|\int_0^1|h^{\sigma}(\tau)|_{g_1(\tau)}\, d\tau\,e^{f_{\sigma}}d\mu_{g_{\sigma}}&\leq \frac{C}{R\delta}  \int_0^1\int_{2\sqrt{R\tau^{-1}}}^{2\sqrt{R\tau^{-1}(1+\delta)}} \rho^{-1} \, d\rho d\tau \\
 &\leq \frac{C}{R} \frac{\log{(1+\delta)}}{\delta}.
\end{split}
 \end{equation*}
 This ends the proof of \eqref{hello-ugly-est}.
 
 We are left with proving that:
 \begin{equation}\label{last-least-int}
 \lim_{R\rightarrow+\infty}\lim_{\delta\rightarrow0}\int_M\int_1^2\varphi'_{R,\delta}(f_{\sigma})\int_0^1(\psi^{\sigma}_{\tau})^{\ast}\left[\left(\frac{1}{|\nabla^{g_{\sigma}}\br_{\sigma}|^2_{g_{\sigma}}}-1\right)\nabla^{g_{\sigma}}f_{\sigma}\cdot \tr_{g_1}h\right]\,d\tau d\sigma e^{f_1}d\mu_{g_1},
 \end{equation}
exists and is finite.

Since $|g_{\sigma}-g_1|_{g_1}\leq |h|_{g_1}$ and $\tr_{g_{\sigma}}h=\tr_{g_1}h+g_{\sigma}^{-1}\ast g_1^{-1}\ast h\ast h$ poinwise, as well as Corollary \ref{interpol-not-opt-dec} to estimate $\nabla^{g_\sigma} \tr_{g_1}h$ and arguing as in the previous steps,
it is equivalent to prove that the previous limit exists and is finite with respect to the measure $d\mu_{g_{\sigma}}$  instead and the term $\tr_{g_{\sigma}}h$ in place of $\tr_{g_{1}}h$. Due to Fubini's theorem together with the change of variable theorem, observe that:
\begin{equation}
\begin{split}\label{prep-IBP-final}
\int_0^1I_{R,\delta}&(\tau)\,d\tau:=\\
&\int_0^1\int_M\varphi'_{R,\delta}(f_{\sigma})(\psi^{\sigma}_{\tau})^{\ast}\left[\left(\frac{1}{|\nabla^{g_{\sigma}}\br_{\sigma}|^2_{g_{\sigma}}}-1\right)\nabla^{g_{\sigma}}f_{\sigma}\cdot \tr_{g_{\sigma}}h\right]\, e^{f_1}d\mu_{g_{\sigma}}d\tau=\\
&\int_0^1\int_M\varphi'_{R,\delta}(\tau f_{\sigma})\left(\frac{1}{|\nabla^{g_{\sigma}}\br_{\sigma}|^2_{g_{\sigma}}}-1\right)\nabla^{g_{\sigma}}f_{\sigma}\cdot \tr_{g_{\sigma}}h\, e^{(\psi^{\sigma}_{\tau^{-1}})^{\ast}f_1}d\mu_{(\psi^{\sigma}_{\tau^{-1}})^{*}g_{\sigma}}d\tau=\\
&\int_0^1\int_M\left(\frac{1}{|\nabla^{g_{\sigma}}\br_{\sigma}|^2_{g_{\sigma}}}-1\right)\nabla^{g_{\sigma}}\left(\varphi_{R,\delta}(\tau f_{\sigma})\right)\cdot \tr_{g_{\sigma}}h\, e^{(\psi^{\sigma}_{\tau^{-1}})^{\ast}f_1}d\mu_{(\psi^{\sigma}_{\tau^{-1}})^{*}g_{\sigma}}\tau^{-1}d\tau.
\end{split}
\end{equation}
Now, introduce the following auxiliary weight for $\delta>0$, $R$ sufficiently large and $\tau\in(0,1]$:
\begin{equation*}
w_{\sigma}(\tau):=\left(\frac{1}{|\nabla^{g_{\sigma}}\br_{\sigma}|^2_{g_{\sigma}}}-1\right)e^{(\psi^{\sigma}_{\tau^{-1}})^{\ast}f_1}\frac{d\mu_{(\psi^{\sigma}_{\tau^{-1}})^{*}g_{\sigma}}}{d\mu_{g_{\sigma}}},
\end{equation*}
whose definition is motivated by the computation \eqref{prep-IBP-final}. Notice that \eqref{eq:6}, [\eqref{eq:-1}, Lemma \ref{sol-id-qual}] and the definition of the Morse flow imply that on $\{f_{\sigma}\geq \tau^{-1} R_0\}$:
\begin{equation}
|w_{\sigma}(\tau)|\leq C\tau^{\frac{n}{2}}f_{\sigma}^{-2}e^{\tau f_{\sigma}},\label{est-wei-w}
\end{equation}
where $C$ is a positive constant uniform in $\tau\in(0,1]$ and $\sigma\in[1,2]$.

Indeed, an integration by parts in space variable shows that for each $\tau\in(0,1]$:
\begin{equation*}
\begin{split}
I_{R,\delta}(\tau)&=-\int_M\varphi_{R,\delta}(\tau f_{\sigma})\div_{g_{\sigma}}\left(w_{\sigma}(\tau)\nabla^{g_{\sigma}}\tr_{g_{\sigma}}h\right)\,d\mu_{g_{\sigma}}.
\end{split}
\end{equation*}
By using the very definition of $\hat{h}$, and by invoking the same reasoning that led to the proof of \eqref{hello-ugly-est}, showing \eqref{last-least-int} amounts to showing that the limit of
\begin{equation*}
\begin{split}
\int_0^1J_{R,\delta}(\tau)&:=-\int_0^1\int_M\varphi_{R,\delta}(\tau f_{\sigma})\div_{g_{\sigma}}\left(w_{\sigma}(\tau)f_{\sigma}^{-\frac{n}{2}}e^{-f_{\sigma}}\nabla^{g_{\sigma}}\tr_{g_{\sigma}}\hat{h}\right)\,d\mu_{g_{\sigma}}\tau^{-1}d\tau,
\end{split}
\end{equation*}
as $\delta$ tends to $0$ and as $R$ goes to $+\infty$ exists and is finite.

To do so, we invoke Corollary \ref{coro-syst-cons} to get:
\begin{equation}
\begin{split}\label{when-does-it-stop}
\div_{g_{\sigma}}\Big(w_{\sigma}(\tau)f_{\sigma}^{-\frac{n}{2}}e^{-f_{\sigma}}&\nabla^{g_{\sigma}}\tr_{g_{\sigma}}\hat{h}\Big) =\\
&=
w_{\sigma}(\tau)f_{\sigma}^{-\frac{n}{2}}e^{-f_{\sigma}}\calL_{-n}\tr_{g_{\sigma}}\hat{h}+f_{\sigma}^{-\frac{n}{2}}e^{-f_{\sigma}}  \nabla^{g_{\sigma}}w_{\sigma}(\tau)\cdot \tr_{g_{\sigma}}\hat{h}\\
&=2^{n+1}w_{\sigma}(\tau)\div_{g_{\sigma}}\calB+ \frac{n}{4}\, w_{\sigma}(\tau)f_{\sigma}^{-\frac{n+2}{2}}e^{-f_{\sigma}}\nabla^{g_{\sigma}}_{\nabla^{g_{\sigma}}f_{\sigma}}\tr_{g_{\sigma}}\hat{h}\\
&\quad+w_{\sigma}(\tau)f_{\sigma}^{-\frac{n}{2}}e^{-f_{\sigma}}\tr_{g_{\sigma}}R[\hat{h}]+f_{\sigma}^{-\frac{n}{2}}e^{-f_{\sigma}} \nabla^{g_{\sigma}}w_{\sigma}(\tau)\cdot \tr_{g_{\sigma}}\hat{h}.
\end{split}
\end{equation}
Thanks to Theorem \ref{theo-6-1-Ber} and estimate \eqref{est-wei-w},
\begin{equation*}
\begin{split}
&\int_0^1\int_M\varphi_{R,\delta}(\tau f_{\sigma})|w_{\sigma}(\tau)|\left[ f_{\sigma}^{-1}\left|\nabla^{g_{\sigma}}_{\nabla^{g_{\sigma}}f_{\sigma}}\tr_{g_{\sigma}}\hat{h}\right|+|\tr_{g_{\sigma}}R[\hat{h}]|\right]f_{\sigma}^{-\frac{n}{2}}e^{-f_{\sigma}}\,d\mu_{g_{\sigma}}\tau^{-1}d\tau\\
&\leq C\int_0^1\int_M\varphi_{R,\delta}(\tau f_{\sigma})f_{\sigma}^{-2-\frac{n}{2}}e^{(\tau-1)f_{\sigma}}\left[ f_{\sigma}^{-1}\left|\nabla^{g_{\sigma}}_{\nabla^{g_{\sigma}}f_{\sigma}}\hat{h}\right|+f_{\sigma}^{-1}\left(|\hat{h}|_{g_{\sigma}}+|\nabla^{g_{\sigma}}\hat{h}|_{g_{\sigma}}\right)\right]\,d\mu_{g_{\sigma}}\tau^{\frac{n}{2}-1}d\tau\\
&\leq C\int_0^1\int_Mf_{\sigma}^{-2-\frac{n}{2}}\left[f_{\sigma}^{-1}\left|\nabla^{g_{\sigma}}_{\nabla^{g_{\sigma}}f_{\sigma}}\hat{h}\right|+f_{\sigma}^{-1}\left(|\hat{h}|_{g_{\sigma}}+|\nabla^{g_{\sigma}}\hat{h}|_{g_{\sigma}}\right)\right]\,d\mu_{g_{\sigma}}\tau^{\frac{n}{2}-1}d\tau<+\infty.
\end{split}
\end{equation*}

As for the term involving the Bianchi one-form $\calB$, we integrate by parts to get:
\begin{equation}
\begin{split}\label{final-bianchi}
\int_M\varphi_{R,\delta}(\tau f_{\sigma})w_{\sigma}(\tau)\div_{g_{\sigma}}\calB\,d\mu_{g_{\sigma}}&=-\tau\int_M\varphi'_{R,\delta}(\tau f_{\sigma})w_{\sigma}(\tau)g_{\sigma}(\nabla^{g_{\sigma}}f_{\sigma},\calB)\,d\mu_{g_{\sigma}}\\
&\quad-\int_M\varphi_{R,\delta}(\tau f_{\sigma})g_{\sigma}(\nabla^{g_{\sigma}}w_{\sigma}(\tau),\calB)\,d\mu_{g_{\sigma}}.
\end{split}
\end{equation}
Observe that the first integral on the righthand side of \eqref{final-bianchi} can be estimated as follows for $\tau\in(0,1]$:
\begin{equation*}
\begin{split}
\tau\left|\int_M\varphi'_{R,\delta}(\tau f_{\sigma})w_{\sigma}(\tau)g_{\sigma}(\nabla^{g_{\sigma}}f_{\sigma},\calB)\,d\mu_{g_{\sigma}}\right|&\leq \frac{C}{R\delta}\tau^{\frac{n}{2}+1}\int_{\frac{R}{\tau}\leq f_{\sigma}\leq \frac{R(1+\delta)}{\tau}}f_{\sigma}^{-\frac{3}{2}}e^{\tau f_{\sigma}}|\calB|_{g_{\sigma}}\,d\mu_{g_{\sigma}}\\
&\leq \frac{C}{R\delta}\tau^{\frac{n}{2}+1}\int_{\frac{R}{\tau}\leq f_{\sigma}\leq \frac{R(1+\delta)}{\tau}} f_{\sigma}^{-\frac{3}{2}} e^{ f_{\sigma}}|\calB|_{g_{\sigma}}\,d\mu_{g_{\sigma}}.
\end{split}
\end{equation*}
By reasoning as in \eqref{inequ-proof-lim-first-var-van-infinity-bis} based on the decay of the frequency function associated to $\calB$ established in Proposition \ref{prop-4-4-Ber}, one gets:
\begin{equation*}
\lim_{R\rightarrow +\infty}\limsup_{\delta\rightarrow 0}\int_0^1\left|\int_M\varphi'_{R,\delta}(\tau f_{\sigma})w_{\sigma}(\tau)g_{\sigma}(\nabla^{g_{\sigma}}f_{\sigma},\calB)\,d\mu_{g_{\sigma}}\right|\,d\tau=0.
\end{equation*}

It remains to estimate the gradient of $w_{\sigma}(\tau)$ in order to handle the remaining terms on the righthand side of \eqref{when-does-it-stop} and \eqref{final-bianchi}. By using that the gradient of $f_{\sigma}$ grows linearly and the uniform equivalence of the metrics $g_{\sigma}(\tau)$ for $\tau\in(0,1]$ together with their derivatives,
\begin{equation*}
|\nabla^{g_{\sigma}}w_{\sigma}(\tau)|_{g_{\sigma}}\leq C\tau^{\frac{n}{2}}e^{\tau f_{\sigma}}f_{\sigma}^{-\frac{3}{2}},
\end{equation*}
we see by again using Theorem \ref{theo-6-1-Ber} and the decay of the frequency function associated to $\calB$ established in Proposition \ref{prop-4-4-Ber}, that the remaining terms can be similarly estimated as before to show that they are Cauchy in $R$.
 \end{proof}
\end{proof}

\bibliographystyle{alpha}
\bibliography{bib-fre-fct-egs}

\end{document}